\documentclass[preprint]{elsarticle} 

\makeatletter
	\def\ps@pprintTitle{%
 	\let\@oddhead\@empty
	\let\@evenhead\@empty
	\def\@oddfoot{\centerline{\thepage}}%
	\let\@evenfoot\@oddfoot}
\makeatother

\usepackage{etoolbox}
\patchcmd{\MaketitleBox}{\footnotesize\itshape\elsaddress\par\vskip36pt}{\footnotesize\itshape\elsaddress\par\parbox[b][36pt]{\linewidth}{\vfill\hfill\textnormal{\today}\hfill\null\vfill}}{}{}%
\patchcmd{\pprintMaketitle}{\footnotesize\itshape\elsaddress\par\vskip36pt}{\footnotesize\itshape\elsaddress\par\parbox[b][36pt]{\linewidth}{\vfill\hfill\textnormal{\today}\hfill\null\vfill}}{}{}%

\usepackage[colorlinks=true,%
            breaklinks=true,%
            linkcolor=blue,
            urlcolor=blue,%
            citecolor=blue,%
            pdftitle={Title of paper}, 
            pdfkeywords={Keywords},	
            pdfauthor={Authors},		
            bookmarksopen=false,
            pdfpagemode=UseNone]{hyperref}

\usepackage[margin = 1.2in]{geometry}
\usepackage[T1]{fontenc}
\usepackage[english]{babel}
\usepackage[latin1]{inputenc} 
\usepackage{mathtools}
\usepackage{siunitx}
\usepackage{amsmath}
\usepackage{amsfonts}
\usepackage{amsthm}
\usepackage{amssymb}
\usepackage{array}
\usepackage{algorithm} 
\usepackage{algorithmicx}
\usepackage{algpseudocode}
\usepackage{subcaption}
\usepackage{siunitx}
\usepackage{amsmath}
\usepackage{mathrsfs}
\usepackage{xpatch}

\usepackage{color}
\usepackage{url}
\usepackage{cleveref}
\usepackage{graphicx}
\usepackage{breakcites}
\usepackage{soul} 
\usepackage{enumitem}
\usepackage[title,titletoc,toc]{appendix}

\newtheoremstyle{mytheoremstyle}{5pt}{5pt}{\itshape}{}{\bfseries}{.}{.5em}{} 
\theoremstyle{mytheoremstyle}

\newtheoremstyle{myremarkstyle}{3pt}{3pt}{\itshape}{}{\bfseries}{.}{.5em}{} 
\theoremstyle{myremarkstyle}

\usepackage{pifont}
\newcommand{\cmark}{\ding{51}}%
\newcommand{\xmark}{\ding{55}}%

\newtheorem{theorem}{Theorem}
\newtheorem{example}{Example}

\newtheorem{remark}{Remark}

\newtheorem{proposition}{Proposition}

\newtheorem{definition}{Definition}

\renewenvironment{proof}%
{\noindent {\textbf{Proof}:} }%
{\hfill $\Box$ \\ }

\newcommand{\bit}{\begin{itemize}}
\newcommand{\eit}{\end{itemize}}
\newcommand{\ben}{\begin{enumerate}}
\newcommand{\een}{\end{enumerate}}

\newcommand{\Hquad}{\hspace{0.5em}} 
 
\begin{document}
\begin{frontmatter}
    \title{Anti-symmetric and Positivity Preserving Formulation of a Spectral Method for Vlasov-Poisson Equations}
    \author[ucsd,lanl]{Opal Issan\corref{cor1}}\ead{oissan@ucsd.edu}
    \author[lanl]{Oleksandr Koshkarov}
    \author[ga]{Federico D. Halpern}
    \author[ucsd]{Boris Kramer}
    \author[lanl]{Gian Luca Delzanno}
    
    \cortext[cor1]{Corresponding author}
    \address[ucsd]{Department of Mechanical and Aerospace Engineering, University of California San Diego, La Jolla, CA, USA}
    \address[lanl]{T-5 Applied Mathematics and Plasma Physics Group, Los Alamos National Laboratory, Los Alamos, NM, USA}
    \address[ga]{General Atomics, P.O. Box 85608, San Diego, CA, USA}
    
    \begin{abstract}
     We analyze the anti-symmetric properties of a spectral discretization for the one-dimensional Vlasov-Poisson equations. The discretization is based on a spectral expansion in velocity with the symmetrically weighted Hermite basis functions, central finite differencing in space, and an implicit Runge Kutta integrator in time. The proposed discretization preserves the anti-symmetric structure of the advection operator in the Vlasov equation, resulting in a stable numerical method. We apply such discretization to two formulations: the canonical Vlasov-Poisson equations and their continuously transformed square-root representation. The latter preserves the positivity of the particle distribution function. We derive analytically the conservation properties of both formulations, including particle number, momentum, and energy, which are verified numerically on the following benchmark problems: manufactured solution, linear and nonlinear Landau damping, two-stream instability, bump-on-tail instability, and ion-acoustic wave. 
    \end{abstract}	
\end{frontmatter}
\section{Introduction} \label{sec:intro}
Many plasma physics phenomena require a kinetic description, rather than a fluid description, to accurately represent the system's dynamics. Important examples include Earth's magnetosphere~\cite{palmroth_2023_magnetotail, jordanova_2018_shields}, solar corona~\cite{marsch_2006_kinetic, cranmer_2003_corona}, galactic dynamics~\cite{binney_2011_galactic}, and laboratory fusion devices~\cite{chen_1984_fusion, freidberg_2008_plasma}. The kinetic Vlasov equation describes the evolution of the particle (e.g., ions, electrons, etc) distribution function in six-dimensional phase space (three spatial and three velocity coordinates) coupled to self-consistent electromagnetic fields.

Numerically solving the Vlasov equation for collisionless plasmas is challenging since it is high dimensional, nonlinearly coupled to the Maxwell or Poisson equations, has a rich geometric structure, and has large spatial and temporal scale separation. 
The scale separation occurs in the microscopic setting since ions are typically at least three orders of magnitude heavier than electrons, and in the macroscopic setting since the system's scales are much larger than the local plasma scales. 
For instance, in the solar corona the Debye length, a kinetic scale describing the distance at which the plasma shields external charges, is of order $10^{-4}-10^{1}\unit{\m}$, and the ion Larmor radius, a kinetic scale describing the radius at which the ions gyrate around magnetic field lines, is of order $10^{1}-10^7\unit{\m}$, whereas the solar corona's system length scale is approximately $10^{11}\unit{\m}$~\cite[\S 2.4]{palmroth_2018_vlasiator}. 
Moreover, the Vlasov equation has infinitely many invariants due to its Hamiltonian structure. Retaining such conservation properties at the fully discrete level is an important yet nontrivial task. 
Numerically handling the high-dimensional phase space of the Vlasov equation is also an ongoing challenge. Standard grid discretization, such as finite differencing, of the six-dimensional phase space results in a prohibitively large number of degrees of freedom. For example, discretizing each direction in phase space using $10^{3}$ grid points results in $10^{18}$ degrees of freedom in each time step, which is equivalent to 8 exabytes of memory in double precision. At the time of this writing, the world's fastest supercomputer (Frontier) has 9.2 petabytes (0.0092 exabytes) of memory~\cite{dongarra_2022_frontier}.

Two approaches for solving the kinetic equations are prevalent: Lagrangian (particle-based) and Eulerian (grid or spectral-based). 
The first approach discretizes phase space via macroparticles and follows their characteristics. Particle-based methods are the most commonly used methods in the community due to their simplicity, robustness, and conservation properties~\cite{kraus_2017_gempic, kormann_jcp_2021, stanier_jcp_2019, guo_jcp_2022}. However, particle-based methods introduce statistical errors commonly referred to as `discrete particle noise' due to random sampling of the particle distribution function and its interpolation to the grid. The most straightforward way to reduce the noise is to increase the number of macroparticles, in which the noise decreases only as $1/\sqrt{N}$, where $N$ is the number of macroparticles~\cite{verboncoeur_2005_pic, camporeale_sps_2016}. Other denoising strategies include filtering~\cite{verboncoeur_2005_pic}, variance reduction techniques~\cite{denton_1995_pic}, and phase space remapping~\cite{wang_2011_pic, myers_2017_pic}. This renders particle-based methods more suitable for problems where a low signal-to-noise ratio is acceptable. 

The second approach is to solve the Vlasov equation in an Eulerian coordinate frame using a grid or spectral-based discretization. Eulerian methods do not have particle noise and therefore it is easier to resolve fine-scale structures. Grid-based methods have been successful typically in reduced dimensions, with a variety of discretization techniques such as finite difference~\cite{carrie_2022_central, takashi_2019_fd, banks_siam_2019}, finite volume~\cite{filbert_2001_fv, vogman_2018, banks_ieee_2010}, finite element~\cite{zaki_1988_fe2, zaki_1988_fe1}, and discontinuous Galerkin~\cite{heath_2012_dg, cheng_2014_dg, juno_2018_dg, datta_jcp_2023}.  
Alternatively, spectral methods expand the particle distribution function in velocity using an orthogonal basis. This can be achieved using different basis sets, including Fourier~\cite{klimas_1994_filter}, Chebyshev~\cite{shoucri_1974_chebyshev}, and Legendre~\cite{manzini_2016_legendre} expansions, but most of the spectral literature for the Vlasov equation use a Hermite spectral expansion in velocity space~\cite{roytershteyn_2018_sps, koshkarov_2021_sps, camporeale_sps_2016, vencels_2015_sps, vencels_2016_sps, delzanno_2015_sps, holloway_1996_sw, pagliantini_2023_adaptive, pagliantini_2023_time_integration, francis_2020_dg_hermite, pezzi_2016_recurrence, chatard_2022, kormann_2021_sw, parker_2015_filtering}.

Hermite-based velocity expansions originally proposed by~\cite{grad_1949_sw} are particularly advantageous since they can capture near-Maxwellian distribution functions with a few basis functions, enabling 3D-3V simulations~\cite{delzanno_2015_sps, vencels_2016_sps, roytershteyn_2018_sps}. This is because the Hermite basis is defined by the Hermite polynomials with a Maxwellian weight. Therefore, an exact Maxwellian distribution can be represented by the zeroth-order Hermite basis function. There are two types of Hermite expansions: \textit{asymmetrically weighted}~(AW) and \textit{symmetrically weighted}~(SW). The two expansions have different properties. Although the AW expansion conserves particle number, momentum, and energy, it suffers from numerical instability~\cite{camporeale_sps_2016}. Conversely, the SW expansion is stable, yet conservation of particle number, momentum, and energy is limited~\cite{manzini_2017_sw, holloway_1996_sw, kormann_2021_sw}. Additionally, both expansions do not enforce the positivity of the particle distribution function, due to the nature of working with an orthogonal basis.

An important property of the Vlasov equation is the anti-symmetric (also known as skew-adjoint, anti-self-adjoint, or skew-symmetry) structure of the phase space advection operator. Since an anti-symmetric operator conserves square norms, a discretization that preserves the anti-symmetric structure of the advection operator in the Vlasov equation is guaranteed to be numerically stable.
The exploitation of this concept dates back to a classic paper by Arakawa~\cite{arakawa_1966_incompressible}, which introduced the first numerically stable 2D finite-difference method for non-linear incompressible fluid flows. Numerical stability was achieved by introducing a linear combination of different representations of the advection operator with different numerical properties. This allows retaining some of the symmetries of the dynamics, such as the anti-symmetry of the flow operator and two square quantities -- kinetic energy and enstrophy -- to be conserved in discrete space, resulting in a numerically stable method. 
Recently, Halpern et al.~\cite{halpern_2018_antisymmetric, halpern_2020_antisymmetric, halpern_2021_antisymmetric} formulated an anti-symmetric representation of the plasma and fluid equations, ranging from the two-fluid Braginskii model down to the Navier-Stokes model. This formulation first transforms the continuous equations by introducing a novel set of dependent variables in which the continuous force operator becomes anti-symmetric. The new dependent variables include the square root of the density, similar to Roe variables for the Euler equations~\cite{roe_1981_sqrt}, which by construction is positivity preserving in the discrete setting. The spatial discretization of the transformed equations is based on central finite differencing, which results in a discrete anti-symmetric force operator, leading to a stable, robust, and conservative method. The numerical stability and conservation properties are guaranteed in the anti-symmetric formulation by drawing connections between the continuous and discrete forms, instead of through direct discretization of the Hamiltonian function. 

This paper has four main contributions.
First, we show that the SW expansion of the particle distribution function preserves the anti-symmetric structure of the advection operator, or equivalently the canonical Poisson bracket, in the Vlasov equation. 
Second, we show that central finite differencing in space can preserve the anti-symmetric structure and result in a discrete anti-symmetric advection operator acting on the expansion coefficients. Preserving the anti-symmetry of the advection operator results in a structure-preserving and stable method. 
Additionally, central finite differencing scales well in high-dimensional parallel code due to its sparsity and allows for more flexibility with non-periodic boundary conditions, unlike the previously proposed Fourier spectral expansion.
Third, we propose a new SW square-root formulation, which solves for the square root of the particle distribution function, with the same SW velocity expansion and central finite differencing spatial discretization. By construction, the SW square-root formulation is positivity preserving and retains the numerical stability and anti-symmetry of the SW formulation.  
Fourth, we derive the conservation properties of the SW and SW square-root formulations for the one-dimensional Vlasov-Poisson equations. We verify numerically such conservation properties on the following benchmark problems: manufactured solution, linear and nonlinear Landau damping, two-stream instability, bump-on-tail instability, and ion-acoustic wave. 

This paper is organized as follows. Section~\ref{sec:vp-hermite-sec-2} introduces the one-dimensional Vlasov-Poisson equations and the velocity discretization for the SW and SW square-root formulations. Section~\ref{sec:anti-symmetry-sec-3} presents the proposed anti-symmetric discretization of the Vlasov equation, leading to a stable method. Section~\ref{sec:conservation-properties} derives the conservation properties of the SW and SW square-root formulations. In section~\ref{sec:numerical-results}, we test the two formulations on a number of classical benchmark problems. Lastly, section~\ref{sec:conclusions} offers conclusions and an outlook for future work.

\section{Vlasov-Poisson Equations: Hermite Spectral Discretization in Velocity}\label{sec:vp-hermite-sec-2}
We present the Vlasov-Poisson equations in section~\ref{sec:vlasov-poisson-equations}. Section~\ref{sec:Hermite} presents the SW Hermite basis functions. We introduce the SW  expansion in section~\ref{sec:SW} and the SW square-root expansion in section~\ref{sec:SW-sqrt}.

\subsection{Vlasov-Poisson Equations}\label{sec:vlasov-poisson-equations}
We study the one-dimensional Vlasov-Poisson equations which model the coupling between collisionless plasmas and a self-consistent electric field.  
The plasma is described by the non-negative and bounded particle distribution function $f^{s}(x, v, t)$ in phase-space, which represents the number of particles located at the spatial position between $x \in \Omega_{x}$ and $x + \mathrm{d}x \in \Omega_{x}$ with velocity between $v \in \Omega_{v}$ and $v +\mathrm{d} v\in \Omega_{v}$ at a given time $t\in \Omega_{t} \subseteq \mathbb{R}_{\geq 0}$. 
The plasma species are denoted by $s$, where $s=`e$' denotes electrons and $s=`i$' denotes singly-charged ions. 
We assume the spatial coordinate $x$ is periodic in the spatial domain $\Omega_{x} = [0, \ell]$, where $\ell$ is the length of the spatial domain, $\Omega_{t} = [0, t_{f}]$, where $t_{f} \in \mathbb{R}_{\geq 0}$ is the final time, and allow the velocity coordinate to be unbounded, i.e. $\Omega_{v} = \mathbb{R}$.
Let $L^{2}(\Omega_{x} \times \Omega_{v})$ denote the real Hilbert space of square-integrable functions on $\Omega_{x} \times \Omega_{v}$ with inner product $\langle f, g\rangle = \int_{\Omega_{x}}\int_{\Omega_{v}}f(x, v, t) g(x, v, t) \mathrm{d} x\mathrm{d}v$. We next introduce the spaces of functions that are periodic in space $x$ and tend to zero as $|v|\to \infty$ as
\begin{align*}
    \mathcal{V}^d \coloneqq & \{ [f_1, \ldots, f_d]^\top(\cdot, \cdot, t)\ \text{s.t.} \ f_i(\cdot, \cdot, t) \in L^{2}(\Omega_{x} \times \Omega_{v}) \  \text{and} \ f_i\geq 0,  \\
    &\Hquad f_i(0, \cdot, t) = f_i(\ell, \cdot, t), \Hquad f_i(\cdot, v, t) \to 0 \Hquad \mathrm{as} \Hquad |v| \to \infty, \ i=1, \ldots, d, \Hquad \forall t \in \Omega_{t}\},\\
    \mathcal{W}^d \coloneqq & \{ [g_1, \ldots, g_d]^\top(\cdot,  t) \ \text{s.t.} \   g_i(\cdot, t) \in L^{2}(\Omega_{x}) \ \text{and} \  g_i(0, t) = g_i(\ell, t), \ i=1, \ldots, d, \Hquad \forall t \in \Omega_{t} \},
\end{align*}
where we abbreviate $\mathcal{V} = \mathcal{V}^1$ and $\mathcal{W} = \mathcal{W}^1$. 
The one-dimensional~(normalized) Vlasov-Poisson equations are given by 
\begin{alignat}{3}
    \left(\frac{\partial}{\partial t} + v \frac{\partial}{\partial x}+ \frac{q^{s}}{m^{s}} E(x, t) \frac{\partial}{\partial v}\right) f^{s}(x, v, t)  &= 0,  \qquad &&\text{in } \Omega_{x} \times \Omega_{v} \times \Omega_{t},\Hquad \forall s, \label{vlasov-continuum}\\
    \frac{\partial E(x, t)}{\partial x} &= \rho(x, t)  \qquad &&\text{in } \Omega_{x} \times \Omega_{t}, \label{poisson-continuum}
\end{alignat}
where $\rho(x, t) \coloneqq \sum_{s} q^{s} \int_{\Omega_{v}} f^{s}(x,v, t) \mathrm{d}v$ is the normalized charge density, $E \in \mathcal{W}$ is the normalized electric field, and $m^{s}$ and $q^{s}$ are the normalized mass and charge of species $s$.  
Another macroscopic quantity that does not show up explicitly in the governing equations but will be used in the following sections is the normalized electric current $J(x, t) \coloneqq \sum_{s} q^{s} \int_{\Omega_{v}} v f^{s}(x, v, t) \mathrm{d} v$.
The uniqueness of the solution is ensured by imposing that $\int_{\Omega_{x}}E(x, t) \mathrm{d} x = 0, \Hquad \forall t \in \Omega_{t}$.
To complete the problem formulation, we set a suitable initial condition to the particle distribution function $f^{s}(x, v, t=0) = f^{s}_{0}(x, v)$ for each of the species $s$, and compute the initial electric field $E(x, t=0)$ by solving the Poisson equation~\eqref{poisson-continuum} at $t=0$.
\paragraph{Normalization} In the above, we normalized the physical quantities as follows. Time $t$ is normalized to the electron plasma frequency $\omega_{pe} = \sqrt{(q^{e})^2 n_{0}^{e}/(\varepsilon_{0}m^{e})}$, where $q^{e}$ is the elementary charge, $m^{e}$ is the electron mass, $\varepsilon_{0}$ is the permittivity of vacuum, and $n_{0}^{e}$ is the reference electron density. The velocity coordinate $v$ is normalized to the electron thermal velocity $v_{t} = \sqrt{k_{B} T^{e}/m^{e}}$, where $k_{B}$ is the Boltzmann constant and $T^{e}$ is a reference electron temperature; the spatial coordinate $x$ is normalized to the electron Debye length $\lambda_{D} = \sqrt{\varepsilon_{0} k_{B} T^{e}/((q^{e})^{2}n_{0}^{e})}$; the electric field $E(x, t)$ is normalized to $m^{e} v_{t} \omega_{pe}/q^{e}$; the particle distribution function $f^{s}(x, v, t)$ is normalized to $n^{e}_{0}/v_{t}$; the mass and charge of species $s$, i.e. $m^{s}$ and $q^{s}$, are normalized by electron mass $m^{e}$ and charge $q^{e}$.

\subsection{Symmetrically-Weighted Hermite Basis Functions and Their Properties}\label{sec:Hermite}
We introduce the \textit{symmetrically-weighted}~(SW) Hermite basis functions~$\{\psi_{n}(\xi^{s})\}_{n \in \mathbb{N}_{0}}$, which are a function of the scaled velocity coordinate~$\xi^{s}$ and are defined as
\begin{equation}\label{hermite-basis-function-definition}
    \psi_{n}(\xi^{s}) \coloneqq (\sqrt{\pi} 2^n n!)^{-\frac{1}{2}} \mathcal{H}_{n}(\xi^{s}) \exp{\left(-\frac{(\xi^{s})^2}{2}\right)}, \qquad \xi^{s}(v) \coloneqq \frac{v-u^{s}}{\alpha^{s}}, 
\end{equation}
where the velocity coordinate $v$ is shifted by $u^{s}\in \mathbb{R}$ and scaled by $\alpha^{s} \in \mathbb{R}_{>0}$~\cite{tao_1993_hermite}. Here, we treat parameters $u^{s}$ and $\alpha^{s}$ as constants. We briefly discuss the adaptive (non-constant) approach in section~\ref{sec:hermite-adaptivity}.
The SW Hermite parameter $u^{s}$ and $\alpha^{s}$ can be interpreted as the characteristic mean flow and thermal velocity of each species $s$ due to the Hermite basis $\psi_{n}(\xi^{s})$ resemblance of a Maxwellian distribution.
Moreover, $\mathcal{H}_{n}(\xi^{s})$ is the ``\textit{physicist}'' Hermite polynomial~\cite[\S 22]{abramowitz_1964_math} of degree $n \in \mathbb{N}_{0}$, i.e.
\begin{equation}\label{hermite-polynomials}
    \mathcal{H}_{n}(\xi^{s}) \coloneqq (-1)^n \exp{\left((\xi^{s})^2\right)} \frac{\mathrm{d}^{n}}{\mathrm{d}(\xi^{s})^n} \exp{\left((-\xi^{s})^2\right)}, 
\end{equation}
with the following recursive properties 
\begin{alignat}{3}
    \frac{\mathrm{d}\mathcal{H}_{n}(\xi^{s})}{\mathrm{d}\xi^s}  &= 2n \mathcal{H}_{n-1}(\xi^s) \qquad &&\mathrm{for} \qquad n\geq 1, \label{hermite-recursion-derivative}\\
    \xi^{s}\mathcal{H}_{n}(\xi^s) &= \frac{1}{2} \mathcal{H}_{n+1}(\xi^s) + n \mathcal{H}_{n-1}(\xi^s) \qquad &&\mathrm{for} \qquad n \geq 2.\label{hermite-recursion} 
\end{alignat}
The SW-Hermite basis functions satisfy the orthogonality relation
\begin{align}\label{orthogonal-property}
    \int_{\mathbb{R}} \psi_{n}(\xi^{s}) \psi_{m}(\xi^{s}) \mathrm{d} \xi^{s} = \delta_{n,m},
\end{align}
where $\delta_{n,m}$ is the Kronecker delta function. Additional properties of the SW-Hermite basis functions that we leverage later include
\begin{align}
    v\psi_{n}(\xi^{s}) &= \alpha^{s} \sqrt{\frac{n+1}{{2}}} \psi_{n+1}(\xi^{s}) + \alpha^{s} \sqrt{\frac{n}{2}} \psi_{n-1}(\xi^{s}) + u^{s} \psi_{n}(\xi^{s}), \label{identity-1}\\
    v^2\psi_{n}(\xi^{s}) &= (\alpha^{s})^2 \frac{\sqrt{(n+1)(n+2)}}{2} \psi_{n+2}(\xi^{s}) + 2\alpha^{s}u^{s}\sqrt{\frac{n+1}{2}}\psi_{n+1}(\xi^s) \label{identity-2}\\
    &+\left((\alpha^{s})^2 \left(\frac{2n+1}{2}\right) + (u^s)^2\right) \psi_{n}(\xi^{s}) \nonumber\\
    &+ 2\alpha^{s} u^s \sqrt{\frac{n}{2}} \psi_{n-1}(\xi^s) + (\alpha^s)^2 \frac{\sqrt{n(n-1)}}{2} \psi_{n-2}(\xi^s), \nonumber\\
    \frac{\mathrm{d} \psi_{n}(\xi^{s})}{\mathrm{d} v} &=  \frac{1}{\alpha^{s}}\sqrt{\frac{n}{2}} \psi_{n-1}(\xi^{s}) - \frac{1}{\alpha^{s}}\sqrt{\frac{n+1}{2}}\psi_{n+1}(\xi^{s}).\label{identity-3}
\end{align}
Identity~\eqref{identity-1} follows from the recursive relation in Eq.~\eqref{hermite-recursion}, identity~\eqref{identity-2} follows from applying the recursive formula in Eq.~\eqref{hermite-recursion} twice, and lastly, identity~\eqref{identity-3} follows from the recursive relation in Eq.~\eqref{hermite-recursion-derivative}.

\subsection{Symmetrically-Weighted Expansion}\label{sec:SW}
We approximate the particle distribution function $f^{s}(x, v, t)$ via a truncated spectral decomposition in the scaled velocity coordinate~$\xi^{s}$, taking the following ansatz
\begin{equation}\label{expansion-seperation-of-variables}
    f^{s}(x, v,t) \approx f^{s, N_{v}}(x, v, t) = \sum_{n=0}^{N_{v}-1} C_{n}^{s}(x, t) \psi_{n}(\xi^{s}), 
\end{equation}
where $\psi_{n}(\xi^{s})$ is the SW-Hermite basis function defined in Eq.~\eqref{hermite-basis-function-definition} and $C_{n}^{s} \in \mathcal{W}$ is the expansion coefficient. The Hermite basis is a natural choice for Maxwellian-type distributions because the Hermite basis is defined by the Hermite polynomials with a Maxwellian weight, see Eq.~\eqref{hermite-basis-function-definition}. Additionally, it is advantageous to consider the Hermite basis for approximating solutions of differential equations in unbounded domains to avoid artificial boundary conditions. 
By substituting the spectral approximation in Eq.~\eqref{expansion-seperation-of-variables} in the Vlasov equation~\eqref{vlasov-continuum}, multiplying against~$\psi_{m}(\xi^{s})$, and integrating with respect to~$\xi^{s}$, we get
\begin{equation*}
    \int_{\mathbb{R}} \psi_{m}(\xi^{s}) \left[ \sum_{n=0}^{N_{v}-1} \frac{\partial C_{n}^{s}(x, t)}{\partial t} \psi_{n}(\xi^{s}) + v \sum_{n=0}^{N_{v}-1} \frac{\partial C_{n}^{s}(x, t)}{\partial x} \psi_{n}(\xi^{s}) +\frac{q^{s}}{m^{s}} E(x, t)\sum_{n=0}^{N_{v}-1} C_{n}^{s}(x, t) \frac{\mathrm{d} \psi_{n}(\xi^{s})}{\mathrm{d} v} \right]\mathrm{d}\xi^{s}= 0.
\end{equation*}
By exploiting the orthogonality and recurrence relation of the Hermite basis functions in Eqns.~\eqref{orthogonal-property},~\eqref{identity-1} and~\eqref{identity-3}, we derive a system of partial differential equations (PDEs) for the expansion coefficients $C_{n}^{s}(x, t)$:
\begin{align}\label{f-vlasov-pde}
    \frac{\partial C_{n}^{s}(x, t)}{\partial t} &+ \frac{\partial }{\partial x}\underbrace{\left(\alpha^{s} \sqrt{\frac{n}{{2}}} C_{n-1}^{s}(x, t) + \alpha^{s} \sqrt{\frac{n+1}{2}} C_{n+1}^{s}(x, t) + u^{s}  C_{n}^{s}(x, t)\right)}_{\coloneqq Q^{s}_{n}(x, t)}\\
    &+\frac{q^{s}}{m^{s}\alpha^{s}} 
 E(x,t)\left(-\sqrt{\frac{n}{2}}C_{n-1}^{s}(x, t) + \sqrt{\frac{n+1}{2}} C_{n+1}^{s}(x, t)\right)= 0, \nonumber
\end{align}
for $n= 0, 1, \ldots, N_{v}-1$ with the convention that $C_{n<0}^{s}(x, t) = 0$ and imposing for the closure of the system $C_{N_{v}}^{s}(x, t) = 0$, which is commonly known as \textit{closure by truncation}. Similarly, after substituting the spectral approximation in Eq.~\eqref{expansion-seperation-of-variables} in the Poisson equation~\eqref{poisson-continuum}, we get
\begin{equation}\label{f-poisson-pde}
    \frac{\partial E(x, t)}{\partial x} = \sum_{s} q^{s} \alpha^{s} \sum_{n=0}^{N_{v}-1} \mathcal{I}_{n} C_{n}^{s}(x, t)\qquad \mathrm{with} \qquad \mathcal{I}_{n} = \int_{\mathbb{R}} \psi_{n}(\xi^{s}) \mathrm{d} \xi^{s}.
\end{equation}
The integral $\mathcal{I}_{n}$ has the following recursive property:
\begin{equation}\label{recursive-relation-I0}
    \mathcal{I}_{0} = \sqrt{2} (\pi)^{\frac{1}{4}}, \qquad \mathcal{I}_{1} = 0, \qquad \mathcal{I}_{n} = \sqrt{\frac{n-1}{n}} \mathcal{I}_{n-2} \qquad \mathrm{for} \qquad n \geq 2. 
\end{equation}
%
\subsection{Symmetrically-Weighted Square-Root Expansion}\label{sec:SW-sqrt}
To preserve the non-negative property of the particle distribution function $f^{s}(x, v, t) \geq 0$, we approximate the square root of the particle distribution function, i.e. $\sqrt{f^{s}(x, v, t)}$, via an SW-Hermite spectral expansion in velocity $v$, where
\begin{equation}\label{spectral-approximation-sqrt}
    \sqrt{f^{s}(x, v, t)} \approx \sqrt{f^{s, N_{v}}(x, v, t)} = \sum_{n=0}^{N_{v}-1} C_{n}^{s}(x, t) \psi_{n}(\xi^{s}).
\end{equation}
In general, the approach of solving for the square root of the particle distribution function can be applied to a grid-based kinetic solver with any discretization in the phase space.
\begin{proposition}\label{prop:sqrt-formulation} 
The Vlasov equation~\eqref{vlasov-continuum} in square-root form is
\begin{equation}\label{vlasov-sqrt-continuum}
    \left(\frac{\partial}{\partial t} + v \frac{\partial}{\partial x}+ \frac{q^{s}}{m^{s}} E(x, t) \frac{\partial}{\partial v}\right) \sqrt{f^{s}(x, v, t)} = 0. 
\end{equation}
A solution to Eq.~\eqref{vlasov-sqrt-continuum} is a solution to the Vlasov equation~\eqref{vlasov-continuum} and vice versa.
\end{proposition}
\begin{proof}
By the product rule and non-negative property of $f^{s}(x, v, t)$, we get
\begin{equation}\label{sqrt-product-rule}
    \frac{\partial f^{s}(x, v, t)}{\partial t} =  \frac{\partial \left(\sqrt{f^{s}(x, v, t)}\right)^2}{\partial t} = 2\sqrt{f^{s}(x, v, t)} \frac{\partial \sqrt{f^{s}(x, v, t)}}{\partial t}.
\end{equation}
Equation~\eqref{sqrt-product-rule} equivalently holds for the partial derivatives of $f^{s}(x, v, t)$ with respect to $x$ and $v$. We recast all partial derivatives in the Vlasov equation~\eqref{vlasov-continuum} using Eq.~\eqref{sqrt-product-rule}, which gives
\begin{equation*}
    2\sqrt{f^{s}(x, v, t)} \left(\frac{\partial}{\partial t} + v \frac{\partial}{\partial x}+ \frac{q^{s}}{m^{s}} E(x, t) \frac{\partial}{\partial v}\right) \sqrt{f^{s}(x, v, t)} = 0.
\end{equation*}
Thus, either $\sqrt{f^{s}(x, v, t)} = 0$  or Eq.~\eqref{vlasov-sqrt-continuum} holds. Since Eq.~\eqref{vlasov-sqrt-continuum} encompasses the trivial solution $\sqrt{f^{s}(x, v, t)} = 0$, the result in Proposition~\ref{prop:sqrt-formulation} is proven.
\end{proof}

By Proposition~\ref{prop:sqrt-formulation} we obtain the same system of PDEs for the expansion coefficients $C_{n}^{s}(x, t)$ in Eq.~\eqref{spectral-approximation-sqrt} as the SW expansion coefficients in Eq.~\eqref{f-vlasov-pde}. Moreover, the square-root spectral approximation~\eqref{spectral-approximation-sqrt} introduces quadratic non-linearity in the Poisson equation~\eqref{poisson-continuum}, such that
\begin{equation}\label{sqrt-f-poisson}
    \frac{\partial E(x, t)}{\partial x} = \sum_{s} q^{s} \alpha^{s} \sum_{n=0}^{N_{v}-1} C_{n}^{s}(x, t)^2.
\end{equation}
The only difference between the SW and SW square-root formulations is the right-hand side of the Poisson equation~\eqref{f-poisson-pde} and \eqref{sqrt-f-poisson}. We derived a closed-form analytic relation between the SW and SW square-root expansion coefficients in~\ref{sec:Appendix-B}.

\section{Anti-symmetric Representation} \label{sec:anti-symmetry-sec-3}
The Vlasov-Poisson system has a rich geometric Hamiltonian structure. Preserving such geometric structure on the discrete level results in accurate long-term numerical simulations of plasma physics processes~\cite{kraus_2017_gempic, kraus_2016_rmhd_variational, xiao_2018_gempic, holloway_1996_hamiltonian}. 
Geometric structure includes conservation laws, such as the conservation of mass, momentum, and energy, symmetries and constraints, such as the particle distribution function being non-negative and bounded, and identities, for example, those from vector calculus, $\text{curl grad} = 0$ and $\text{div curl} = 0$. Recently, Halpern et al.~\cite{halpern_2018_antisymmetric, halpern_2020_antisymmetric, halpern_2021_antisymmetric} developed an anti-symmetric structure-preserving formulation of the Navier-Stokes and magnetohydrodynamics equations leading to simple, robust, and stable simulations.
Here, we draw insight from that work and focus on preserving the anti-symmetric property of the advection operator in the Vlasov equation~\eqref{vlasov-continuum}, thereby preventing nonphysical dissipation. For ease of exposition, we present the anti-symmetric formulation in the 1D-1V setting, however, the anti-symmetric formulation can be extended to multidimensional problems with minor modifications. 

Section~\ref{adjoint-sec:definitions} presents preliminary definitions and notation. In section~\ref{sec:anti-symmetry}, we show that the semi-discrete Vlasov equation preserves the anti-symmetry of the advection operator. We explain the proposed method's stability properties in section~\ref{sec:stability} and its time-reversibility in section~\ref{sec:time-reversibility}. Section~\ref{sec:hamiltonian-poisson-bracket} relates the anti-symmetric properties to the Hamiltonian canonical and noncanonical Poisson brackets. Lastly, section~\ref{sec:filamentation} discusses filamentation mitigation strategies, and section~\ref{sec:hermite-adaptivity} discusses the adaptivity of the Hermite basis parameters. 

\subsection{Definitions and Notation}\label{adjoint-sec:definitions}
We start our discussion by recalling the definition of the \textit{adjoint} of an operator and an \textit{anti-symmetric} operator along with simple examples. 
\begin{definition}[Adjoint of an operator~{\cite[\S 8.5.8]{renardy_2004_introduction}}] \label{def:adjoint}
    Let $H$ be a complex Hilbert space, whose inner product is denoted by $\langle \cdot, \cdot \rangle$. We consider a densely defined operator $(\mathcal{D}(A), A)$ from a Hilbert space $H_{1}$ to a Hilbert space $H_{2}$. The adjoint of the operator $(\mathcal{D}(A), A)$ is defined as the operator $(\mathcal{D}(A^{*}), A^{*})$ from $H_{2}$ to $H_{1}$, such that $\langle Ay, z \rangle = \langle y, A^{*} z\rangle,  \forall y \in H_{1}, \forall z\in H_{2}$. The asterisk denotes the conjugate transpose of the operator; for real operators, the conjugate transpose is the transpose, i.e. $A^{*} = A^{\top}$.
\end{definition}
\begin{example}\label{example-adjoint-discrete}
    We give an example of the computation of an adjoint of a discrete operator. Let $\mathcal{D}(A) = \mathbb{R}^{n}$, $A \in \mathbb{R}^{m \times n}$,  $y \in \mathbb{R}^{n}$, $z \in \mathbb{R}^{m}$, then the adjoint of the discrete operator $(\mathcal{D}(A), A)$ is its transpose $(\mathcal{D}(A^{\top}), A^{\top})$, where $\mathcal{D}(A^{\top}) = \mathbb{R}^{m}$ and $A^{\top} \in \mathbb{R}^{n\times m}$.
\end{example}
\begin{example}\label{example-adjoint-continous}
    We also give an example of the computation of an adjoint of a continuous operator. Let  $H = C^{1}(0, 1)$, the space of continuous functions that have continuous first derivatives on $[0, 1]$, $\mathcal{D}(A) = \{y \in H\Hquad \vert \Hquad y(0) = 0, \Hquad y(1) = 0 \}$, $y, z \in H$, and $A = \frac{\mathrm{d}}{\mathrm{d}x}$. We derive the adjoint of $A$ using integration by parts
    \begin{equation*}
        \langle Ay, z\rangle = \int_{0}^{1} zAy\mathrm{d} x = yz\Big|_{0}^{1} - \int_{0}^{1} y Az \mathrm{d}x = \langle y, -Az \rangle
    \end{equation*}
    such that $(\mathcal{D}(A), A^{\top}) = (\mathcal{D}(A), -A)$.
\end{example}
Example~\ref{example-adjoint-continous} leads to the definition of an anti-symmetric operator. 
\begin{definition}[Anti-symmetric operator~{\cite[\S 2]{halpern_2021_antisymmetric}}]\label{def:anti-symmetric}
The densely defined operator $(\mathcal{D}(A), A)$ from $H$ to $H$ is anti-symmetric{\footnote{also referred to as ``skew-adjoint'', ``anti-self-adjoint'', or ``skew-symmetric'' in literature.}} if $(\mathcal{D}(A), A) = (\mathcal{D}(A^{*}), -A^{*})$, where $(\mathcal{D}(A^{*}), A^{*})$ is the adjoint of the operator $(\mathcal{D}(A), A)$. An anti-symmetric operator $(\mathcal{D}(A), A)$ has the following properties:
\begin{alignat}{3}
    &\text{(anti-commutative) }\qquad\Hquad  \langle Ay, z\rangle &&= -\langle y, Az \rangle, \qquad &&\forall y, z \in H, \label{anti-commutative}\\
    &\text{(preserves square-norms) } \langle Ay, y\rangle &&= 0, \qquad &&\forall y \in H \label{preserves-square-norms}.
\end{alignat}
\end{definition}
%

\subsection{Continuum and Discrete Anti-symmetric Representation of the Vlasov Equation} \label{sec:anti-symmetry}
The advection operator in the Vlasov equation~\eqref{vlasov-continuum} is anti-symmetric (in a periodic domain), which we describe in detail in section~\ref{sec:anti-symmetric-continuum}. 
We prove in section~\ref{sec:anti-symmetric-v} that the anti-symmetric structure of the advection operator is preserved after discretization in velocity using SW-Hermite expansion, and in section~\ref{sec:anti-symmetric-x}, we show that anti-symmetry is preserved after discretization in space using centered finite differencing. 
%

\subsubsection{Anti-symmetric Representation in the Continuum}\label{sec:anti-symmetric-continuum}
We represent the Vlasov equation~\eqref{vlasov-continuum} and its square-root equivalent~\eqref{vlasov-sqrt-continuum} in anti-symmetric form, where the advection operator is anti-symmetric.
\begin{theorem}
The Vlasov equation~\eqref{vlasov-continuum} can be written in an anti-symmetric form as 
\begin{equation}\label{vlasov-continuum-anti-symmetric-form}
    \frac{\partial}{\partial t}f^{s}(x, v, t) = A^{s}(x, v, t) f^{s}(x, v, t)\qquad \text{or} \qquad \frac{\partial}{\partial t}\sqrt{f^{s}(x, v, t)} = A^{s}(x, v, t) \sqrt{f^{s}(x, v, t)},
\end{equation}
where 
\begin{equation}\label{advection-operator-continuum}
    A^{s}(x, v, t) = -\dot{\vec{x}} \cdot \nabla_{\vec{x}} 
\end{equation}
and $\vec{x} = [x, v]^{\top}$, $\dot{\vec{x}} = \frac{\mathrm{d}}{\mathrm{d}t} \vec{x} = [v, \frac{q^{s}}{m^{s}}E(x, t)]^{\top}$, such that the advection operator $(\mathcal{V}, A^{s}(x, v, t))$ is anti-symmetric in a periodic domain. 
\end{theorem}
\begin{proof}
The Vlasov equation~\eqref{vlasov-continuum} can be written as 
\begin{equation}\label{conservative-vlasov-continuum}
    \frac{\partial}{\partial t} f^{s}(x, v, t) =- \dot{\vec{x}} \cdot \nabla_{\vec{x}} f^{s}(x, v, t) = A^{s}(x, v, t) f^{s}(x, v, t),
\end{equation}
which proves the first equality in Eq.~\eqref{vlasov-continuum-anti-symmetric-form}. By the non-negativity of $f^{s}(x, v, t)$, the right hand side of Eq.~\eqref{conservative-vlasov-continuum} can be transformed to 
\begin{equation*}
    -\dot{\vec{x}} \cdot \nabla_{\vec{x}} f^{s}(x, v, t)  = -2\sqrt{f^{s}(x, v, t)} \left(\dot{\vec{x}} \cdot \nabla_{\vec{x}} \sqrt{f^{s}(x, v, t)}\right) = 2\sqrt{f^{s}(x, v, t)} A^{s}(x, v, t) \sqrt{f^{s}(x, v, t)}.
\end{equation*}
The partial derivative of the distribution function in time, i.e. $\frac{\partial f^{s}(x, v, t)}{\partial t}$, can be expressed in terms of its square root, see Eq.~\eqref{sqrt-product-rule}, such that Eq.~\eqref{conservative-vlasov-continuum} becomes the second equality in Eq.~\eqref{vlasov-continuum-anti-symmetric-form}.
The advection operator $\left(\mathcal{V}, A^{s}(x, v, t)\right)$ in Eq.~\eqref{advection-operator-continuum} is anti-symmetric since $\langle A^{s} y, z\rangle = \langle y, -A^{s} z\rangle$ holds for all $y, z\in \mathcal{V}$. The anti-symmetry is obtained by the product rule,  divergence theorem, and periodic boundary conditions. 
Therefore, we get $(\mathcal{V}, A^{s}(x, v, t)) = (\mathcal{V}, -A^{s}(x, v, t)^{\top})$. 
\end{proof}

\subsubsection{Anti-symmetric Representation after Discretization in Velocity}\label{sec:anti-symmetric-v}
We show that after discretization in velocity using the SW-Hermite spectral expansion, we preserve the anti-symmetric structure of the advection operator in the Vlasov equation~\eqref{vlasov-continuum-anti-symmetric-form} for the SW and SW square-root formulations. The anti-symmetry of $A^{s}(x, v, t)$ is preserved since the SW-Hermite expansion results in a self-adjoint Galerkin projection operator~\cite{gottlieb_1977_stability}.
\begin{theorem}\label{thm:theorem-velocity-antisymmetry}
    The SW-Hermite velocity spectral discretization in Eqns.~\eqref{expansion-seperation-of-variables} and~\eqref{spectral-approximation-sqrt} preserves the anti-symmetric structure of the advection operator $\left(\mathcal{V}, A^{s}(x, v, t)\right)$ in Eq.~\eqref{advection-operator-continuum}.
\end{theorem}
\begin{proof}
The system of PDEs for the SW and SW square-root formulation expansion coefficients $C_{n}^{s}(x, t)$ in Eq.~\eqref{f-vlasov-pde} can be written in vector form as 
\begin{equation}\label{vlasov-vector-form-pde}
    \frac{\partial}{\partial t} \Psi^{s}(x, t) = A_{v}^{s}(x, t)\Psi^{s}(x, t),
\end{equation}
where 
\begin{equation*}
    \Psi^{s}(x, t) \coloneqq \begin{bmatrix}
           C_{0}^{s}(x, t)\\
           C_{1}^{s}(x, t)\\
           \vdots \\
           C_{N_{v}-1}^{s}(x, t)\\
         \end{bmatrix} \in \mathcal{W}^{N_{v}},  
\end{equation*}
\begin{equation*}
    A_{v}^{s}(x, t) = \begin{bmatrix}
    -u^{s} \frac{\partial }{\partial x}&  -\sqrt{\frac{1}{2}}G_{+}^{s}(x, t) & 0 &  \ldots & 0 \\
    \sqrt{\frac{1}{2}} G_{-}^{s}(x, t) & -u^{s} \frac{\partial }{\partial x} & -\sqrt{\frac{2}{2}} G_{+}^{s}(x, t) &  \ldots & 0\\
    & &  & & \\
    & \ddots & \ddots &  \ddots&   \\
    & &  & & \\
    0 & 0 &  \sqrt{\frac{N_v-2}{2}} G_{-}^{s}(x, t)& -u^{s} \frac{\partial }{\partial x} & -\sqrt{\frac{N_v-1}{2}} G_{+}^{s}(x, t)   \\
    0 & 0 & 0 & \sqrt{\frac{N_v-1}{2}} G_{-}^{s}(x, t)& -u^{s} \frac{\partial }{\partial x}
 \end{bmatrix} \in \mathcal{W}^{N_{v} \times N_{v}},
\end{equation*}
such that
\begin{equation*}
    G_{+}^{s}(x, t) = \frac{q^{s}}{m^s \alpha^s} E(x, t) + \alpha^{s} \frac{\partial}{\partial x} \qquad \text{and}\qquad 
    G_{-}^{s}(x, t) = \frac{q^{s}}{m^s \alpha^s} E(x, t) - \alpha^{s} \frac{\partial}{\partial x}.
\end{equation*}
Since the operator $(\mathcal{W}, \frac{\partial}{\partial x})$ is anti-symmetric in a periodic domain, the operator $\left(\mathcal{W}^{N_{v}}, A_{v}^{s}(x, t)\right)$ is also anti-symmetric in a periodic domain, i.e. $(\mathcal{W}^{N_{v}}, A_{v}^{s}(x, t)) = (\mathcal{W}^{N_{v}}, -A_{v}^{s}(x, t)^{\top})$. 
\end{proof}

Theorem~\ref{thm:theorem-velocity-antisymmetry} holds only for the closure by truncation, i.e. $C_{N_{v}}^{s}(x, t) = 0$ in Eq.~\eqref{f-vlasov-pde}. Otherwise, modifying the last row in the operator $(\mathcal{W}^{N_{v}}, A_{v}^{s}(x, t))$ or including a source term breaks anti-symmetry.

\subsubsection{Anti-symmetric Representation after Discretization in Space}\label{sec:anti-symmetric-x}
After using the SW-Hermite spectral expansion in velocity, we discretize the spatial domain using centered finite differencing. We show that the centered finite differencing scheme preserves the anti-symmetric structure of the advection operator in the Vlasov equation~\eqref{vlasov-continuum-anti-symmetric-form} for the SW and SW square-root formulations. In the semi-discrete level, the dynamics are described as infinitesimal rotations since the advection operator becomes an anti-symmetric matrix whose nonzero eigenvalues are purely imaginary in the form of complex conjugate pairs~\cite{halpern_2021_antisymmetric}.  It is particularly advantageous to use centered finite differencing since aside from retaining the anti-symmetric property of the advection operator it also preserves important calculus identities such as $\text{curl grad}=0$ and $\text{div curl}=0$ in multi-dimensional settings~\cite{halpern_2021_antisymmetric}. 
\begin{theorem}
    Centered finite difference representation of spatial derivatives in Eq.~\eqref{vlasov-vector-form-pde} preserves the anti-symmetric structure of the advection operator~$(\mathcal{W}^{N_{v}}, A^{s}_{v}(x, t))$ in Eq.~\eqref{vlasov-vector-form-pde}, and thus preserves the anti-symmetric structure of the advection operator~$(\mathcal{V}, A^{s}(x, v, t))$ in Eq.~\eqref{vlasov-continuum-anti-symmetric-form}.
\end{theorem}
\begin{proof}
We discretize the spatial direction uniformly with $\Delta x$ mesh spacing and denote the discretized expansion coefficients as $\mathbf{C}^{s}_{n}(t) = [C^{s}_{n}(x^{(1)}, t), C^{s}_{n}(x^{(2)}, t), \ldots, C^{s}_{n}(x^{(N_{x})}, t)]^{\top} \in \mathbb{R}^{N_{x}}$, where $N_{x}$ is the number of mesh points in space and $x^{(j)}$ denotes the spatial grid location with index $j=1, 2, \ldots, N_{x}$. Discretizing Eq.~\eqref{vlasov-vector-form-pde} using centered finite differencing results in the following semi-discrete system of ordinary differential equations (ODEs):
\begin{equation}\label{vlasov-ode}
    \frac{\mathrm{d}}{\mathrm{d}t} \mathbf{\Psi}^{s}(t) = \mathbf{A}^{s}(t) \mathbf{\Psi}^{s}(t)
\end{equation}
where 
\begin{equation}\label{psi-state-definition}
    \mathbf{\Psi}^{s}(t) = \begin{bmatrix}
        \mathbf{C}^{s}_{0}(t)\\
        \mathbf{C}^{s}_{1}(t)\\
        \vdots \\
        \mathbf{C}^{s}_{N_{v}-1}(t)
    \end{bmatrix} \in  \mathbb{R}^{N_{v}N_{x}}
\end{equation}
\begin{equation*}
    \mathbf{A}^{s}(t) = \begin{bmatrix}
    -u^{s} \mathbf{D}^{p}&  -\sqrt{\frac{1}{2}}\mathbf{G}_{+}^{s, p}(t) & 0 &  \ldots & 0 \\
    \sqrt{\frac{1}{2}} \mathbf{G}_{-}^{s, p}(t) & -u^{s} \mathbf{D}^{p} & -\sqrt{\frac{2}{2}} \mathbf{G}_{+}^{s, p}(t) &  \ldots & 0\\
    & &  & & \\
    & \ddots & \ddots &  \ddots&   \\
    & &  & & \\
    0 & 0 &  \sqrt{\frac{N_v-2}{2}} \mathbf{G}_{-}^{s, p}(t)& -u^{s} \mathbf{D}^{p} & -\sqrt{\frac{N_v-1}{2}} \mathbf{G}_{+}^{s, p}(t)   \\
    0 & 0 & 0 & \sqrt{\frac{N_v-1}{2}} \mathbf{G}_{-}^{s, p}(t)& -u^{s} \mathbf{D}^{p}
 \end{bmatrix} \in \mathbb{R}^{N_{v} N_{x}\times N_{v}N_{x}},
\end{equation*}
such that $\mathbf{G}_{+}^{s, p}(t) = \frac{q^{s}}{m^s \alpha^s} \mathbf{E}(t) + \alpha^{s} \mathbf{D}^{p}$,  $\mathbf{G}_{-}^{s, p}(t) = \frac{q^{s}}{m^s \alpha^s} \mathbf{E}(t) - \alpha^{s}\mathbf{D}^{p}$, and the discrete form of the electric field is denoted by $\mathbf{E}(t) =\left[E(x^{(1)}, t), E(x^{(2)}, t), \ldots, E(x^{(N_{x})}, t)\right]^{\top} \in \mathbb{R}^{N_{x}}$. The semi-discrete Vlasov equation~\eqref{vlasov-ode} in the SW formulation is coupled to the following semi-discrete Poisson equation 
\begin{equation}\label{posisson-ode-sw}
    \mathbf{D}^{p} \mathbf{E}(t) = \sum_{s} q^{s} \alpha^{s} \sum_{n=0}^{N_{v}-1} \mathcal{I}_{n} \mathbf{C}_{n}^{s}(t), 
\end{equation}
and similarly, the equivalent semi-discrete Vlasov equation~\eqref{vlasov-ode} in the SW square-root formulation is coupled to 
\begin{equation}\label{poisson-ode-sw-sqrt}
    \mathbf{D}^{p} \mathbf{E}(t) = \sum_{s} q^{s} \alpha^{s} \sum_{n=0}^{N_{v}-1} \mathbf{C}_{n}^{s}(t)^2, 
\end{equation}
The periodic centered finite difference operator of order $p$ is denoted by 
\begin{equation}\label{central-finite-differencing}
    \mathbf{D}^{p} = \begin{bmatrix}
    d^{p}_{0} & d^{p}_{1} & d^{p}_{2} & \ldots & d^{p}_{N_{x}-1}\\
    d^{p}_{N_{x}-1} & d^{p}_{0} & d^{p}_{1} &  \ldots & d^{p}_{N_{x}-2}\\
    d^{p}_{N_{x}-2} & d^{p}_{N_{x}-1} & d^{p}_{0} &  \ldots & d^{p}_{N_{x}-3}\\
    & & &\ddots & \\
    d^{p}_{1} & d^{p}_{2} & d^{p}_{3} & \ldots & d^{p}_{0}\\
    \end{bmatrix} \in \mathbb{R}^{N_{x} \times N_{x}}.
\end{equation}
The stencil coefficients $d_{0}^{p}, d_{1}^{p}, \ldots, d_{N_{x}-1}^{p}$ are derived from the Taylor series approximation of order $p$. For example, the stencil coefficients for $p = 2, 4, 6, 8$ are defined as
\begin{alignat*}{3}
d^{2} &\coloneqq \frac{1}{\Delta x}\left[0 \Hquad \frac{1}{2} \Hquad 0 \Hquad \cdots \Hquad 0 \Hquad -\frac{1}{2} \right] \qquad \qquad &&\text{(2nd order)}\\
    d^{4} &\coloneqq \frac{1}{\Delta x} \left[0 \Hquad \frac{2}{3} \Hquad -\frac{1}{12} \Hquad 0 \Hquad \cdots \Hquad 0 \Hquad \frac{1}{12} \Hquad -\frac{2}{3}\right] \qquad \qquad &&\text{(4th order)}\\
    d^{6} &\coloneqq \frac{1}{\Delta x} \left[0 \Hquad \frac{3}{4} \Hquad  \Hquad -\frac{3}{20} \Hquad \frac{1}{60} \Hquad 0 \Hquad \cdots \Hquad 0 \Hquad -\frac{1}{60} \Hquad \frac{3}{20} \Hquad -\frac{3}{4}\right] \qquad \qquad &&\text{(6th order)}\\
    d^{8} &\coloneqq \frac{1}{\Delta x} \left[0 \Hquad \frac{4}{5} \Hquad  \Hquad -\frac{1}{5} \Hquad \frac{4}{105} \Hquad \frac{-1}{280} \Hquad 0 \Hquad \cdots \Hquad 0 \Hquad \frac{1}{280} \Hquad -\frac{4}{105} \Hquad \frac{1}{5} \Hquad -\frac{4}{5}\right] \qquad \qquad &&\text{(8th order)}
\end{alignat*}
The finite difference operator $\mathbf{D}^{p}$, defined in Eq.~\eqref{central-finite-differencing}, is a singular matrix, thus it is not invertible. Upon reflection around $d_0^{p}=0$, we can see that the centered finite difference operator $\mathbf{D}^{p}$ is anti-symmetric since $(\mathbb{R}^{N_{x}}, \mathbf{D}^{p}) = (\mathbb{R}^{N_{x}}, -(\mathbf{D}^{p})^{\top})$. Thus, the discretized advection operator is anti-symmetric as $(\mathbb{R}^{N_{v}N_{x}}, \mathbf{A}^{s}(t)) = (\mathbb{R}^{N_{v}N_{x}}, -\mathbf{A}^{s}(t)^{\top})$. 
\end{proof}

\begin{remark}
A non-uniform spacing in the spatial direction $x$ breaks anti-symmetry, since then $(\mathbb{R}^{N_{x}},  \mathbf{D}^{p}) \neq - (\mathbb{R}^{N_{x}}, (\mathbf{D}^{p})^{\top})$.
\end{remark}

\begin{remark}
Any spectral discretization in space, such as Fourier spectral expansion, also preserves the anti-symmetric structure of the advection operator~$(\mathcal{W}^{N_{v}}, A^{s}_{v}(x, t))$ in Eq.~\eqref{vlasov-vector-form-pde}, and thus preserves the anti-symmetric structure of the advection operator~$(\mathcal{V}, A^{s}(x, v, t))$ in Eq.~\eqref{vlasov-continuum-anti-symmetric-form}. \ref{sec:Appendix-A} derives the semi-discrete Vlasov-Poisson equations using a spatial Fourier spectral expansion and discusses in detail its anti-symmetric property. 
\end{remark}
\noindent
Handling problems with strong spatial discontinuities will require more sophisticated spatial discretization schemes, such as weighted essentially nonoscillatory or discontinuous Galerkin~\cite{koshkarov_2021_sps}.

\subsection{Numerical Stability}\label{sec:stability}
The SW and SW square-root formulations are algebraically stable as $\frac{\mathrm{d}}{\mathrm{d}t}\left(\mathbf{\Psi}^{s}(t)^{\top}\mathbf{\Psi}^{s}(t)\right) \leq 0$, see~\cite{gottlieb_1977_stability}. In particular, $\mathbf{\Psi}^{s}(t)^{\top}\mathbf{\Psi}^{s}(t)$, is a quadratic invariant of the SW and SW square root formulations. This is due to the anti-symmetric structure of the semi-discrete Vlasov equation~\eqref{vlasov-ode} and the following proposition. 
\begin{proposition}{{\cite[\S IV.1]{hairer_2006_geometric}}}\label{thm:anti-symmetric-invariants}
Consider an ODE of the form 
\begin{equation*}
    \frac{\mathrm{d}}{\mathrm{d}t} y = A(y)y,
\end{equation*}
where $y(t) \in \mathbb{R}^{n}$ and $(\mathbb{R}^{n}, A(y))$ is an anti-symmetric operator, then $y^{\top} y$ is an invariant of the ODE. 
\end{proposition}
\begin{proof} We recall the proof here to keep it self-consistent:
$\frac{\mathrm{d}}{\mathrm{d}t} (y^{\top} y) = (\frac{\mathrm{d}}{\mathrm{d}t}y)^{\top} y + y^{\top} \frac{\mathrm{d}}{\mathrm{d}t} y = y^{\top} A(y)^{\top} y + y^{\top} A(y) y = y^{\top}\left(A(y)^{\top}+ A(y)\right) y = 0.$
\end{proof}
Thus, since the square of the $L^{2}$-norm of the state vector $\mathbf{\Psi}^{s}(t)$ is conserved in the semi-discrete level, it guarantees that the system is algebraically stable for both SW and SW square-root formulations. Note that $\mathbf{\Psi}^{s}(t)^{\top}\mathbf{\Psi}^{s}(t)$ in the SW formulation is equivalent to the $L^{2}$-norm of the distribution function, i.e. $\int_{\Omega_{x}}\int_{\Omega_{v}} f^{s}(x, v, t)^{2} \mathrm{d}v \mathrm{d} x$, also known as the \textit{enstrophy}, and in the SW square-root formulation it is the discrete equivalent of the number of particles of species $s$, i.e. $\int_{\Omega_{x}}\int_{\Omega_{v}} f^{s}(x, v, t)\mathrm{d}v \mathrm{d}x$. The SW conservation of enstrophy is a well-known result shown by~\cite{holloway_1996_sw, schumer_1998_sw_aw, kormann_2021_sw} with a Fourier expansion in space, however, this work is the first to casts it under the umbrella of anti-symmetry and central finite differencing.

\subsection{Time Reversibility}\label{sec:time-reversibility}
The Vlasov-Poisson system of equations~\eqref{vlasov-continuum}-\eqref{poisson-continuum} are time-reversible, i.e. if all velocities and time were reversed, the plasma's prior motions would be replicated.
The anti-symmetric formulation is approximately time reversible when using temporal integrators that are not time-reversal symmetric~\cite{halpern_2020_antisymmetric}. For example, consider an explicit Euler scheme applied to Eq.~\eqref{vlasov-ode}, such that the solution at each time iteration $k$ is  
\begin{equation}\label{forward-euler-step}
    \mathbf{\Psi}^{s}_{k+1} = \left(\mathbf{I} +\Delta t\mathbf{A}^{s}_{k}\right) \mathbf{\Psi}^{s}_{k},
\end{equation}
where $\mathbf{I}$ is the identity matrix and $\Delta t$ is the time step. When $\Delta t$ is sufficiently small, we have
\begin{equation*}
    \left( \mathbf{I} + \Delta t \mathbf{A}^{s}_{k} \right)^{\top} \left(\mathbf{I} +\Delta t \mathbf{A}^{s}_{k}\right) = \left( \mathbf{I} - \Delta t \mathbf{A}^{s}_{k} \right)\left(\mathbf{I} +\Delta t \mathbf{A}^{s}_{k}\right) = \mathbf{I} + \mathcal{O}(\Delta t^2).
\end{equation*}
Multiplying Eq.~\eqref{forward-euler-step} by $(\mathbf{I} + \Delta t\mathbf{A}^{s}_{k})^{\top}$, results in  $\left(\mathbf{I} - \Delta t\mathbf{A}^{s}_{k}\right)\mathbf{\Psi}^{s}_{k+1} = \mathbf{\Psi}^{s}_{k} + \mathcal{O}(\Delta t^2)$, which is approximately (up to second-order accuracy) the time-reversed version of the explicit Euler method, see Eq.~\eqref{forward-euler-step}. 
Note that any time-symmetric Runge-Kutta method is implicit~\cite[\S V.2]{hairer_2006_geometric}, and the most straightforward among these methods is the implicit midpoint method. Although explicit temporal integrators are not time-symmetric, the anti-symmetric formulation renders such explicit integrators approximately time-reversible. 
We demonstrate the (approximate) time-reversibility of the anti-symmetric formulation on the linear Landau damping problem in section~\ref{sec:landau-numerical-results} with a third-order explicit Runge-Kutta method. 

\subsection{Relation to the Anti-symmetric Poisson Bracket}\label{sec:hamiltonian-poisson-bracket}
The canonical Poisson bracket~\cite{morrison_1981_1d1v_vp, morrison_1998_casimirs} is defined as
\begin{equation}\label{canonical-poisson-bracket}
    \left\{f, g\right\} \coloneqq \frac{\partial g}{\partial v} \frac{\partial f}{\partial x} - \frac{\partial g}{\partial x}\frac{\partial f}{\partial v},
\end{equation}
for two arbitrary functions $f, g$, which is equivalent to the Jacobian determinant $\partial (f, g)/\partial (x, v)$. The canonical Poisson bracket in Eq.~\eqref{canonical-poisson-bracket} is an anti-symmetric operator since $\left\{f, g\right\} = - \left\{g, f\right\}$. The Vlasov equation can then be written as 
\begin{equation*}
    \frac{\partial f^{s}}{\partial t} = \left\{f^{s}, \frac{1}{m^{s}}\frac{\delta \mathcal{E}}{\delta f^{s}}\right\},
\end{equation*}
where the functional 
\begin{equation} \label{eq:Efunctional}
    \mathcal{E}[f^{s}] \coloneqq \frac{1}{2} \int_{\Omega_{x}}\int_{\Omega_{v}}\sum_{s}m^{s} v^2 f^{s}(x, v, t) \mathrm{d}v \mathrm{d}x + \frac{1}{2} \int_{\Omega_{x}} E(x, t)^2\mathrm{d}x
\end{equation}
is the total energy of the system, i.e., the Hamiltonian. The variational derivative\footnote{The variational derivative $\frac{\delta \mathcal{E}}{\delta f^{s}}$ of a functional $\mathcal{E}[f^{s}]$ is defined as 
\begin{equation*}
    \int_{\Omega_{v}}\int_{\Omega_{x}} \frac{\delta \mathcal{E}}{\delta f^{s}} \delta f^{s} \mathrm{d}x \mathrm{d}v \coloneqq \lim_{\varepsilon \to 0}\frac{\mathcal{E}[f^{s}+\varepsilon\delta f^{s}] - \mathcal{E}[f^{s}]}{\varepsilon}
\end{equation*}.} of the total energy is $\frac{\delta \mathcal{E}}{\delta f^{s}} = \frac{1}{2}m^{s} v^2 + q^{s}\phi(x, t)$, where $\phi(x, t)$ is the electric potential, i.e. $E(x, t) =-\frac{\partial \phi(x, t)}{\partial x}$. The variational derivative is computed via integration by parts and the Poisson equation~\eqref{poisson-continuum}.
The anti-symmetric advection operator $A^{s}(x, v, t)$, defined in Eq.~\eqref{advection-operator-continuum}, can also be written as the canonical Poisson bracket, that is
\begin{equation*}
    A^{s}(x, v, t) =\left\{\cdot, \frac{1}{m^{s}}\frac{\delta \mathcal{E}}{\delta f^{s}}\right\},
\end{equation*}
Thus, by preserving the anti-symmetry of the advection operator in semi-discrete form, the proposed method can also preserve the anti-symmetry of the canonical Poisson bracket. 

The non-canonical Poisson bracket of the Vlasov-Poisson system~\cite{morrison_1981_1d1v_vp} is defined as
\begin{equation*}
    \{ \{ \mathcal{F}, \mathcal{G}\}\} \coloneqq \int_{\Omega_{v}}\int_{\Omega_{x}} \frac{f^{s}}{m^{s}}\left\{\frac{\delta \mathcal{F}}{\delta f^{s}}, \frac{\delta \mathcal{G}}{\delta f^{s}} \right\}\mathrm{d} x \mathrm{d} v, 
\end{equation*}
where $\mathcal{F}[f^{s}]$ and $\mathcal{G}[f^{s}]$ are arbitrary functionals acting on $f^s$. The Vlasov equation~\eqref{vlasov-continuum} can then be represented as 
\begin{equation*}
    \frac{\mathrm{d}}{\mathrm{d}t}\mathcal{F} = \{\{ \mathcal{F}, \mathcal{E}\}\}, 
\end{equation*}
where $\mathcal{F}[f^{s}]$ is any functional of $f^s$ and $\mathcal{E}$ is the total energy from Eq.~\eqref{eq:Efunctional}.  The noncanonical Poisson bracket is also anti-symmetric, such that $\{\{\mathcal{F}, \mathcal{E}\}\} = -\{\{\mathcal{E}, \mathcal{F}\}\}$. Therefore, by the anti-symmetry of the noncanonical Poisson bracket we know that $\frac{\mathrm{d}}{\mathrm{d} t}\mathcal{E} =\{ \{ \mathcal{E}, \mathcal{E} \} \} = -\{ \{ \mathcal{E}, \mathcal{E} \} \} = 0$. Any discretization that preserves the anti-symmetry of the noncanonical Poisson bracket will automatically conserve the Hamiltonian of the system, $\mathcal{E}[f^{s}]$, and other functional invariants, such as Casimirs $\mathcal{C}[f^{s}]$ (where $\{ \{\mathcal{C}, \mathcal{E} \} \}= 0$). As we show in section~\ref{sec:conservation-properties}, the SW and SW square-root formulations do not conserve the total energy $\mathcal{E}[f^{s}]$ of the system and therefore do not preserve the anti-symmetry of the noncanonical Poisson bracket. 
The work by~\cite{scovel_1994_lie_poisson} suggests that developing a discretization of the Vlasov-Poisson system with a discrete (Jacobi) noncanonical Poisson bracket can not be achieved via an Eulerian solver.

\subsection{Filamentation}\label{sec:filamentation}
Collisionless plasma can develop finer and finer scales in velocity resulting in phase space filamentation~\cite{cheng_1976_vp}. This is a consequence of the non-dissipative Hamiltonian structure of the governing equations. Filamentation is a common issue in Eulerian solvers due to their limited velocity resolution. In particular, in the case of using a spectral expansion in velocity space, the higher-order expansion coefficients try to overcompensate running out of resolution in velocity, such as in the nonlinear Landau damping test case (section~\ref{sec:nonlinear-landau-numerical-results}), which can cause numerical instabilities. Filamentation can also cause recurrence effects, where the solution is artificially periodic in time, such as in the linear Landau damping test case (section~\ref{sec:linear-landau-numerical-results}). In particular, the recurrence period scales approximately as the square root of the number of Hermite basis functions $\mathcal{O}(\sqrt{N_{v}})$~\cite{canosa_1974_recurrence, schumer_1998_sw_aw}. Previous work by~\cite{camporeale_sps_2016} introduces an artificial collisional term based on the Lenard-Bernstein collisional operator~\cite{lenard_bernstein_1958_collisions}, with a small collisional frequency, which dampens higher-order coefficients and, as a result, controls the filamentation effects. Incorporating this artificial collisional term does not break mass, momentum, or energy conservation when using the AW expansion, since such invariants are functions of the first three coefficients. However, applying an artificial collisional term to the SW expansion will break conservation laws since its invariants depend on all the coefficients in the expansion. Similarly, filtering techniques, such as Klimas~\cite{klimas_1994_filter} and Hou-Li~\cite{hou_li_2007_filter} filters, can be easily constructed for the AW expansion, yet are not trivially designed in the SW case since conservation laws depend on all expansion coefficients, which we discuss in detail in section~\ref{sec:conservation-properties}. Lastly, one can also apply flux limiter techniques to the higher-order coefficients to avoid spurious oscillations/numerical instabilities due to filamentation. The three strategies: (1) artificial collisions, (2) filtering, and (3) flux-limiters, are standard methods to handle filamentation, which should be implemented in a manner that avoids substantially influencing the conservation properties and physical results.

\subsection{Adaptive Hermite Parameters}\label{sec:hermite-adaptivity}
In general, allowing the Hermite spectral parameters $u^{s}$ and $\alpha^{s}$ in Eq.~\eqref{hermite-basis-function-definition} to vary in time can improve spectral convergence, stability, and mitigate filamentation~\cite{pagliantini_2023_adaptive, vencels_2016_masters_thesis, ma_2005_hermite, chatard_2022, fatone_2022_hermite, yana_2019_filtering, cai_2013_Nrxx, yin_2023_Nrxx}. 
For example, the two-stream instability, discussed in section~\ref{sec:two-stream-numerical-results}, leads to electron heating (manifested by an increase in $\alpha^{e}$) and deceleration of the average electron velocity of each beam (indicated by a decrease in $|u^{e}|$). Hence, the motivation for modifying the Hermite parameters over time stems from mathematical considerations, seeking to minimize errors in Hermite spectral approximation, and physics perspectives, aiming to describe accurately the plasma heating and bulk motion.
Since the Hermite shifting parameter $u^{s}$ can be interpreted as the characteristic fluid velocity and the scaling parameter $\alpha^{s}$ can be interpreted as the characteristic thermal velocity, the two Hermite parameters can be updated in time by monitoring the macroscopic quantities as shown by~\cite{pagliantini_2023_adaptive, vencels_2015_sps}, where
\begin{align}
    u^{s}(t) &\coloneqq \frac{1}{\mathcal{N}^{s}(t)} \int_{\Omega_{v}} \int_{\Omega_{x}} v f^{s}(x, v, t) \mathrm{d} x\mathrm{d} v, \qquad \mathrm{with} \qquad \mathcal{N}^{s}(t) \coloneqq \int_{\Omega_{v}} \int_{\Omega_{x}} f^{s}(x, v, t) \mathrm{d} x\mathrm{d} v, \label{update-u}\\
    \alpha^{s}(t) &\coloneqq \sqrt{\frac{1}{\mathcal{N}^{s}(t)} \int_{\Omega_{v}} \int_{\Omega_{x}} \left(v - u^{s}(t)\right)^2 f^{s}(x, v, t)\mathrm{d} x \mathrm{d} v}. \label{update-alpha}
\end{align}
Recent work by~\cite{pagliantini_2023_adaptive} showed that the time-dependent update of the Hermite parameters (for the AW-Hermite basis)  results in the evolution of the same semi-discrete equations as the non-adaptive method, only requiring an additional projection step at each update of the Hermite parameters.
Therefore, a comparable time-adaptive algorithm for the SW-Hermite basis will preserve the same semi-discrete equations as the non-adaptive method outlined in Eq.~\eqref{vlasov-ode}, maintaining the anti-symmetric structure of the advection operator.
Similar to the time-adaptive algorithm, a more general space-time adaptive algorithm, i.e. $u^{s}(x, t)$ and $\alpha^{s}(x, t)$, can also be done via tracking the first three moments of the distribution function. However, the spatial adaptivity of the Hermite parameters will change the semi-discrete equations and consequently the structure of the advection operator, which we leave for further investigation in future work.

\section{Conservation Properties of the Symmetrically Weighted and Symmetrically Weighted Square-Root Formulations} \label{sec:conservation-properties}
An ideal discretization method should conserve all the invariants of the Vlasov-Poisson system. This is a challenging problem since the system has an infinite number of invariants. 
The Vlasov-Poisson system invariants include the number of particles, total momentum, total energy, and infinitely many Casimirs~\cite{morrison_1998_casimirs}, i.e. $\mathcal{C}[f^{s}] = \int_{\Omega_{v}} \int_{\Omega_{x}} h(f^{s}(x, v, t))\mathrm{d}x \mathrm{d}v$, where $h(\cdot)$ is any function of the particle distribution function $f^{s}$. For example, by choosing $h(f^{s}) = f^{s}$, the Casimir invariant is the number of particles of species $s$, and by choosing $h(f^{s}) = f^{s}\log(f^{s})$, the Casimir invariant is the kinetic entropy, and lastly, by choosing $h(f^{s}) = (f^{s})^{p}$ with $p \in \mathbb{N}_{>0}$, the Casimir invariant is the $L^{p}$-norm of the distribution function~\cite{holloway_1996_sw}. 
We focus on analyzing the classical invariants, namely, the number of particles, total momentum, and total energy. 

The conservation constraints for the SW and SW square-root formulations (aside from energy conservation for the SW square-root formulation) are due to the SW expansion in velocity, i.e. analogous conservation properties hold prior to finite difference discretization in space. For each formulation, we derive the \textit{drift rate} as the rate at which conservation no longer holds, which is a quadratic expression of the expansion coefficients and electric field. The drift rates (aside from total energy in the SW square-root formulation) depend only on the last expansion coefficient, which we will see in Theorem~\ref{thm:mass-conservation-sw}--\ref{thm:momentum-conservation-sw-sqrt} below, such that in the limit of an infinite number of velocity spectral terms $N_{v}\to \infty$ the drift rate approaches zero. The derived drift rates correspond to the closure by truncation, i.e. $C_{N_{v}}(x, t) = 0$. One can modify the drift rates by changing the closure, however, only closure by truncation can preserve anti-symmetry. 

Table~\ref{tab:SW_vs_SW_sqrt} summarizes the conservation and stability properties of the SW and SW square-root formulations from section~\ref{sec:anti-symmetry-sec-3} and section~\ref{sec:conservation-properties}.
We discuss in greater detail the conservation properties of the SW formulation in section~\ref{sec:SW-conservation-properties} and the SW square-root formulation in section~\ref{sec:SW-sqrt-conservation-properties}. In all derivations we use the trapezoidal rule\footnote{The trapezoid rule for approximating the spectral expansion coefficients $C_{n}^{s}(x, t)$ spatial integral in a periodic domain is given by $\int_{0}^{\ell}C_{n}^{s}(x, t) \mathrm{d}x\approx \Delta x  \mathbf{1}^{\top} \mathbf{C}_{n}^{s}(t)$, where $\mathbf{1} \in \mathbb{R}^{N_{x}}$ is a column vector of all ones.} to approximate all integrals in space.
\noindent
\begin{table}
\caption{A comparison between the SW and SW square-root formulation properties.}
\centering
\begin{tabular}{c |c |c }
\textbf{Properties} & \textbf{SW} & \textbf{SW square-root} \\
\hline
conservation of the number of particles & if $N_{v}$ is odd & \cmark  \\
\hline
conservation of total momentum  & if $N_{v}$ is even \& $u^{s} = 0, \Hquad \forall s$ & \xmark \\
\hline
conservation of total energy  & if $N_{v}$ is odd \& $u^{s} = 0, \Hquad \forall s$ & \xmark\\
\hline
positivity preserving & \xmark & \cmark\\
\hline 
algebraically stable & \cmark & \cmark
\end{tabular}
\label{tab:SW_vs_SW_sqrt}
\end{table}

\subsection{Symmetrically Weighted Formulation Conservation Properties}\label{sec:SW-conservation-properties}
We show below that the SW formulation conserves the number of particles if $N_{v}$ is odd, total momentum if $N_{v}$ is even and $u^{s}=0$ for all species $s$, and total energy if $N_{v}$ is odd and $u^{s} = 0$ for all species $s$. Thus, the SW formulation conserves either total momentum or particle number and total energy, requiring $u^{s} = 0$ for all $s$.

\begin{theorem}[SW formulation particle number conservation]\label{thm:mass-conservation-sw}
Let $f^{s}_{\mathrm{sw}}(x, v, t)$ be the solution of the semi-discrete Vlasov-Poisson equations~\eqref{vlasov-ode}-\eqref{posisson-ode-sw} solved via the SW formulation. Then, the number of particles of each plasma species $s$, i.e.
\begin{equation*}
    \mathcal{N}^{s}(t) \coloneqq \int_{0}^{\ell}\int_{\mathbb{R}} f^{s}_{\mathrm{sw}}(x, v, t) \mathrm{d}v \mathrm{d}x,
\end{equation*}
is conserved in time if $N_{v}$ is odd, where the drift rate is
\begin{alignat}{3}\label{change-in-mass-sw}
\frac{\mathrm{d}\mathcal{N}^{s}}{\mathrm{d}t}=
\begin{cases}
0 \qquad &\text{if } N_{v} \text{ is odd},\\
-\Delta x\frac{q^{s}}{m^{s}} \sqrt{\frac{N_{v}-1}{2}} \mathcal{I}_{N_{v}-2} \mathbf{E}(t)^{\top} \mathbf{C}^{s}_{N_v-1}(t) \qquad &\text{if } N_{v} \text{ is even.}\\
\end{cases}
\end{alignat}
\end{theorem}
\begin{proof}
The number of particles of each plasma species $s$, $\mathcal{N}^{s}(t)$, in the SW formulation is given by 
\begin{equation}\label{mass-sw-definition}
    \mathcal{N}^{s}(t) \coloneqq \int_{0}^{\ell}\int_{\mathbb{R}} f^{s}_{\mathrm{sw}}(x, v, t) \mathrm{d}v \mathrm{d}x = \Delta x \alpha^{s} \sum_{n=0}^{N_{v}-1} \mathcal{I}_{n} \mathbf{1}^{\top} \mathbf{C}_{n}^{s}(t),
\end{equation}
where $\mathbf{1} \in \mathbb{R}^{N_{x}}$ is a column vector of all ones and $\mathcal{I}_{n}$ is defined in Eq.~\eqref{recursive-relation-I0}. The derivative in time of the particle number is 
\begin{alignat*}{3}
    \frac{\mathrm{d}\mathcal{N}^{s}}{\mathrm{d}t} &= \Delta x\alpha^{s} \sum_{n=0}^{N_{v}-1}\mathcal{I}_{n} \mathbf{1}^{\top} \frac{\mathrm{d} \mathbf{C}_{n}^{s}(t)}{\mathrm{d} t} \\
    &= \Delta x \sum_{n=0}^{N_{v}-1} \mathcal{I}_{n} \left[-\alpha^{s} \underbrace{\mathbf{1}^{\top}\mathbf{D}^{p} \mathbf{Q}^{s}_{n}(t)}_{=0}+
    \frac{q^{s}}{m^{s}} \mathbf{E}(t)^{\top} \left(\sqrt{\frac{n}{2}} \mathbf{C}^{s}_{n-1}(t) -\sqrt{\frac{n+1}{2}} \mathbf{C}_{n+1}^{s}(t) \right)\right]  \Hquad &&\text{from Eq. }\eqref{vlasov-ode}\\
    &= \Delta x\frac{q^{s}}{m^{s}} \mathbf{E}(t)^{\top} \sum_{n=0}^{N_{v}-1} \mathcal{I}_{n}  \left(\sqrt{\frac{n}{2}} \mathbf{C}^{s}_{n-1}(t) -\sqrt{\frac{n+1}{2}} \mathbf{C}_{n+1}^{s}(t) \right) \Hquad &&\text{from Eqns. }\eqref{anti-commutative}~\&~\eqref{central-finite-differencing}
\end{alignat*}
where $\mathbf{Q}_{n}^{s}(t) = \alpha^{s} \sqrt{\frac{n}{2}} \mathbf{C}_{n-1}^{s}(t) + \alpha^{s} \sqrt{\frac{n+1}{2}} \mathbf{C}_{n+1}^{s}(t) + u^{s} \mathbf{C}_{n}^{s}(t)$.
Since the summation over $n$ is telescopic and from the recurrence relation of $\mathcal{I}_{n}$ in Eq.~\eqref{recursive-relation-I0}, we obtain Eq.~\eqref{change-in-mass-sw}. 
\end{proof}

\begin{theorem}[SW formulation total momentum conservation]\label{change-in-momentum-thm} Let $f^{s}_{\mathrm{sw}}(x, v, t)$ be the solution of the semi-discrete Vlasov-Poisson equations~\eqref{vlasov-ode}-\eqref{posisson-ode-sw} solved via the SW formulation. Then, the total momentum, $\mathcal{P}(t)$, defined as
\begin{equation*}
    \mathcal{P}(t) \coloneqq \sum_{s} m^{s} \int_{0}^{\ell}\int_{\mathbb{R}} v f^{s}_{\mathrm{sw}}(x, v, t) \mathrm{d}v \mathrm{d}x,
\end{equation*}
is conserved in time if $N_{v}$ is even and $u^{s}=0$ for all $s$, where the drift rate is
\begin{equation}\label{change-in-momentum-sw}
\frac{\mathrm{d}\mathcal{P}}{\mathrm{d}t} = \begin{cases}
-\Delta x N_v\mathcal{I}_{N_v-1} \mathbf{E}(t)^{\top} \sum_{s} \alpha^{s} q^{s}\mathbf{C}_{N_v-1}^{s}(t)\qquad &\text{if } N_{v} \text{ is odd},\\
 -\Delta x \sqrt{\frac{N_v-1}{2}} \mathcal{I}_{N_v-2} \mathbf{E}(t)^{\top} \sum_{s} u^{s} q^{s} \mathbf{C}_{N_v-1}^{s}(t)\qquad&\text{if } N_{v} \text{ is even}.\\
\end{cases}
\end{equation}
\end{theorem}
\begin{proof}
The total momentum $\mathcal{P}(t)$ in the SW formulation is given by 
\begin{equation}\label{momentum-sw-definition}
    \mathcal{P}(t)\coloneqq \sum_{s} m^{s} \int_{0}^{\ell}\int_{\mathbb{R}} v f^{s}_{\mathrm{sw}}(x, v, t) \mathrm{d}v \mathrm{d}x =  \Delta x\sum_{s} m^{s} \alpha^{s} \sum_{n=0}^{N_v-1} \mathcal{I}_{n}^{1} \mathbf{1}^{\top} \mathbf{C}_{n}^{s}(t),
\end{equation}
where the integral $\mathcal{I}^{1}_{n}$ is defined as
\begin{equation*}
\mathcal{I}^{1}_{n} \coloneqq \int_{\mathbb{R}} v \psi_{n}(\xi^{s}) \mathrm{d}\xi^{s} = \begin{cases}
\alpha^{s}\sqrt{\frac{n+1}{2}} \mathcal{I}_{n+1} + \alpha^{s} \sqrt{\frac{n}{2}} \mathcal{I}_{n-1} \qquad &\text{if } n \geq 1\text{ and odd,}\\
u^{s}\mathcal{I}_{n} \qquad &\text{if } n\geq 0 \text{ and even.}
\end{cases}
\end{equation*}
The derivative in time of the total momentum is
\begin{alignat}{3}\label{change-in-momentum-sw-proof}
    \frac{\mathrm{d}\mathcal{P}}{\mathrm{d}t} &= \Delta x\sum_{s} m^{s} \alpha^{s} \sum_{n=0}^{N_v-1} \mathcal{I}_{n}^{1} \mathbf{1}^{\top} \frac{\mathrm{d} \mathbf{C}_{n}^{s}(t)}{\mathrm{d} t}   \\
    &= \Delta x \sum_{s} \sum_{n=0}^{N_{v}-1} \mathcal{I}_{n}^{1} \left[-m^{s} \alpha^{s} \underbrace{\mathbf{1}^{\top}\mathbf{D}^{p} \mathbf{Q}^{s}_{n}(t)}_{=0}+
    q^{s} \mathbf{E}(t)^{\top} \left(\sqrt{\frac{n}{2}} \mathbf{C}^{s}_{n-1}(t) -\sqrt{\frac{n+1}{2}} \mathbf{C}_{n+1}^{s}(t) \right)\right]  \Hquad &&\text{from Eq. }\eqref{vlasov-ode}\nonumber\\
    &= \Delta x \mathbf{E}(t)^{\top} \sum_{s} q^{s} \sum_{n=0}^{N_{v}-1} \mathcal{I}_{n}^{1}  \left(\sqrt{\frac{n}{2}} \mathbf{C}^{s}_{n-1}(t) -\sqrt{\frac{n+1}{2}} \mathbf{C}_{n+1}^{s}(t) \right) \Hquad &&\text{from Eqns. }\eqref{anti-commutative}~\&~\eqref{central-finite-differencing}.\nonumber
\end{alignat}
We divide the summation into even and odd indices and use the recursive relation of $\mathcal{I}_{n}$ from Eq.~\eqref{recursive-relation-I0}, such that the even indices result in
\begin{equation}\label{even-indicies}
    \sum_{\substack{n=0 \\ n\text{ is even}}}^{N_{v}-1} \mathcal{I}_{n}^{1} \left( \sqrt{\frac{n}{2}} \mathbf{C}_{n-1}^{s}(t)-\sqrt{\frac{n+1}{2}} \mathbf{C}_{n+1}^{s}(t) \right)  =  \begin{cases}
        0 \qquad &\text{if } N_{v} \text{ is odd,}\\
        -\sqrt{\frac{N_v-1}{2}} u^{s} \mathcal{I}_{N_v-2} \mathbf{C}_{N_v-1}(t) &\text{if } N_v \text{ is even,}
    \end{cases}
\end{equation}
and the odd indices result in 
\begin{align}\label{odd-indicies}
    \sum_{\substack{n=1 \\ n\text{ is odd}}}^{N_{v}-1} \mathcal{I}_{n}^{1} \left( \sqrt{\frac{n}{2}} \mathbf{C}_{n-1}^{s}(t)-\sqrt{\frac{n+1}{2}} \mathbf{C}_{n+1}^{s}(t) \right) &=  \sum_{\substack{n=1 \\ \text{n is odd}}}^{N_{v}-1} \alpha^{s} \mathcal{I}_{n-1} \mathbf{C}_{n-1}^{s}(t)\\
    &+\begin{cases}
        -(N_{v}-1)\mathcal{I}_{N_v-1} \alpha^{s} \mathbf{C}^{s}_{N_v-1}(t) \qquad &\text{if } N_{v} \text{ is odd,}\\
        0 &\text{if } N_v \text{ is even.} \nonumber
    \end{cases}
\end{align}
Thus, by inserting Eqns.~\eqref{even-indicies} and \eqref{odd-indicies} into Eq.~\eqref{change-in-momentum-sw-proof}, we get
\begin{align*}
    \frac{\mathrm{d}\mathcal{P}}{\mathrm{d}t} &= \Delta x \mathbf{E}(t)^{\top}\sum_{s} q^{s}\alpha^{s} \sum_{\substack{n=1 \\ n\text{ is odd}}}^{N_{v}-1} \mathcal{I}_{n-1} \mathbf{C}_{n-1}^{s}(t)\\
    &+ \begin{cases}
        -\Delta x (N_{v}-1) \mathcal{I}_{N_v-1}  \mathbf{E}(t)^{\top} \sum_{s} q^{s} \alpha^{s} \mathbf{C}^{s}_{N_v-1}(t) \qquad &\text{if } N_{v} \text{ is odd,}\\
        -\Delta x \sqrt{\frac{N_v-1}{2}} \mathcal{I}_{N_v-2}\mathbf{E}(t)^{\top} \sum_{s} q^{s} u^{s}\mathbf{C}^{s}_{N_v-1}(t) &\text{if } N_v \text{ is even.} \nonumber
    \end{cases}
\end{align*}
The first term in the expression above is equal to zero from the Poisson equation~\eqref{posisson-ode-sw}, the recursive relation of $\mathcal{I}_{n}$ in Eq.~\eqref{recursive-relation-I0}, and the anti-symmetry properties~\eqref{preserves-square-norms} of the centered finite differencing derivative operator $(\mathbb{R}^{N_{x}}, \mathbf{D}^{p})$ defined in Eq.~\eqref{central-finite-differencing}, i.e. 
\begin{equation*}
    \Delta x \mathbf{E}(t)^{\top}\sum_{s} q^{s}\alpha^{s} \sum_{\substack{n=1 \\ n\text{ is odd}}}^{N_{v}-1} \mathcal{I}_{n-1} \mathbf{C}_{n-1}^{s}(t) = \begin{cases}
      -\Delta x \mathcal{I}_{N_v -1} \mathbf{E}(t)^{\top} \sum_{s} q^s \alpha^{s} \mathbf{C}^{s}_{N_v -1}(t)&\text{if } N_v \text{ is odd,}\\
     \Delta x \mathbf{E}(t)^{\top} \mathbf{D}^{p} \mathbf{E}(t) = 0 &\text{if } N_v \text{ is even,}
    \end{cases}
\end{equation*}
which implies that Eq.~\eqref{change-in-momentum-sw} holds. 
\end{proof}

\begin{theorem}[SW formulation total energy conservation] Let $f^{s}_{\mathrm{sw}}(x, v, t)$ be the solution of the semi-discrete Vlasov-Poisson equations~\eqref{vlasov-ode}-\eqref{posisson-ode-sw} solved via the SW formulation. Consider the total energy $\mathcal{E}(t) \coloneqq \mathcal{E}_{\mathrm{kin}}(t) + \mathcal{E}_{\mathrm{pot}}(t)$, where 
\begin{equation*}
     \mathcal{E}_{\mathrm{kin}}(t) \coloneqq \frac{1}{2} \sum_{s} m^{s} \int_{0}^{\ell} \int_{\mathbb{R}} v^{2} f_{\mathrm{sw}}^{s}(x, v, t) \mathrm{d}v \mathrm{d}x \qquad \mathrm{and} \qquad \mathcal{E}_{\mathrm{pot}}(t) \coloneqq \frac{1}{2} \int_{0}^{\ell}E(x, t)^2 \mathrm{d}x.
\end{equation*}
The total energy $\mathcal{E}(t)$ is conserved in time if $N_{v}$ is odd and $u^{s}=0$ for all $s$, such that the drift rate is
\begin{equation}\label{change-in-total-energy-sw}
\frac{\mathrm{d}\mathcal{E}}{\mathrm{d}t} = \begin{cases}
-\Delta xN_{v}\mathcal{I}_{N_v-1} \mathbf{E}(t)^{\top} \sum_{s} u^{s} \alpha^{s}q^{s} \mathbf{C}^{s}_{N_v-1}(t)\qquad &\text{if } N_{v} \text{ is odd},\\
-\Delta x \sqrt{\frac{N_v-1}{2}} \mathcal{I}_{N_v-2} \mathbf{E}(t)^{\top} \sum_{s} q^{s} \left(\frac{(2N_v -1) (\alpha^{s})^2 + (u^s)^2}{2}+ \frac{q^{s}}{m^{s}} (\mathbf{D}^{p})^{\dagger}\mathbf{E}(t) \odot \right) \mathbf{C}^{s}_{N_v-1}(t) &\text{if } N_{v} \text{ is even,}
\end{cases}
\end{equation}
where $(\mathbf{D}^{p})^{\dagger} \in \mathbb{R}^{N_{x} \times N_{x}}$ is the pseudoinverse of the finite difference operator $\mathbf{D}^{p}$ defined in Eq.~\eqref{central-finite-differencing}.
\end{theorem}
\begin{proof}
The kinetic energy $\mathcal{E}_{\mathrm{kin}}(t)$ in the SW formulation is given by 
\begin{equation}\label{kinetic-energy-sw-definition}
    \mathcal{E}_{\mathrm{kin}}(t) \coloneqq \frac{1}{2} \sum_{s} m^{s} \int_{0}^{\ell} \int_{\mathbb{R}} v^{2} f_{\mathrm{sw}}^{s}(x, v, t) \mathrm{d}v \mathrm{d}x = \frac{\Delta x}{2}\sum_{s} \alpha^{s} m^{s} \sum_{n=0}^{N_v-1} \mathcal{I}_{n}^{2} \mathbf{1}^{\top}  \mathbf{C}_{n}^{s}(t), 
\end{equation}
where the integral $\mathcal{I}^{2}_{n}$ is defined as 
\begin{alignat*}{3}
    \mathcal{I}_{n}^{2} \coloneqq \begin{cases}
        2\alpha^s u^s \left( \sqrt{\frac{n+1}{2}} \mathcal{I}_{n+1} + \sqrt{\frac{n}{2}} \mathcal{I}_{n-1}\right) = 2u^{s} \mathcal{I}_{n}^{1} \qquad &\text{if } n\geq 1 \text{ and odd,}\\
        (\alpha^{s})^2 \left(\sqrt{\frac{(n+1)(n+2)}{4}} \mathcal{I}_{n+2} + \left( \frac{2n+1}{2} + \left(\frac{u^s}{\alpha^s}\right)^2\right) \mathcal{I}_{n} + \sqrt{\frac{n(n-1)}{4}} \mathcal{I}_{n-2}\right)\qquad &\text{if }n\geq 0 \text{ and even.}
    \end{cases}
\end{alignat*}
The derivative in time of the kinetic energy is
\begin{alignat}{3}
    \frac{\mathrm{d}\mathcal{E}_{\mathrm{kin}}}{\mathrm{d}t} &= \frac{\Delta x}{2}\sum_{s} \alpha^{s} m^{s} \sum_{n=0}^{N_v-1} \mathcal{I}_{n}^{2} \mathbf{1}^{\top} \frac{\mathrm{d} \mathbf{C}_{n}^{s}(t)}{\mathrm{d} t}\nonumber \\
    &= \frac{\Delta x}{2} \sum_{s} \sum_{n=0}^{N_{v}-1} \mathcal{I}_{n}^{2} \left[-m^{s} \alpha^{s} \underbrace{\mathbf{1}^{\top}\mathbf{D}^{p} \mathbf{Q}^{s}_{n}(t)}_{=0}+
    q^{s} \mathbf{E}(t)^{\top} \left(\sqrt{\frac{n}{2}} \mathbf{C}^{s}_{n-1}(t) -\sqrt{\frac{n+1}{2}} \mathbf{C}_{n+1}^{s}(t) \right)\right]  \Hquad &&\text{from Eq. }\eqref{vlasov-ode}\nonumber\\
    &= \frac{\Delta x}{2} \mathbf{E}(t)^{\top} \sum_{s} q^{s} \sum_{n=0}^{N_{v}-1} \mathcal{I}_{n}^{2}  \left(\sqrt{\frac{n}{2}} \mathbf{C}^{s}_{n-1}(t) -\sqrt{\frac{n+1}{2}} \mathbf{C}_{n+1}^{s}(t) \right) \Hquad &&\text{from Eqns. }\eqref{anti-commutative}~\&~\eqref{central-finite-differencing} \nonumber\\
    &= \Delta x \mathbf{E}(t)^{\top}\mathbf{J}(t) \label{kinetic-energy-drift}\\
&+\begin{cases}
    -\Delta x N_{v}\mathcal{I}_{N_v-1} \mathbf{E}(t)^{\top} \sum_{s} u^{s} \alpha^{s} q^{s} \mathbf{C}^{s}_{N_v-1}(t) \qquad &\text{if } N_{v} \text{ is odd,}\\
    -\frac{\Delta x}{2} \sqrt{\frac{N_v-1}{2}}\mathcal{I}_{N_v-2}\mathbf{E}(t)^{\top} \sum_{s} q^{s}\left((2N_{v}-1)(\alpha^{s})^2 + (u^s)^2 \right) \mathbf{C}_{N_v-1}^{s}(t)&\text{if } N_v \text{ is even,} 
    \end{cases} \nonumber
\end{alignat}
where $\mathbf{J}(t) = \sum_{s} q^{s} \alpha^{s} \sum_{n=0}^{N_v-1} \mathcal{I}_{n}  \mathbf{Q}_{n}^{s}(t)$ is the discretized electric current.
The potential energy $\mathcal{E}_{\mathrm{pot}}(t)$ in the SW formulation is given by
\begin{equation}\label{pot-energy-sw-definition}
     \mathcal{E}_{\mathrm{pot}}(t) = \frac{1}{2} \int_{0}^{\ell} E(x, t)^2 \mathrm{d}x = \frac{\Delta x}{2} \mathbf{E}(t)^{\top}\mathbf{E}(t),
\end{equation}
and its time derivative is
\begin{equation}\label{change-in-potential-sw-1}
    \frac{\mathrm{d}\mathcal{E}_{\mathrm{pot}}}{\mathrm{d}t} =\Delta x \mathbf{E}(t)^{\top} \frac{\mathrm{d} \mathbf{E}(t)}{\mathrm{d}t }.
\end{equation}
By taking the time derivative of the semi-discrete Poisson equation~\eqref{f-poisson-pde}, we get 
\begin{alignat*}{3}
    \mathbf{D}^{p}\frac{\mathrm{d} \mathbf{E}(t)}{\mathrm{d}t } &= \sum_{s} q^{s} \alpha^{s} \sum_{n=0}^{N_v-1} \mathcal{I}_{n} \frac{\mathrm{d}\mathbf{C}_{n}^{s}(t)}{\mathrm{d}t}\\
    &= -\sum_{s} q^{s} \alpha^{s} \sum_{n=0}^{N_v-1} \mathcal{I}_{n} \left( \mathbf{D}^{p} \mathbf{Q}_{n}^{s}(t) + \frac{q^{s}}{m^{s}\alpha^{s}}\mathbf{E}(t) \odot \left(\sqrt{\frac{n}{2}} \mathbf{C}_{n-1}^{s}(t) -\sqrt{\frac{n+1}{2}} \mathbf{C}_{n+1}^{s}(t)\right)\right) \Hquad &&\text{from Eq. }\eqref{vlasov-ode}\\
    &= -\mathbf{D}^{p}\left(\sum_{s} q^{s} \alpha^{s} \sum_{n=0}^{N_v-1} \mathcal{I}_{n}  \mathbf{Q}_{n}^{s}(t) \right)  + \begin{cases}
0 \qquad &\text{if } N_{v} \text{ is odd},\\
-\sum_{s} \frac{(q^{s})^2}{m^{s}} \sqrt{\frac{N_{v}-1}{2}} \mathcal{I}_{N_{v}-2} \mathbf{E}(t) \odot \mathbf{C}^{s}_{N_v-1}(t) \qquad &\text{if } N_{v} \text{ is even.}\\
\end{cases} \\
    &= -\mathbf{D}^{p}\mathbf{J}(t) + \begin{cases}
0 \qquad &\text{if } N_{v} \text{ is odd},\\
-\sum_{s} \frac{(q^{s})^2}{m^{s}} \sqrt{\frac{N_{v}-1}{2}} \mathcal{I}_{N_{v}-2} \mathbf{E}(t) \odot \mathbf{C}^{s}_{N_v-1}(t) \qquad &\text{if } N_{v} \text{ is even,}\\
\end{cases}
\end{alignat*}
where $\odot$ denotes the Hadamard (element-wise) product. Thus, 
\begin{equation*}
    \frac{\mathrm{d} \mathbf{E}(t)}{\mathrm{d}t } = -\mathbf{J}(t) + c\mathbf{1} + \begin{cases}
0 \qquad &\text{if } N_{v} \text{ is odd},\\
-(\mathbf{D}^{p})^{\dagger} \sum_{s} \frac{(q^{s})^2}{m^{s}} \sqrt{\frac{N_{v}-1}{2}} \mathcal{I}_{N_{v}-2} \mathbf{E}(t) \odot \mathbf{C}^{s}_{N_v-1}(t) \qquad &\text{if } N_{v} \text{ is even.}\\
\end{cases}
\end{equation*}
where $c\in \mathbb{R}$ is an arbitrary constant and $(\mathbf{D}^{p})^{\dagger} \in \mathbb{R}^{N_{x} \times N_{x}}$ is the pseudoinverse of the finite difference operator $\mathbf{D}^{p}$ defined in Eq.~\eqref{central-finite-differencing}, and Eq.~\eqref{change-in-potential-sw-1} becomes
\begin{equation}\label{change-in-potential-sw-2}
    \frac{\mathrm{d}\mathcal{E}_{\mathrm{pot}}}{\mathrm{d}t} = \Delta x \mathbf{E}(t)^{\top} \frac{\mathrm{d} \mathbf{E}(t)}{\mathrm{d}t } = -\Delta x \mathbf{E}(t)^{\top} \mathbf{J}(t) + \begin{cases}
0 \qquad &\text{if } N_{v} \text{ is odd},\\
-\Delta x \mathbf{E}(t)^{\top}(\mathbf{D}^{p})^{\dagger}\sum_{s} \frac{(q^{s})^2}{m^{s}} \sqrt{\frac{N_{v}-1}{2}} \mathcal{I}_{N_{v}-2} \mathbf{E}(t) \odot \mathbf{C}^{s}_{N_v-1}(t) \qquad &\text{if } N_{v} \text{ is even.}\\
\end{cases}
\end{equation}
Summing up Eqns.~\eqref{kinetic-energy-drift} and \eqref{change-in-potential-sw-2} results in Eq.~\eqref{change-in-total-energy-sw}.
\end{proof}
\subsection{Symmetrically Weighted Square-Root Formulation Conservation Properties}\label{sec:SW-sqrt-conservation-properties}
We show below that the SW square-root formulation conserves the number of particles without restrictions on the number of velocity spectral terms $N_v$ or velocity coordinate scaling and shifting parameters. However, the total momentum and energy are not conserved in time.
\begin{theorem}[SW square-root formulation particle number conservation]\label{thm:mass-conservation-sw-sqrt}
Let $f^{s}_{\mathrm{swsr}}(x, v, t)$ be the solution of the semi-discrete Vlasov-Poisson equations~\eqref{vlasov-ode} and~\eqref{poisson-ode-sw-sqrt} solved via the SW square-root formulation. Then, the number of particles of each plasma species $s$, i.e. 
\begin{equation*}
    \mathcal{N}^{s}(t) \coloneqq \int_{0}^{\ell}\int_{\mathbb{R}} f^{s}_{\mathrm{swsr}}(x, v, t) \mathrm{d}v \mathrm{d}x,
\end{equation*}
is conserved in time, such that $\frac{\mathrm{d}\mathcal{N}^{s}}{\mathrm{d}t}=0$.
\end{theorem}
\begin{proof}
The number of particles of each plasma species $s$, $\mathcal{N}^{s}(t)$, in the SW square-root formulation is given by
\begin{equation}\label{mass-sw-sqrt-definition}
    \mathcal{N}^{s}(t) \coloneqq \int_{0}^{\ell}\int_{\mathbb{R}} f^{s}_{\mathrm{swsr}}(x, v, t) \mathrm{d}v \mathrm{d}x =  \Delta x \alpha^{s} \sum_{n=0}^{N_v-1} \mathbf{C}_{n}^{s}(t)^{\top} \mathbf{C}_{n}^{s}(t) = \Delta x \alpha^{s} \mathbf{\Psi}^{s}(t)^{\top} \mathbf{\Psi}^{s}(t).
\end{equation}
From the anti-symmetric property of the advection operator $\mathbf{A}^{s}(t) \in \mathbb{R}^{N_{v}N_{x} \times N_{v}N_{x}}$ in semi-discrete form, see Eq.~\eqref{vlasov-ode} and Proposition~\ref{thm:anti-symmetric-invariants}, the quadratic term $\mathbf{\Psi}^{s}(t)^{\top} \mathbf{\Psi}^{s}(t)$ is an invariant of the semi-discrete equations. Thus, the particle number $\mathcal{N}^{s}(t)$ of species $s$ is conserved in the SW square-root formulation.
\end{proof}

\begin{theorem}[SW square-root formulation total momentum conservation]\label{thm:momentum-conservation-sw-sqrt} Let $f^{s}_{\mathrm{swsr}}(x, v, t)$ be the solution of the semi-discrete Vlasov-Poisson equations~\eqref{vlasov-ode} and~\eqref{poisson-ode-sw-sqrt} solved via the SW square-root formulation. Then, the total momentum, $\mathcal{P}(t)$, which is defined as
\begin{equation*}
    \mathcal{P}(t) \coloneqq \sum_{s} m^{s} \int_{0}^{\ell}\int_{\mathbb{R}} v f^{s}_{\mathrm{swsr}}(x, v, t) \mathrm{d}v \mathrm{d}x,
\end{equation*}
is not conserved in time, where the drift rate is 
\begin{equation}\label{drift-in-momentum-sw-sqrt}
    \frac{\mathrm{d} \mathcal{P}}{\mathrm{d} t} = -\Delta x N_{v} \mathbf{E}(t)^{\top}  \sum_{s} q^{s} \alpha^{s}\mathbf{C}_{N_{v} -1}^{s}(t)^2
\end{equation}
\end{theorem}
\begin{proof}
The total momentum $\mathcal{P}(t)$ in the SW square-root formulation is given by 
\begin{align}\label{momentum-sw-sqrt-definition}
    \mathcal{P}(t) &\coloneqq \sum_{s} m^{s}\int_{0}^{\ell}\int_{\mathbb{R}} v f^{s}_{\mathrm{swsr}}(x, v, t)\mathrm{d}v \mathrm{d} x\\
    &=\Delta x \sum_{s} m^{s} \alpha^{s} 
    \sum_{n=0}^{N_v-1} \left(u^{s} \mathbf{C}_{n}^{s}(t)^{\top}\mathbf{C}_{n}^{s}(t) +  2\alpha^{s} \sqrt{\frac{n}{2}} \mathbf{C}_{n-1}^{s}(t)^{\top} \mathbf{C}_{n}^{s}(t)\right).
    \nonumber
\end{align}
The time derivative of the total momentum is 
\begin{align*}
    \frac{\mathrm{d}\mathcal{P}}{\mathrm{d}t}&=\Delta x \sum_{s} m^{s} \alpha^{s} \sum_{n=0}^{N_v-1} \frac{\mathrm{d}}{\mathrm{d}t}\left(u^{s} \mathbf{C}_{n}^{s}(t)^{\top}\mathbf{C}_{n}^{s}(t) +  2\alpha^{s} \sqrt{\frac{n}{2}} \mathbf{C}_{n-1}^{s}(t)^{\top} \mathbf{C}_{n}^{s}(t)\right)\\
    &=\Delta x \sum_{s} m^{s} \alpha^{s} \left(2u^{s} \underbrace{\frac{\mathrm{d}}{\mathrm{d}t} \mathbf{\Psi}^{s}(t)^{\top}\mathbf{\Psi}^{s}(t)}_{=0 \text{ from Proposition}~\ref{thm:anti-symmetric-invariants}} + 2 \alpha^{s} \sum_{n=0}^{N_v-1}  \sqrt{\frac{n}{2}} \left(\mathbf{C}_{n}^{s}(t)^{\top}\frac{\mathrm{d}}{\mathrm{d}t}\mathbf{C}_{n-1}^{s}(t)  + \mathbf{C}_{n-1}^{s}(t)^{\top} \frac{\mathrm{d}}{\mathrm{d}t}\mathbf{C}_{n}^{s}(t) \right)\right).
\end{align*}
From the semi-discrete form of the Vlasov equation~\eqref{vlasov-ode}, we get that 
\begin{align*}
    \sum_{n=0}^{N_v -1}  \sqrt{\frac{n}{2}} \mathbf{C}_{n}^{s}(t)^{\top}\frac{\mathrm{d}}{\mathrm{d}t}\mathbf{C}_{n-1}^{s}(t)  &= \sum_{n=0}^{N_v -1}  -\sqrt{\frac{n}{2}} \mathbf{C}_{n}^{s}(t)^{\top} \mathbf{D}^{p} \left(\alpha^{s} \sqrt{\frac{n-1}{2}} \mathbf{C}_{n-2}^{s}(t) + \alpha^{s} \sqrt{\frac{n}{2}}\mathbf{C}_{n}^{s}(t)  + u^{s} \mathbf{C}_{n-1}^{s}(t)\right)  \\
    &+ \sum_{n=0}^{N_v-1} \frac{q^{s}}{m^s \alpha^{s}} \mathbf{E}(t)^{\top} \left(\frac{\sqrt{n(n-1)}}{2} \mathbf{C}_{n-2}^{s}(t) \odot \mathbf{C}_{n}^{s}(t) - \frac{n}{2} \mathbf{C}_{n}^{s}(t)^2 \right), \\
    \sum_{n=0}^{N_v -1}  \sqrt{\frac{n}{2}} \mathbf{C}_{n-1}^{s}(t)^{\top}\frac{\mathrm{d}}{\mathrm{d}t}\mathbf{C}_{n}^{s}(t) &= \sum_{n=0}^{N_v -1}  -\sqrt{\frac{n}{2}}  \mathbf{C}_{n-1}^{s}(t)^{\top} \mathbf{D}^{p} \left(\alpha^{s} \sqrt{\frac{n}{2}} \mathbf{C}_{n-1}^{s}(t) + \alpha^{s} \sqrt{\frac{n+1}{2}}\mathbf{C}_{n+1}^{s}(t)  + u^{s} \mathbf{C}_{n}^{s}(t)\right) \\
    &+ \sum_{n=0}^{N_v -1} \frac{q^{s}}{m^s \alpha^{s}} \mathbf{E}(t)^{\top} \left(\frac{n}{2} \mathbf{C}_{n-1}^{s}(t)^2  -\frac{\sqrt{n(n+1)}}{2} \mathbf{C}_{n-1}^{s}(t) \odot \mathbf{C}_{n+1}^{s}(t) \right).
\end{align*}
The telescoping sum results in 
\begin{align*}
    \sum_{n=0}^{N_v-1} \sqrt{\frac{n}{2}} \left(\mathbf{C}_{n}^{s}(t)^{\top}\frac{\mathrm{d}}{\mathrm{d}t}\mathbf{C}_{n-1}^{s}(t) + \mathbf{C}_{n-1}^{s}(t)^{\top} \frac{\mathrm{d}}{\mathrm{d}t}\mathbf{C}_{n}^{s}(t) \right) &=  \frac{q^{s}}{2m^s \alpha^{s}} \mathbf{E}(t)^{\top} \left( \left(\sum_{n=0}^{{N_v -1}} \mathbf{C}_{n}^{s}(t)^2 \right)- N_{v}\mathbf{C}_{N_{v}-1}^{s}(t)^2 \right).
\end{align*}
Thus, 
\begin{alignat*}{3}
\frac{\mathrm{d}\mathcal{P}}{\mathrm{d}t} &= \Delta x \mathbf{E}(t)^{\top} \underbrace{\sum_{s} q^{s} \alpha^{s} \sum_{n=0}^{N_v-1}\mathbf{C}_{n}^{s}(t)^2}_{=\mathbf{D}^{p} \mathbf{E}(t) \text{ from Eq.}~\eqref{poisson-ode-sw-sqrt}} -\Delta x N_v \mathbf{E}(t)^{\top}  \sum_{s} q^{s} \alpha^{s}\mathbf{C}_{N_{v} -1}^{s}(t)^2\\
&= -\Delta x N_v\mathbf{E}(t)^{\top}  \sum_{s} q^{s} \alpha^{s}\mathbf{C}_{N_{v} -1}^{s}(t)^2 \qquad &\text{from Eqns.}~\eqref{preserves-square-norms}~\&~\eqref{central-finite-differencing}.
\end{alignat*}
\end{proof}

\begin{theorem}[SW square-root formulation total energy conservation] Let $f^{s}_{\mathrm{swsr}}(x, v, t)$ be the solution of the semi-discrete Vlasov-Poisson equations~\eqref{vlasov-ode} and \eqref{poisson-ode-sw-sqrt} solved via the SW square-root formulation. Consider the total energy $\mathcal{E}(t) \coloneqq \mathcal{E}_{\mathrm{kin}}(t) + \mathcal{E}_{\mathrm{pot}}(t)$, where 
\begin{equation*}
     \mathcal{E}_{\mathrm{kin}}(t) \coloneqq \frac{1}{2} \sum_{s} m^{s} \int_{0}^{\ell} \int_{\mathbb{R}} v^{2} f_{\mathrm{swsr}}^{s}(x, v, t) \mathrm{d}v \mathrm{d}x \qquad \mathrm{and} \qquad \mathcal{E}_{\mathrm{pot}}(t) \coloneqq \frac{1}{2} \int_{0}^{\ell}E(x, t)^2 \mathrm{d}x.
\end{equation*}
The total energy $\mathcal{E}(t)$ is not conserved in time, where the drift rate is 
\begin{align}
    \frac{\mathrm{d}\mathcal{E}}{\mathrm{d}t} =&-\Delta x\frac{N_{v}}{2}\sqrt{\frac{N_{v}-1}{2}} \sum_{s} (\alpha^{s})^2 \left(m^{s}(\alpha^{s})^2\mathbf{C}_{N_v-1}^{s}(t)^{\top} \mathbf{D}^{p} \mathbf{C}_{N_v-2}^{s}(t) +q^{s}\mathbf{E}(t)^{\top} \left(\mathbf{C}_{N_v-2}^{s}(t) \odot \mathbf{C}_{N_v-1}^{s}(t)\right)\right)\nonumber \\
    &-\Delta x N_{v}\mathbf{E}(t)^{\top} \sum_{s} q^{s}u^{s}\alpha^{s} \mathbf{C}^{s}_{N_{v}-1}(t)^2 + \Delta x\mathbf{E}(t)^{\top} \sum_{s}\sum_{n=0}^{N_v-1} q^{s} \alpha^{s}\mathbf{C}_{n}^{s}(t) \odot \mathbf{Q}_{n}^{s}(t) \label{change-in-total-energy-sw-sqrt}\\
    &-2\Delta x \mathbf{E}(t)^{\top} (\mathbf{D}^{p})^{\dagger}\sum_{s} q^{s} \alpha^{s} \sum_{n=0}^{N_v-1} \mathbf{C}_{n}^{s}(t) \odot \mathbf{D}^{p}\mathbf{Q}_{n}^{s}(t). \nonumber
\end{align}
\end{theorem}
\begin{proof}
The kinetic energy $\mathcal{E}_{\mathrm{kin}}(t)$ in the SW square-root formulation is given by
\begin{align}\label{kin-energy-sw-sqrt-definition}
    \mathcal{E}_{\mathrm{kin}}(t) &\coloneqq \frac{1}{2} \sum_{s} m^{s} \int_{0}^{\ell} \int_{\mathbb{R}} v^{2} f^{s}_{\mathrm{swsr}}(x, v, t)\mathrm{d}v\mathrm{d}x \\
    &= \Delta x\sum_{s} m^{s}\alpha^{s} \sum_{n=0}^{N_v-1} \frac{1}{2} \left((u^{s})^2 + (\alpha^s)^2 n + \frac{(\alpha^s)^2}{2}\right) \left(\mathbf{C}_{n}^{s}(t)^{\top}\mathbf{C}_{n}^{s}(t)\right) \nonumber\\
    &+ \sum_{n=0}^{N_v-1} 2 u^{s}\alpha^{s} \sqrt{\frac{n}{2}}\left(\mathbf{C}_{n}^{s}(t)^{\top}  \mathbf{C}_{n-1}^{s}(t)\right)+\frac{(\alpha^{s})^2}{2} \sum_{n=0}^{N_v-1} \sqrt{n(n-1)}  \left(\mathbf{C}_{n}^{s}(t)^{\top}  \mathbf{C}_{n-2}^s(t)\right). \nonumber
\end{align}
The derivative in time of the kinetic energy is
\begin{alignat*}{3}
    \frac{\mathrm{d}\mathcal{E}_{\mathrm{kin}}}{\mathrm{d}t}  &= \Delta x\sum_{s} m^{s}\alpha^{s} \sum_{n=0}^{N_v-1} \frac{1}{2} \left((u^{s})^2 + (\alpha^s)^2 n + \frac{(\alpha^s)^2}{2}\right) \frac{\mathrm{d}}{\mathrm{d}t}\left(\mathbf{C}_{n}^{s}(t)^{\top}\mathbf{C}_{n}^{s}(t)\right) \nonumber\\
    &+ \sum_{n=0}^{N_v-1} 2 u^{s}\alpha^{s} \sqrt{\frac{n}{2}} \frac{\mathrm{d}}{\mathrm{d}t}\left(\mathbf{C}_{n}^{s}(t)^{\top}  \mathbf{C}_{n-1}^{s}(t)\right)+\frac{(\alpha^{s})^2}{2} \sum_{n=0}^{N_v-1} \sqrt{n(n-1)}  \frac{\mathrm{d}}{\mathrm{d}t}\left(\mathbf{C}_{n}^{s}(t)^{\top}  \mathbf{C}_{n-2}^s(t)\right)\Hquad  && \text{from Eq.}~\eqref{identity-2}\\
    &= \Delta x\sum_{s} m^{s}\alpha^{s} \sum_{n=0}^{N_v-1} \frac{n(\alpha^s)^2}{2}   \frac{\mathrm{d}}{\mathrm{d}t}\left(\mathbf{C}_{n}^{s}(t)^{\top}\mathbf{C}_{n}^{s}(t)\right) +\underbrace{\left( \frac{(u^{s})^2}{2} + \frac{(\alpha^{s})^2}{4} \right) \frac{\mathrm{d}}{\mathrm{d}t} \mathbf{\Psi}^{s}(t)^{\top} \mathbf{\Psi}^{s}(t)}_{=0} && \text{from Proposition}~\ref{thm:anti-symmetric-invariants}\\
    &+ \sum_{n=0}^{N_v-1} 2 u^{s}\alpha^{s} \sqrt{\frac{n}{2}} \frac{\mathrm{d}}{\mathrm{d}t}\left(\mathbf{C}_{n}^{s}(t)^{\top}  \mathbf{C}_{n-1}^{s}(t)\right)+\frac{(\alpha^{s})^2}{2} \sum_{n=0}^{N_v-1} \sqrt{n(n-1)}  \frac{\mathrm{d}}{\mathrm{d}t}\left(\mathbf{C}_{n}^{s}(t)^{\top}  \mathbf{C}_{n-2}^s(t)\right).
\end{alignat*}
We expand each term in the expression above, such that 
\begin{align*}
    \sum_{n=0}^{N_v-1} \frac{n (\alpha^s)^2}{2} \frac{\mathrm{d}}{\mathrm{d}t}\left(\mathbf{C}_{n}^{s}(t)^{\top}\mathbf{C}_{n}^{s}(t) \right) &= \sum_{n=0}^{N_v-1} -(\alpha^{s})^3 \sqrt{\frac{n}{2}} \mathbf{C}_{n}^{s}(t)^{\top} \mathbf{D}^{p} \mathbf{C}^{s}_{n-1}+ \frac{q^{s}\alpha^{s}}{m^s}\sqrt{\frac{n}{2}}\mathbf{E}^{\top}\left( \mathbf{C}_{n-1}^{s}(t) \odot \mathbf{C}_{n}^{s}(t) \right), \\
    \sum_{n=0}^{N_v-1} 2 u^{s}\alpha^{s} \sqrt{\frac{n}{2}} \frac{\mathrm{d}}{\mathrm{d}t}\left(\mathbf{C}_{n}^{s}(t)^{\top}  \mathbf{C}_{n-1}^{s}(t)\right) &=  \frac{q^{s}u^{s}}{m^s}  \mathbf{E}(t)^{\top} \left(\left(\sum_{n=0}^{N_v-1} \mathbf{C}_{n}^{s}(t)^{2}\right) -N_{v} \mathbf{C}_{N_{v} -1}^{s}(t)^2 \right), \\
    \sum_{n=0}^{N_v-1} \frac{(\alpha^{s})^2}{2}  \sqrt{n(n-1)}\frac{\mathrm{d}}{\mathrm{d}t }\left(\mathbf{C}_{n}^{s}(t)^{\top}  \mathbf{C}_{n-2}^s(t)\right) &= \sum_{n=0}^{N_v-1} -(\alpha^{s})^3 \sqrt{\frac{n-1}{2}} \mathbf{C}^{s}_{n-2}(t)^{\top} \mathbf{D}^{p} \mathbf{C}_{n-1}^{s}(t) \\
    &+\sum_{n=0}^{N_v-1}  \frac{q^{s}\alpha^{s}}{m^s} \sqrt{\frac{n-1}{2}} \mathbf{E}(t)^{\top} \left(\mathbf{C}_{n-1}^{s}(t) \odot \mathbf{C}_{n-2}^{s}(t)\right)\\
    &-\frac{q^{s}\alpha^{s}}{2m^s}  (N_v-2)\sqrt{\frac{N_v-1}{2}}\mathbf{E}(t)^{\top} \left(\mathbf{C}_{N_v-2}^{s}(t) \odot \mathbf{C}_{N_v-1}^{s}(t)\right)\\
    &-\frac{(\alpha^{s})^3}{2} (N_v-2)\sqrt{\frac{N_v -1}{2}}\mathbf{C}_{N_v-1}^{s}(t)^{\top}\mathbf{D}^{p} \mathbf{C}_{N_v-2}^{s}(t).
\end{align*}
From the anti-commutative property of the anti-symmetric operator $(\mathbb{R}^{N_{x}}, \mathbf{D}^{p})$, see Eqns.~\eqref{anti-commutative} and~\eqref{central-finite-differencing}, we end up with 
\begin{align*}
    \frac{\mathrm{d}\mathcal{E}_{\mathrm{kin}}}{\mathrm{d}t} = &-\Delta x\frac{N_{v}}{2}\sqrt{\frac{N_{v}-1}{2}} \sum_{s} (\alpha^{s})^2 \left(m^{s}(\alpha^{s})^2\mathbf{C}_{N_v-1}^{s}(t)^{\top} \mathbf{D}^{p} \mathbf{C}_{N_v-2}^{s}(t) +q^{s}\mathbf{E}(t)^{\top} \left(\mathbf{C}_{N_v-2}^{s}(t) \odot \mathbf{C}_{N_v-1}^{s}(t)\right)\right)\\
    &-\Delta x N_{v}\mathbf{E}(t)^{\top} \sum_{s} q^{s}u^{s}\alpha^{s} \mathbf{C}^{s}_{N_{v}-1}(t)^2 + \Delta x\mathbf{E}(t)^{\top} \sum_{s}\sum_{n=0}^{N_v-1} q^{s} \alpha^{s}\mathbf{C}_{n}^{s}(t) \odot \mathbf{Q}_{n}^{s}(t).
\end{align*}
The potential energy $\mathcal{E}_{\mathrm{pot}}(t)$ is given by 
\begin{equation}\label{pot-energy-sw-sqrt-definition}
    \frac{\mathrm{d}\mathcal{E}_{\mathrm{pot}}}{\mathrm{d}t} \coloneqq \frac{1}{2} \int_{0}^{\ell} E(x, t)^2 \mathrm{d}x = \frac{\Delta x}{2} \mathbf{E}(t)^{\top} \mathbf{E}(t).
\end{equation}
The derivative in time of the potential energy is
\begin{equation*}
     \frac{\mathrm{d}\mathcal{E}_{\mathrm{pot}}}{\mathrm{d}t} =\Delta x \mathbf{E}(t)^{\top} \frac{\mathrm{d} \mathbf{E}(t)}{\mathrm{d}t }.
\end{equation*}
By taking the time derivative of the semi-discrete Poisson equation~\eqref{poisson-ode-sw-sqrt}, we get 
\begin{alignat*}{3}
    \mathbf{D}^{p}\frac{\mathrm{d} \mathbf{E}(t)}{\mathrm{d}t } &= \sum_{s} q^{s} \alpha^{s} \sum_{n=0}^{N_v-1} \frac{\mathrm{d}\mathbf{C}_{n}^{s}(t)^{2}}{\mathrm{d}t}\\
    &= -2\sum_{s} q^{s} \alpha^{s} \sum_{n=0}^{N_v-1} \mathbf{C}_{n}^{s}(t) \odot \left( \mathbf{D}^{p} \mathbf{Q}_{n}^{s}(t) + \frac{q^{s}}{m^{s}\alpha^{s}}\mathbf{E}(t)^{\top} \left(\sqrt{\frac{n}{2}} \mathbf{C}_{n-1}^{s}(t)\ -\sqrt{\frac{n+1}{2}} \mathbf{C}_{n+1}^{s}(t)\right)\right) \\
    &= -2\sum_{s} q^{s} \alpha^{s} \sum_{n=0}^{N_v-1} \mathbf{C}_{n}^{s}(t)  
    \odot \mathbf{D}^{p}\mathbf{Q}_{n}^{s}(t).
\end{alignat*}
Thus, 
\begin{equation}\label{ampere-derived}
    \frac{\mathrm{d} \mathbf{E}(t)}{\mathrm{d}t } = -2 (\mathbf{D}^{p})^{\dagger}\sum_{s} q^{s} \alpha^{s} \sum_{n=0}^{N_v-1} \mathbf{C}_{n}^{s}(t)  
    \odot \mathbf{D}^{p}\mathbf{Q}_{n}^{s}(t),
\end{equation}
resulting in Eq.~\eqref{change-in-total-energy-sw-sqrt}.
\end{proof}

\begin{remark}
    The finite difference operator $\mathbf{D}^{p} \in \mathbb{R}^{N_{x} \times N_{x}}$ defined in Eq.~\eqref{central-finite-differencing} does not preserve the product rule, i.e. $\mathbf{D}^{p} (\mathbf{y} \odot \mathbf{z}) \neq (\mathbf{D}^{p} \mathbf{y}) \odot \mathbf{z}  + \mathbf{y} \odot (\mathbf{D}^{p} \mathbf{z})$. Otherwise, the last two terms in the energy drift in Eq.~\eqref{change-in-total-energy-sw-sqrt} cancel out and the energy drift is a function of the electric field and the last two spectral coefficient terms. On the other hand, any spectral discretization preserves the product rule, such as the Fourier spectral discretization, see~\ref{sec:Appendix-A}.
\end{remark}

\section{Numerical Results} \label{sec:numerical-results}
We numerically investigate the accuracy, convergence, and conservation properties of the SW and SW square-root formulations. The numerical setup for the test cases is described in section~\ref{sec:numerical-setup}. We test the SW and SW square-root formulations on classical benchmark problems: (1)~manufactured solution in section~\ref{sec:manufactored-solution-results}, (2)~linear and nonlinear Landau damping in section~\ref{sec:landau-numerical-results}, (3)~two-stream instability in section~\ref{sec:two-stream-numerical-results}, (4)~bump-on-tail instability in section~\ref{sec:bump-on-tail-numerical-results}, and (5) ion-acoustic wave in section~\ref{sec:ion-acoustic-numerical-results}. 

\subsection{Numerical Setup}\label{sec:numerical-setup}
The number of particles, total momentum, and total energy are either linear or quadratic invariants of the system, see Eqns.~\eqref{mass-sw-definition}\eqref{momentum-sw-definition}\eqref{kinetic-energy-sw-definition}\eqref{pot-energy-sw-definition} for the SW formulation and Eqns.~\eqref{mass-sw-sqrt-definition}\eqref{momentum-sw-sqrt-definition}\eqref{kin-energy-sw-sqrt-definition}\eqref{pot-energy-sw-sqrt-definition} for the SW square-root formulation. These can be conserved at the fully discrete level using implicit Gauss-Legendre temporal integrators. We use the second-order implicit midpoint temporal integrator, which is part of the Gauss-Legendre family of integrators. 
The nonlinear system at each time step is solved using an unpreconditioned Jacobian-Free-Newton-Krylov (JFNK) method~\cite{knoll_jfnk_2004} implemented in Python's \texttt{scipy} package with relative tolerance set to $10^{-8}$ and absolute tolerance set to $10^{-14}$. The Krylov linear solver that is embedded in the JFNK method is the default linear solver restarted generalized minimum residual method (LGMRES)~\cite{baker_2005_lgmres} with relative and absolute tolerance set to $10^{-5}$.
The semi-discrete Poisson equations~\eqref{posisson-ode-sw}-\eqref{poisson-ode-sw-sqrt} are solved using the generalized minimal residual method (GMRES)~\cite{youcef_1986_gmres} with a relative tolerance and absolute tolerance of $10^{-15}$. 
Aside from the ion-acoustic wave test, we set the ions as a static neutralizing background since the ions are approximately motionless on the electron time scales. 
The number of grid points in space is set to $N_{x} = 100$ and the time step to $\Delta t = 10^{-2}$, unless specified otherwise. Additionally, we use a second-order central finite differencing derivative approximation in space. 
In contrast to asymmetrically-weighted solvers~\cite{camporeale_sps_2016, delzanno_2015_sps, koshkarov_2021_sps}, we do not consider an artificial collisional operator in this work.
The initial electron distribution for all test cases, aside from the manufactured solution and the ion acoustic wave, is represented by 
\begin{equation*}
    f^{e}(x, v, t=0) = \frac{n_{0}(1+\epsilon \cos(k x))}{\sqrt{\pi}}
    \begin{cases}
        \frac{1}{\sqrt{2}\alpha^{e}_{\mathrm{sw}}} \exp\left(-\frac{1}{2}\left(\frac{v - u^{e}}{\alpha^{e}_{\mathrm{sw}}}\right)^2\right), \qquad &\text{SW formulation}\\
        \frac{1}{\alpha^{e}_{\mathrm{swsr}}} \exp\left(-\left(\frac{v - u^{e}}{\alpha^{e}_{\mathrm{swsr}}}\right)^2\right) ,\qquad &\text{SW square-root formulation}
    \end{cases}
\end{equation*}
where $\epsilon$ is the amplitude of the initial perturbation, $n_{0}$ is the average density, $\alpha^{e}_{\mathrm{sw}}$ is the electron velocity scaling for the SW formulation, $\alpha^{e}_{\mathrm{swsr}}$ is the electron velocity scaling for the SW square-root formulation, $u^{e}$ is the electron velocity shifting, and $k$ is the perturbation's wavenumber. The parameters for each numerical test (aside from the manufactured solution test) are shown in Table~\ref{tab:set-up-initial-condition}. Moreover, we set a realistic ion-to-electron mass ratio $m^{i}/ m^{e}=1836$ (ions are treated as protons).

\noindent
\begin{table}
\caption{Parameter setup for the linear Landau damping, nonlinear Landau damping, two-stream instability, and bump-on-tail instability. The subscripts `sw' and `swsr' correspond to the SW and SW square-root formulations, respectively.}
\centering
\begin{tabular}{c |c |c |c |c}
\textbf{Parameter} & \textbf{Linear Landau} & \textbf{Nonlinear Landau} & \textbf{Two-stream} & \textbf{Bump-on-tail}\\
\hline
$\ell$ & $2\pi$ & $4\pi$ & $2\pi$ & $20\pi$\\ 
\hline
$k$ & $1$ & $0.5$ & $1$ & $0.3$\\
\hline
$\epsilon$ & $0.01$ & $0.5$ & $10^{-3}$ & $0.03$\\
\hline
$t_{f}$ & $10$ & $10$ & $45$ & $20$\\
\hline
$n_{0}$ & $1$ & $1$ & $0.5$ \& $0.5$ & $0.9$ \& $0.1$\\
\hline
$u^{e}$ & 0 & 0 & $-1$ \& $1$ & $0$ \& $4.5$\\
\hline
$\alpha^{e}_{\mathrm{sw}}$ & $1$& $1$& $ 0.5/\sqrt{2}$ \& $ 0.5/\sqrt{2}$& $1$ \& $0.5$\\
\hline
$\alpha^{e}_{\mathrm{swsr}}$ &$\sqrt{2}$ & $\sqrt{2}$ & $0.5$ \& $0.5$ & $\sqrt{2}$ \& $1/\sqrt{2}$\\
\end{tabular}
\label{tab:set-up-initial-condition}
\end{table}

\subsection{Manufactured Solution}\label{sec:manufactored-solution-results}
We test the accuracy of the SW and SW square-root formulations using the method of manufactured solutions. We consider the transport of uncharged particles corresponding to 
\begin{equation*}
    \left(\frac{\partial}{\partial t} + v \frac{\partial}{\partial x}\right) f^{m}(x,v,t) = 0 \qquad \text{or} \qquad \left(\frac{\partial}{\partial t} + v \frac{\partial}{\partial x}\right) \sqrt{f^{m}(x,v,t)} = 0 \qquad \text{in }\Omega_{x} \times \Omega_{v} \times \Omega_{t}.
\end{equation*}
We set $\ell=2\pi$, $t_{f}=1$, and the initial condition to 
\begin{equation*}
    f^{m}(x, v, t=0)= f_{0}^{m}(x, v) = \pi^{-\frac{1}{4}} (2-\cos(x))^2 \exp{\left(-\left(\frac{v-1}{2}\right)^2\right)}.
\end{equation*}
From the method of characteristics, the analytic solution to the linear advection equation is given by $f^{m}(x, v, t) = f_{0}^{m}(x-vt, v)$. We set $\Delta t=10^{-3}$ for both formulations. The initial condition for the SW formulation is parameterized with $\alpha_{\mathrm{sw}}^{m} = \sqrt{2}$, $u_{\mathrm{sw}}^{m}=1$, $\mathbf{C}^{m}_{0, \mathrm{sw}}(t=0) = \left(2 - \cos(\mathbf{x})\right)^2$, and $\mathbf{C}^{m}_{n, \mathrm{sw}}(t=0)= 0$ for $n =1, 2, \ldots,  N_{v}-1$, where $\mathbf{x} = [x^{(1)}, x^{(2)}, \ldots, x^{(N_{x})}]^{\top}\in \mathbb{R}^{N_{x}}$. Similarly, the SW square-root formulation is parameterized with $\alpha_{\mathrm{swsr}}^{m} = 2$, $u_{\mathrm{swsr}}^{m} = 1$, $\mathbf{C}^{m}_{0, \mathrm{swsr}}(t=0) = \pi^{\frac{1}{8}}(2-\cos(\mathbf{x}))$, and $\mathbf{C}^{m}_{n, \mathrm{swsr}}(t=0)= 0$ for $n = 1, 2, \ldots N_{v}-1$. The governing equations for the expansion coefficients in the SW and SW square-root formulations are of the same form, i.e. 
\begin{equation*}
    \frac{\mathrm{d}}{\mathrm{d}t}\mathbf{C}^{m}_{n}(t) = -\mathbf{D}^{p=2} \left(\alpha^{m}\sqrt{\frac{n}{2}} \mathbf{C}^{m}_{n-1}(t) + \alpha^{m}\sqrt{\frac{n+1}{2}} \mathbf{C}^{m}_{n+1}(t) + u^{m} \mathbf{C}^{m}_{n}(t)\right).
\end{equation*}

Figure~\ref{fig:manufactored-solution-relative-error} shows the SW and SW square-root numerical results at $t=1$, along with the relative error with $N_{v}=N_{x}=100$. The numerical results show that the two formulations are in good agreement with the analytic solution and have comparable accuracy. The relative error in the SW formulation indicates the interaction of higher-order modes in comparison to the SW square-root formulation. 
Figure~\ref{fig:convergence-plot-v} shows that indeed the SW formulation requires $\sim44$ Hermite modes to approximate the manufactured solution, whereas the SW square-root formulation only requires $\sim36$ modes to approximate the manufactured solution. 
Figure~\ref{fig:convergence-plot-x} presents the $L^{2}$-norm absolute error of the two formulations as a function of the number of spatial grid points $N_{x}$. The $L^{2}$-norm absolute error is computed using $10^{3}$ grid points uniformly spaced in the velocity direction from $v \in [-5, 5]$ with $N_{x} = 50, 100, 200, 400$. The spatial resolution convergence rate is $2$, which agrees with the analytic second-order central finite difference convergence. Figure~\ref{fig:convergence-plot-x} also shows that the SW square-root formulation is slightly more accurate than the SW formulation for the manufactured test case. 

\begin{figure}
    \centering
    \begin{subfigure}[b]{0.51\textwidth}
         \centering
         \caption{SW formulation}
         \includegraphics[width=\textwidth]{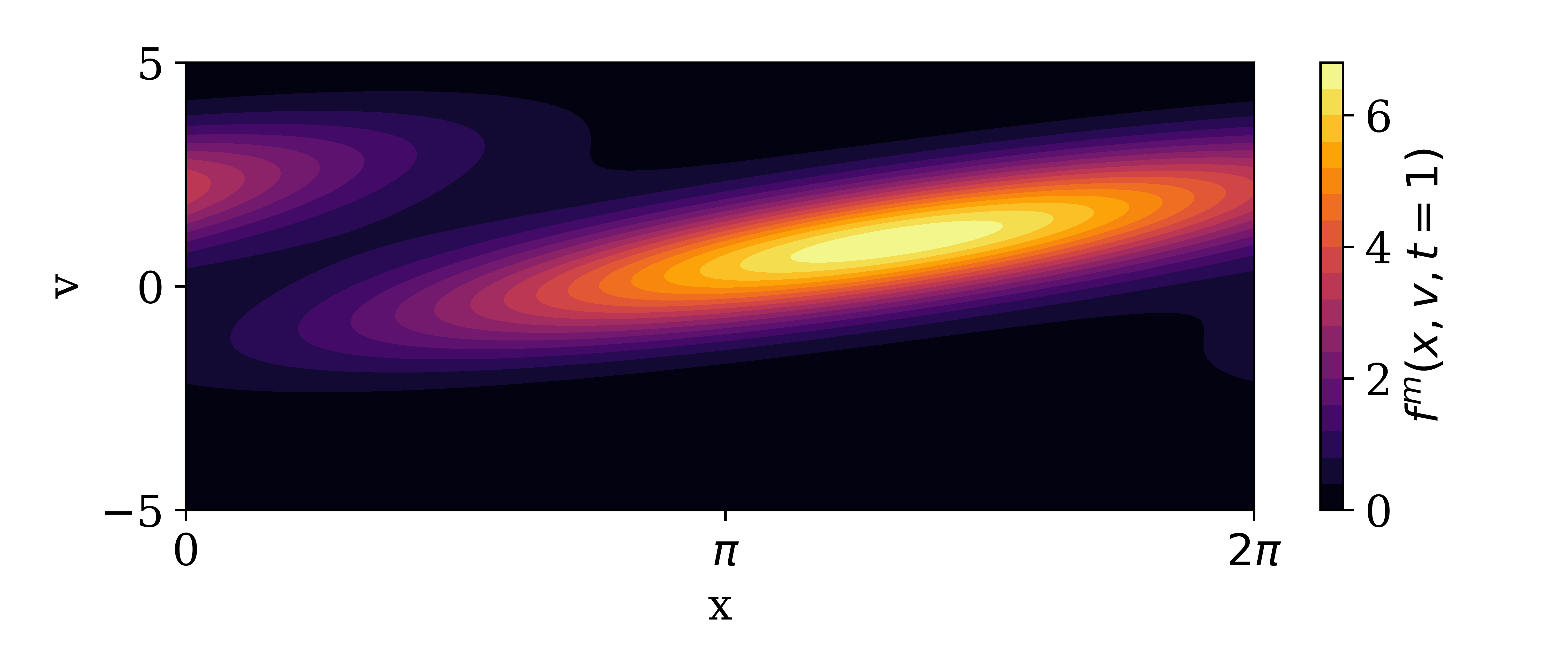}
         \label{fig:SW-manufactored}
     \end{subfigure}
     \hspace{-15pt}
    \begin{subfigure}[b]{0.51\textwidth}
         \centering
         \caption{SW square-root formulation}
         \includegraphics[width=\textwidth]{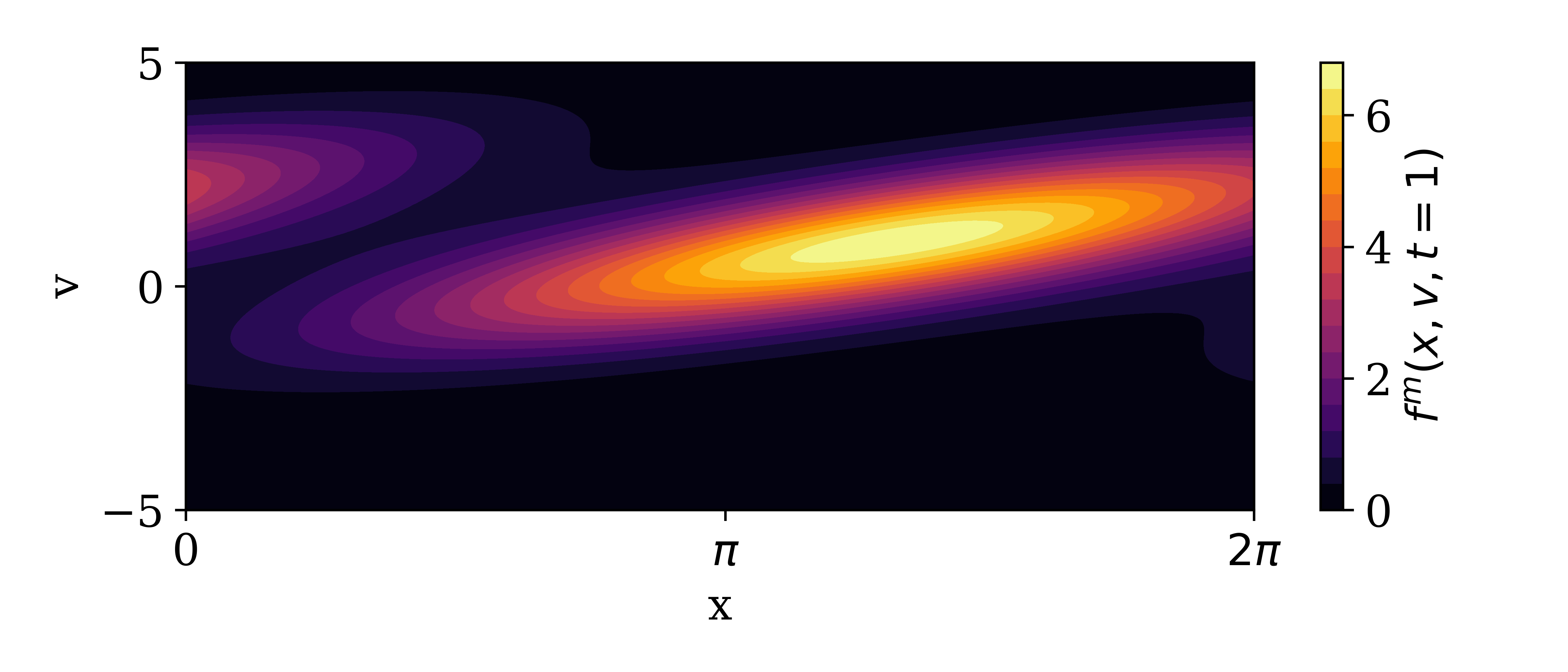}
         \label{fig:SW-sqrt-manufactored}
     \end{subfigure}
    \begin{subfigure}[b]{0.51\textwidth}
         \centering
         \caption{SW formulation relative error}
         \includegraphics[width=\textwidth]{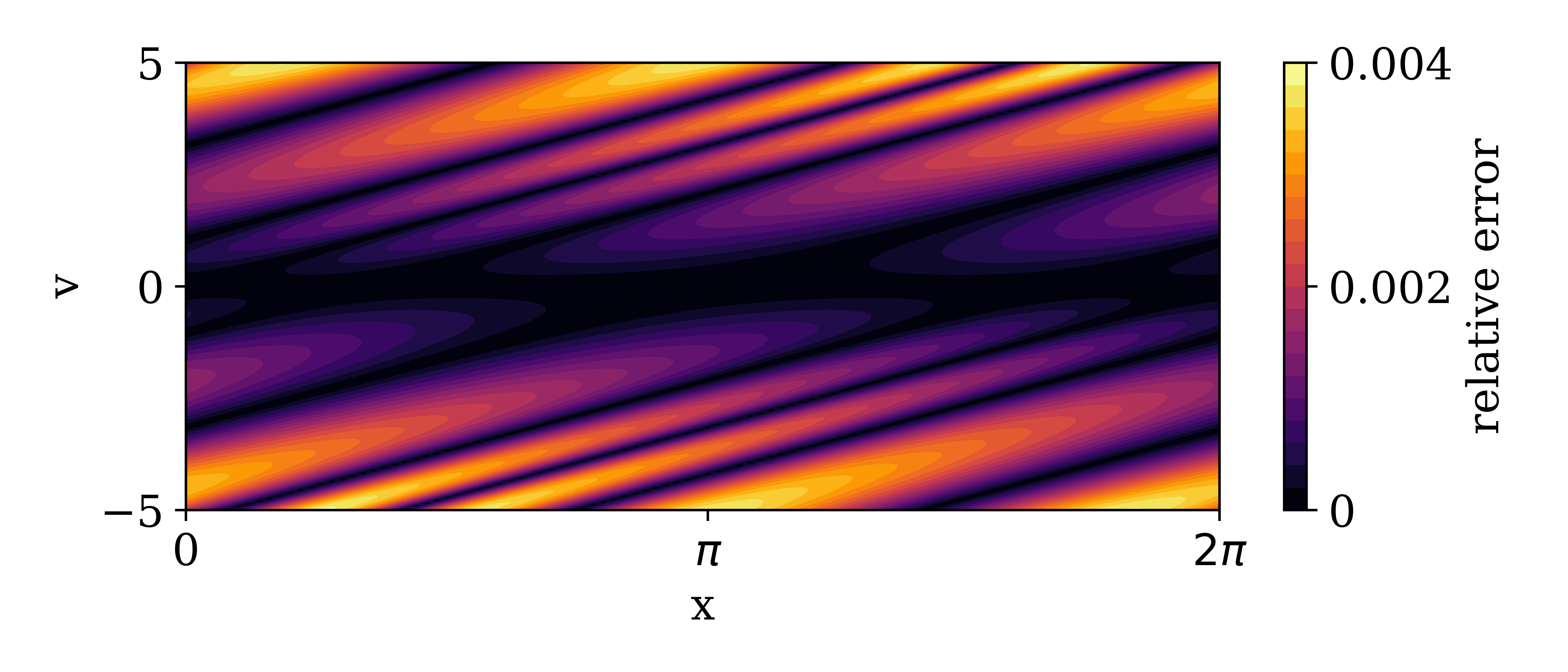}
         \label{fig:relative-error-manufactured-SW}
     \end{subfigure}
     \hspace{-15pt}
        \begin{subfigure}[b]{0.51\textwidth}
         \centering
         \caption{SW square-root formulation relative error}
         \includegraphics[width=\textwidth]{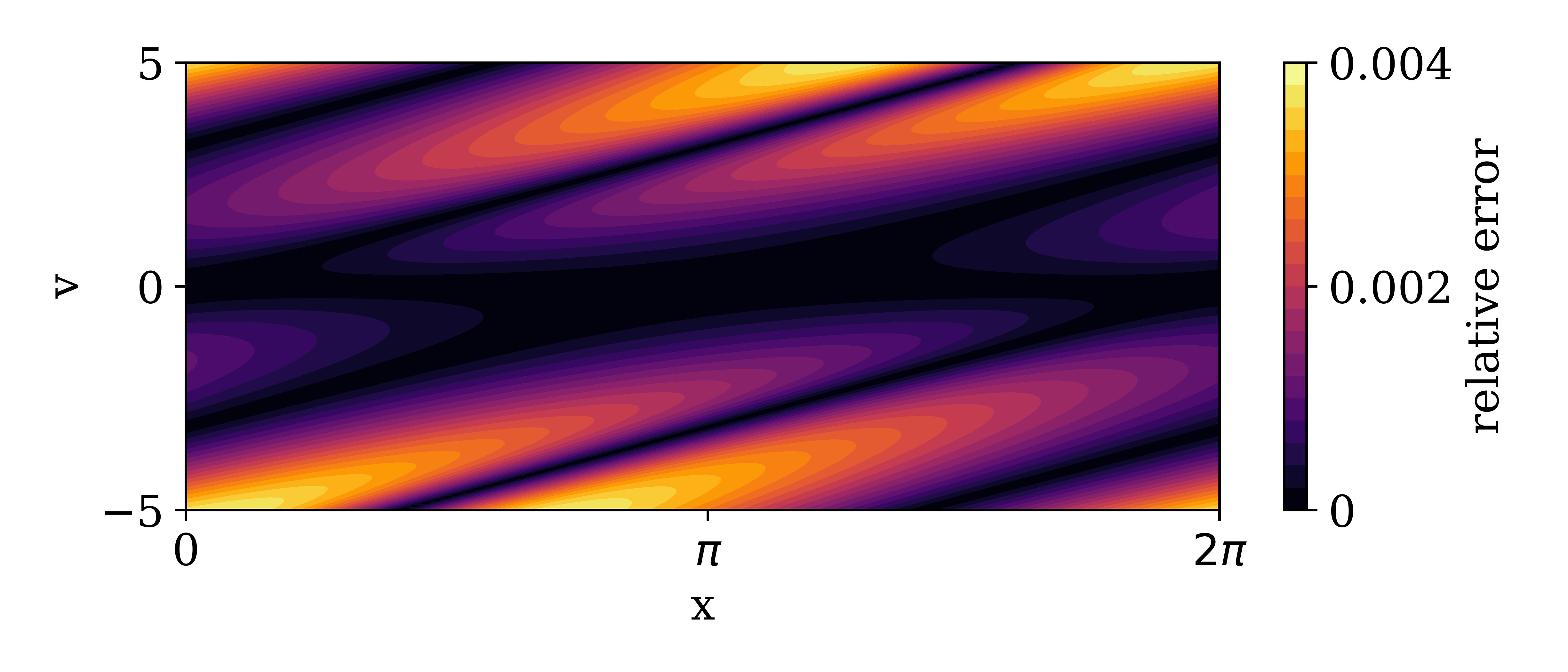}
         \label{fig:relative-error-manufactured-SW-sqrt}
     \end{subfigure}
    \caption{The (a)~SW and (b)~SW square-root formulations manufactured solution distribution function $f^{m}(x, v, t=1)$ results with $N_{x} = N_{v}=100$. Subfigures~(c) and~(d) show the relative error between the analytic solution and the SW and SW square-root numerical results, respectively. The numerical results show that as expected the SW and SW square-root formulation accuracy is comparable, yet the SW formulation has more interaction between higher-order modes. }
    \label{fig:manufactored-solution-relative-error}
\end{figure}

\begin{figure}
    \centering
    \begin{subfigure}[b]{0.5\textwidth}
        \centering 
        \caption{Velocity spectral convergence}
        \includegraphics[width=\textwidth]{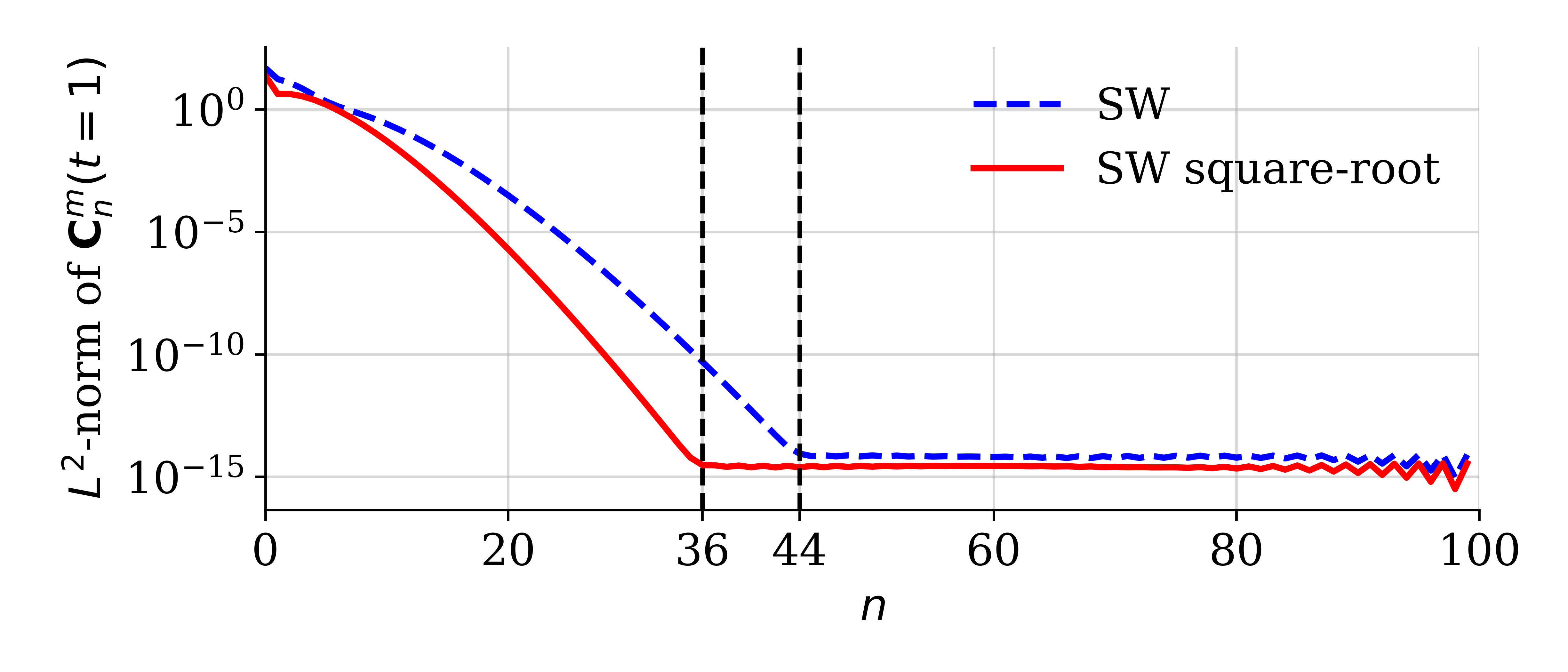}
        \label{fig:convergence-plot-v}
    \end{subfigure}
    \hspace{-12pt}
    \begin{subfigure}[b]{0.5\textwidth}
        \centering 
        \caption{Spatial finite difference convergence}
        \includegraphics[width=\textwidth]{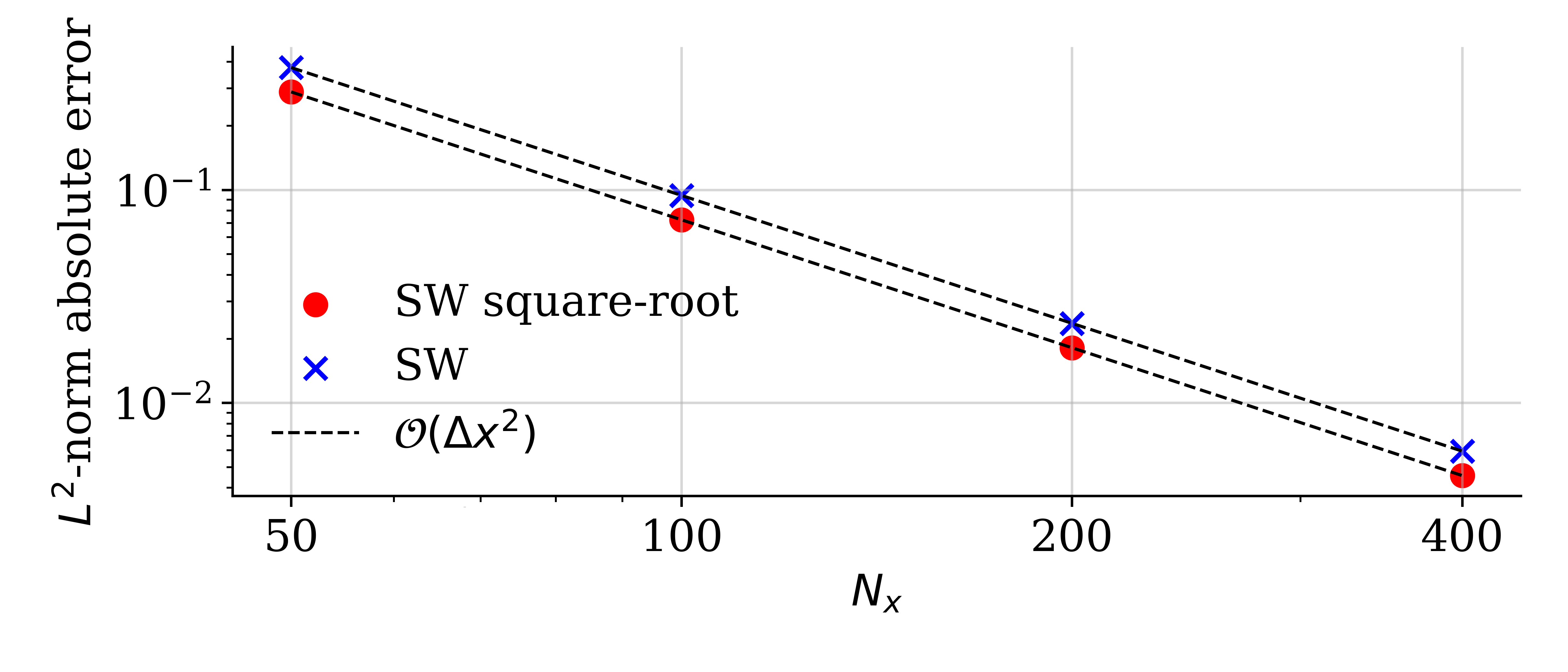}
        \label{fig:convergence-plot-x}
    \end{subfigure}
    \caption{Subfigure~(a) shows the $L^2$-norm of the expansion coefficients with $N_{v}=N_{x}=100$. The spectral convergence illustrates that the SW square-root formulation requires fewer spectral coefficients in comparison to the SW formulation to approximate the manufactured solution. Subfigure~(b) shows the $L^2$-norm absolute error of the SW and SW square-root manufactured solutions. In subfigure~(b), we set $N_{v}=100$ and $N_{x} = 50, 100, 200, 400$. The numerical convergence rate is $2$ for both formulations, which agrees with the second-order central finite difference convergence rate. }
    \label{fig:convergence-plot}
\end{figure}

\subsection{Landau Damping}\label{sec:landau-numerical-results}
Landau damping is a classic benchmark problem for kinetic plasma codes caused by wave-particle resonance~\cite{landau_46}, where particles gain energy from or lose energy to the wave. In a near-Maxwellian distribution, more particles have velocities slightly slower than the wave phase velocity so there is a net transfer of energy from the wave to the particles, and as a result, the electric field damps and particles gain energy.  We examine the linear and nonlinear Landau damping test cases in section~\ref{sec:linear-landau-numerical-results} and section~\ref{sec:nonlinear-landau-numerical-results}, respectively. 

\subsubsection{Linear Landau Damping}\label{sec:linear-landau-numerical-results}
For a small perturbation in the initial electron distribution, Landau damping can be described using linear theory. We set the initial perturbation amplitude to $\epsilon=0.01$.
Figure~\ref{fig:conservation-linear-landau} shows the conservation of particle number, momentum, and energy of the SW and SW square-root formulations with $N_{v}=100$ and $N_{v}=101$. 
Since we set the electron velocity shifting parameter to $u^{e}=0$, the particle number and energy are conserved in the SW formulation when $N_{v}$ is odd, as shown in Figure~\ref{fig:conservation-linear-landau-sw-101}, and momentum is conserved when $N_{v}$ is even, as shown in Figure~\ref{fig:conservation-linear-landau-sw-100}. 
On the contrary, the particle number is conserved in the SW square-root formulation for any $N_{v}$. Figure~\ref{fig:conservation-linear-landau-sw-sqrt-101} and Figure~\ref{fig:conservation-linear-landau-sw-sqrt-100} show the SW square-root formulation conservation properties for $N_{v}=101$ and $N_{v}=100$, respectively, which are almost identical, unlike the SW formulation conservation properties which depend on $N_{v}$ being even or odd. 
Figure~\ref{fig:conservation-linear-landau} also shows the drift rate in the conservation for each setup. The numerical results show that the analytic conservation properties derived in section~\ref{sec:conservation-properties} match the numerical drift rate in the conservation of particle number, momentum, and energy. 
Despite the momentum drift rates given by Eqns.~\eqref{change-in-momentum-sw} and~\eqref{drift-in-momentum-sw-sqrt}, momentum is conserved close to machine precision in the SW square-root formulation (for both $N_{v}=100$ and $N_{v}=101$) and the SW formulation (with $N_{v}=101$). This is because the momentum drift rate is a function of the product of the electric field and the last spectral coefficient, which are both small due to wave damping and spectral convergence.

As mentioned in section~\ref{sec:time-reversibility}, the anti-symmetry of the semi-discrete Eq.~\eqref{vlasov-ode} results in explicit Runge-Kutta temporal integrators being approximately time-reversal symmetric. We numerically demonstrate this with the SW square-root formulation. Figure~\ref{fig:damping-rate-linear-landau-runge-kutta} shows the electric field amplitude as a function of time for the forward and backward in-time simulations using the non-adaptive 3rd-order explicit Runge-Kutta integrator of Bogacki-Shampine~\cite{bogacki_shampine_1989_integrator} with $\Delta t = 10^{-3}$ and $N_{v}=100$. Both the forward and backward in-time simulations agree with the linear theory electric field damping rate\footnote{The linear theory damping rate $\gamma = -0.851$ for the linear Landau damping test is obtained by solving the dispersion equation numerically with wavenumber $k=1$~\cite[Table 1]{landau_damping_rate_1973}, i.e. $2 + \zeta Z(\zeta) = 0$, where $\zeta = (w_{r} + \gamma i)/\sqrt{2}$ and $Z(\zeta) = \pi^{-1/2}\int_{-\infty}^{\infty} (x-\zeta)^{-1}\exp(-x^2)\mathrm{d}x$ is the plasma dispersion function, see~\cite[\S 2]{gary_1993_theory} for a detailed derivation.
}.
The numerical results show that although explicit Runge-Kutta temporal integrators are not time-reversal symmetric, the anti-symmetric structure of the semi-discrete equations~\eqref{vlasov-ode} makes such temporal integrators approximately time-reversible. Additionally,  Figure~\ref{fig:damping-rate-linear-landau-runge-kutta} shows that the SW square-root formulation starts to exhibit recurrence
phenomena (artificial echoes) due to the closure by truncation at $t \sim 8$ (normalized). As demonstrated  by~\cite{camporeale_sps_2016}, the onset of the temporal recurrence can be extended by increasing the number of velocity spectral terms $N_{v}$, a relationship that roughly scales as $\mathcal{O}(\sqrt{N_{v}})$.
\begin{figure}
    \centering
     \begin{subfigure}[b]{0.49\textwidth}
         \centering
         \caption{SW formulation $N_{v}=101$ (odd)}
         \includegraphics[width=\textwidth]{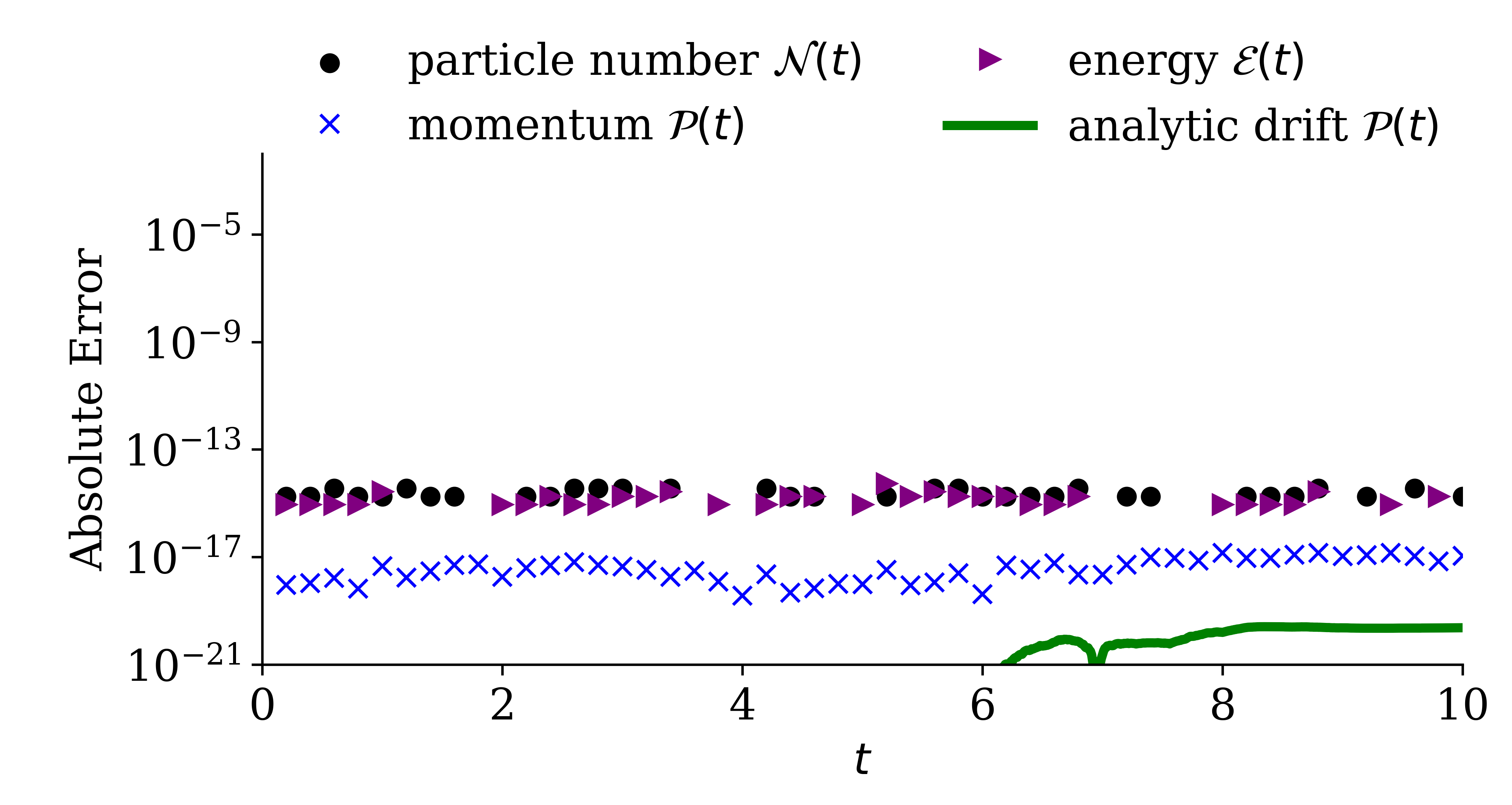}
         \label{fig:conservation-linear-landau-sw-101}
     \end{subfigure}
         \begin{subfigure}[b]{0.49\textwidth}
         \centering
         \caption{SW formulation $N_{v}=100$ (even)}
         \includegraphics[width=\textwidth]{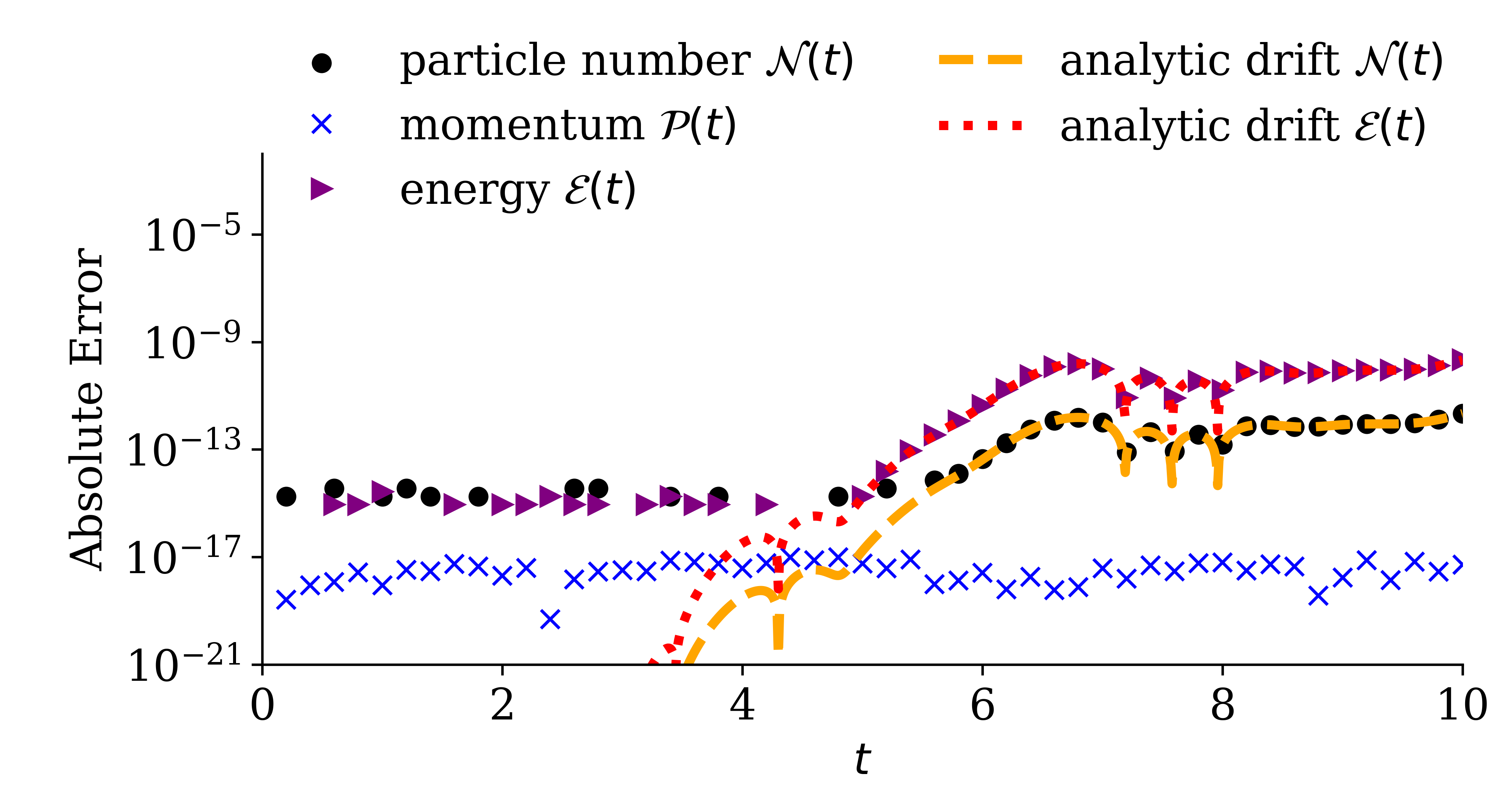}
         \label{fig:conservation-linear-landau-sw-100}
     \end{subfigure}
    \begin{subfigure}[b]{0.49\textwidth}
         \centering
         \caption{SW square-root formulation $N_{v}=101$ (odd)}
         \includegraphics[width=\textwidth]{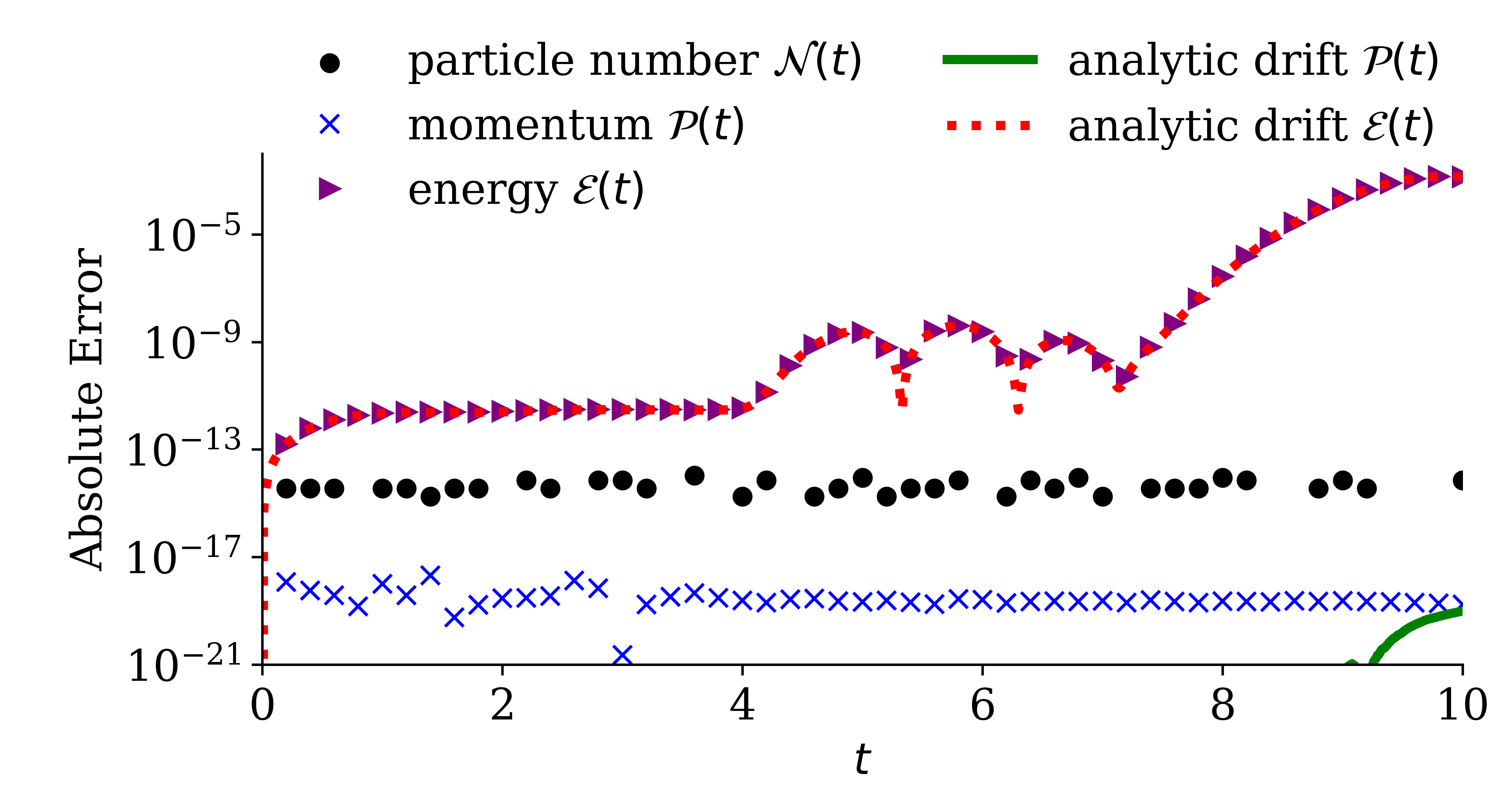}
         \label{fig:conservation-linear-landau-sw-sqrt-101}
     \end{subfigure}
    \begin{subfigure}[b]{0.49\textwidth}
         \centering
         \caption{SW square-root formulation $N_{v}=100$ (even)}
         \includegraphics[width=\textwidth]{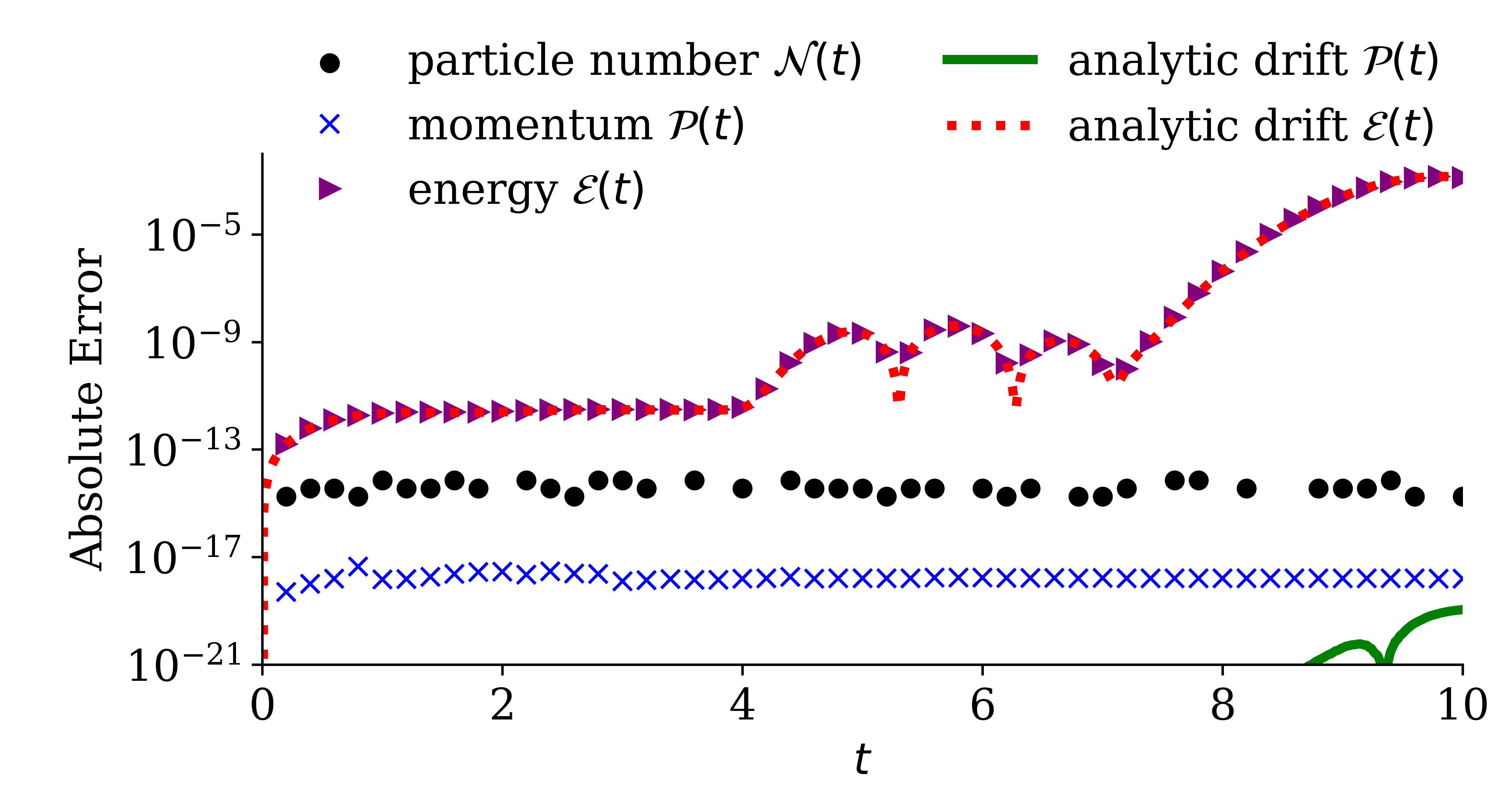}
         \label{fig:conservation-linear-landau-sw-sqrt-100}
     \end{subfigure}
    \caption{Conservation of particle number, momentum, and energy via the SW formulation~(a/b) and the SW square-root formulation~(c/d) for linear Landau damping. Subfigures (a/c) have an odd number of spectral terms ($N_v = 101$) and subfigures (b/d) have an even number of spectral terms ($N_v = 100$). The numerical drift rate (the rate at which conservation no longer holds) in each setup matches the analytic drift rate, which we derive in Eqns.~\eqref{change-in-mass-sw}, \eqref{change-in-momentum-sw}, \eqref{change-in-total-energy-sw}, \eqref{drift-in-momentum-sw-sqrt}, and~\eqref{change-in-total-energy-sw-sqrt}.}
    \label{fig:conservation-linear-landau}
\end{figure}

\begin{figure}
    \centering
    \includegraphics[width=0.6\textwidth]{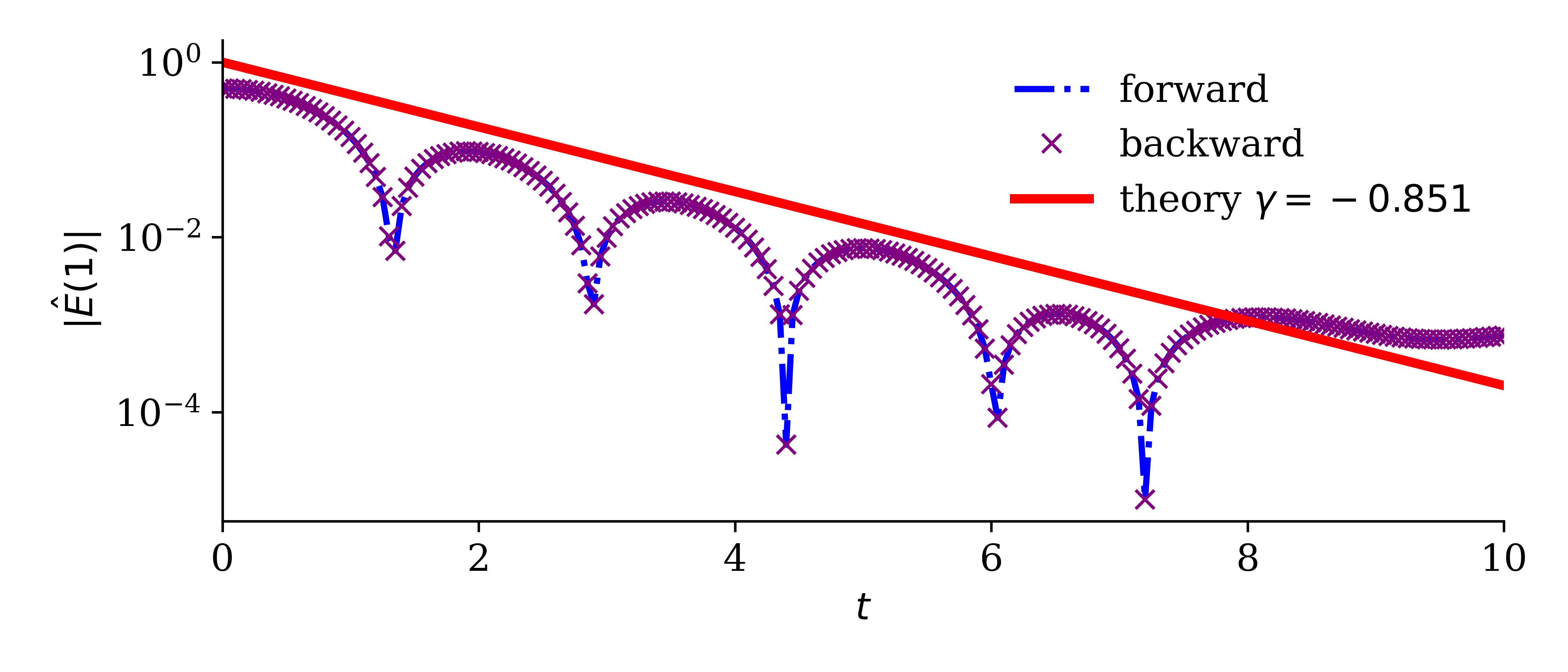}
    \caption{The linear Landau damping electric field amplitude as a function of time computed via the SW square-root formulation. The numerical damping rate agrees with the theoretical damping rate from~\cite[Table 1]{landau_damping_rate_1973}. The forward and backward simulations are evolved via the non-adaptive 3rd-order explicit Runge-Kutta integrator of Bogacki-Shampine~\cite{bogacki_shampine_1989_integrator} with $\Delta t = 10^{-3}$ and $N_{v}=100$. Note that such an explicit Runge-Kutta temporal integrator is not time-reversal symmetric, yet the anti-symmetric structure of the equations makes this method approximately time-reversible.}
    \label{fig:damping-rate-linear-landau-runge-kutta}
\end{figure}

\subsubsection{Nonlinear Landau Damping}\label{sec:nonlinear-landau-numerical-results}
The linear theory fails to describe the damping rate if the magnitude of the initial perturbation is large as nonlinear effects become important. We set the initial perturbation amplitude for the nonlinear Landau damping test case to $\epsilon=0.5$. The electron distribution function $f^{e}(x, v, t)$ at various time instances is shown in Figure~\ref{fig:nonlinear-landau-snapshots}. Figure~\ref{fig:nonlinear-landau-evolution-sw} illustrates that the SW formulation results become negative in phase space starting from $t=5$. By construction, the SW square-root formulation preserves the distribution function's non-negative property, as shown in Figure~\ref{fig:nonlinear-landau-evolution-sw-sqrt}. At  $t=10$, it appears that the SW formulation results display a slightly higher degree of filamentation than the SW formulation results, as shown in Figure~\ref{fig:filamentation-nonlinear-landau}. Note that the SW and SW square-root formulations require a higher number of Hermite basis functions $N_{v}$ (velocity resolution) to accurately simulate the nonlinear Landau damping for long-time simulations ($t\sim 50$ normalized). This is attributed to the emergence of small-scale structures in velocity space.

Figure~\ref{fig:conservation-nonlinear-landau} shows the conservation of particle number, momentum, and energy for the nonlinear Landau damping test case. Similar to the linear Landau damping test case, we set the velocity shifting parameter $u^{e}=0$, thus the SW formulation conserves particle number and energy when $N_{v}$ is odd as shown in Figure~\ref{fig:conservation-nonlinear-landau-sw-101} and conserves momentum when $N_{v}$ is even as shown in Figure~\ref{fig:conservation-nonlinear-landau-sw-100}. As predicted, the SW square-root formulation conserves particle number, independently of $N_{v}$, as shown in Figure~\ref{fig:conservation-nonlinear-landau-sw-101} (with $N_{v}=101$) and Figure~\ref{fig:conservation-nonlinear-landau-sw-sqrt-100} (with $N_{v}=100$). Interestingly, although energy and momentum are not conserved in the SW square-root formulation, their drift rate seems to grow very slowly in time. 

\begin{figure}
    \centering
     \begin{subfigure}[b]{0.49\textwidth}
         \centering
         \caption{SW formulation}
         \includegraphics[width=\textwidth]{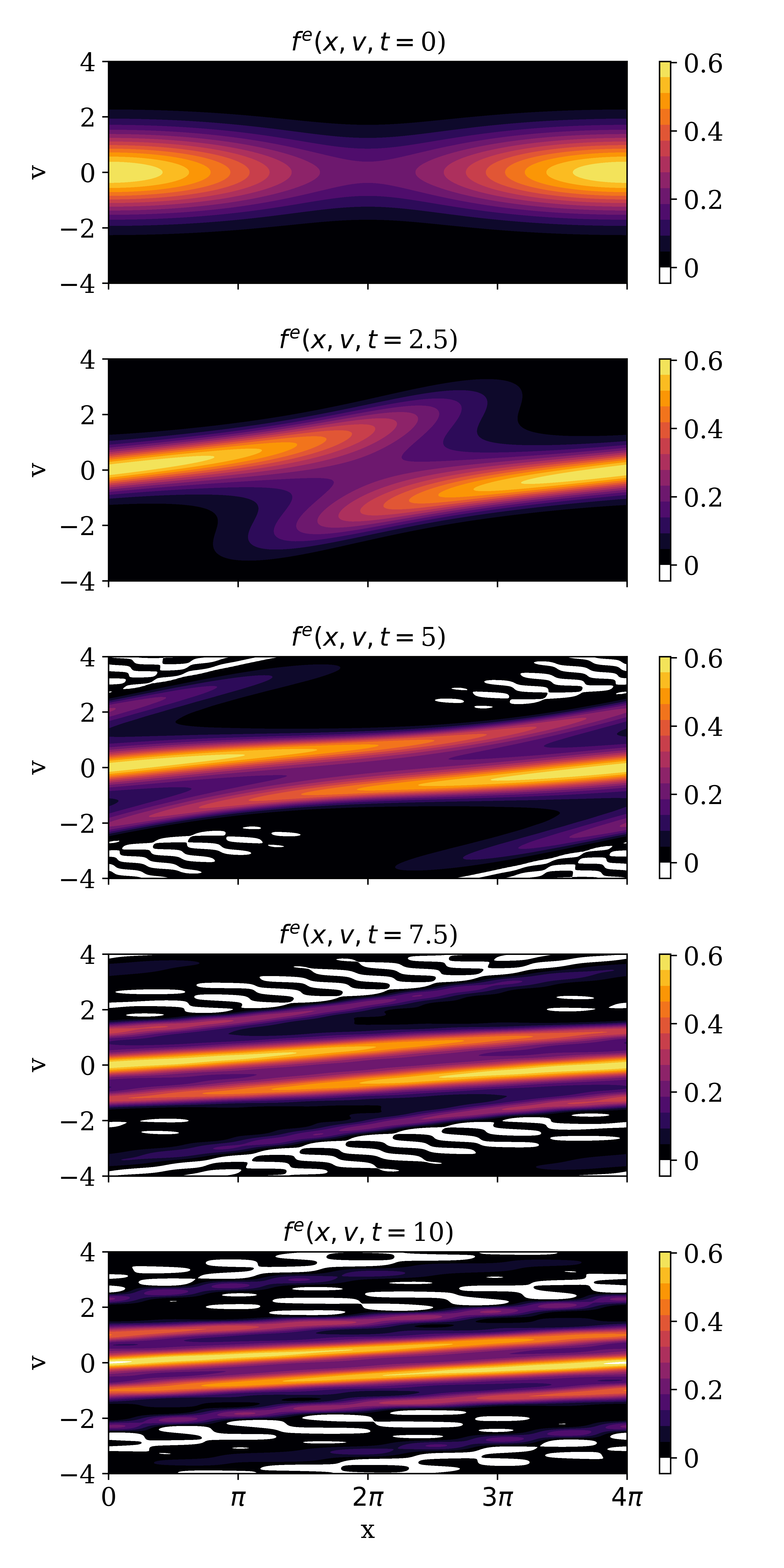}
         \label{fig:nonlinear-landau-evolution-sw}
     \end{subfigure}
    \begin{subfigure}[b]{0.49\textwidth}
         \centering
         \caption{SW square-root formulation}
         \includegraphics[width=\textwidth]{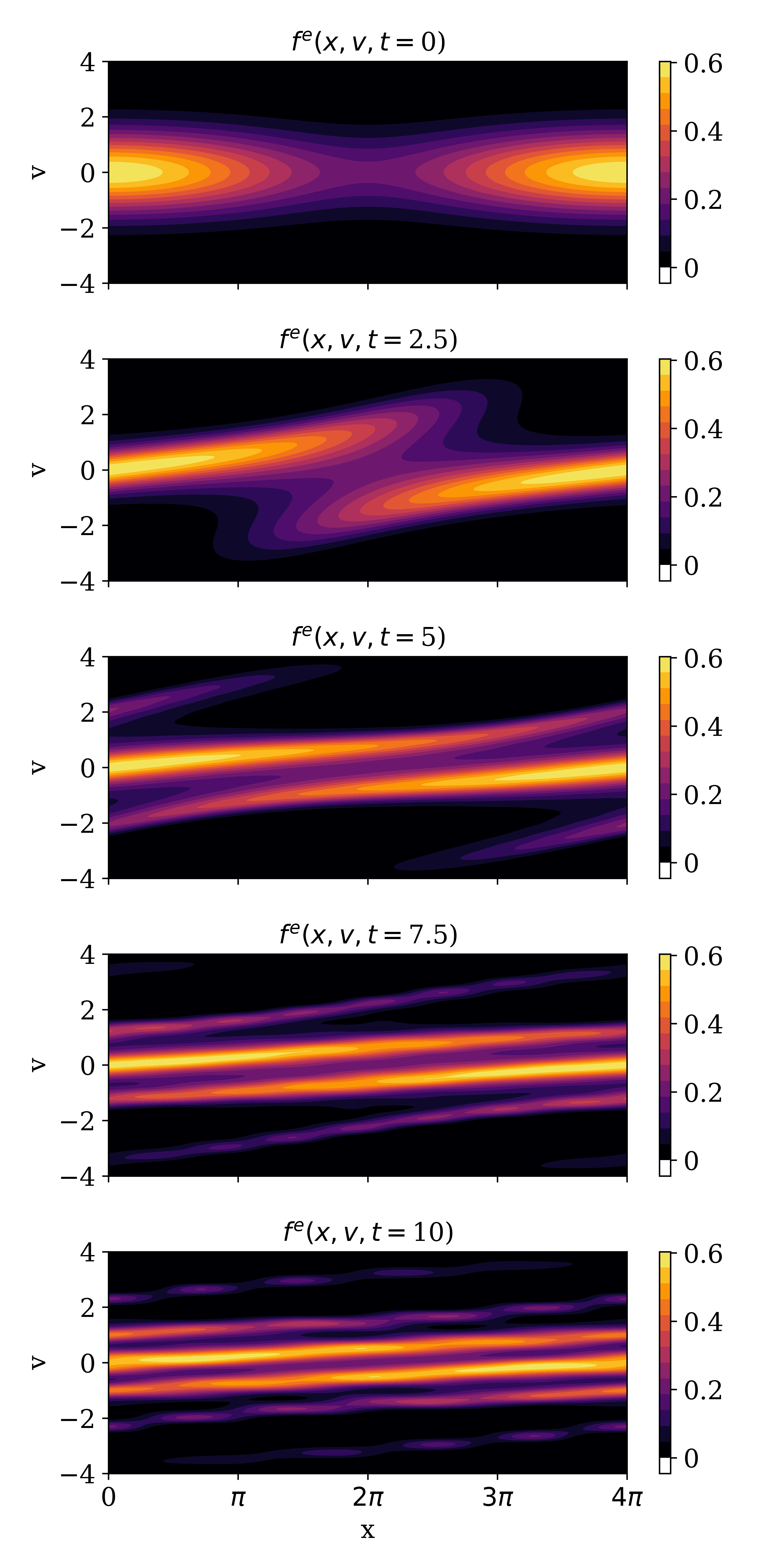}
         \label{fig:nonlinear-landau-evolution-sw-sqrt}
     \end{subfigure}
    \caption{The evolution of the electron distribution function $f^{e}(x, v, t)$ for nonlinear Landau damping with $N_{v}=100$. The distribution function becomes negative in the (a)~SW formulation whereas the (b)~SW square-root formulation is positivity preserving by construction. }
    \label{fig:nonlinear-landau-snapshots}
\end{figure}

\begin{figure}
    \centering
    \includegraphics[width=0.6\textwidth]{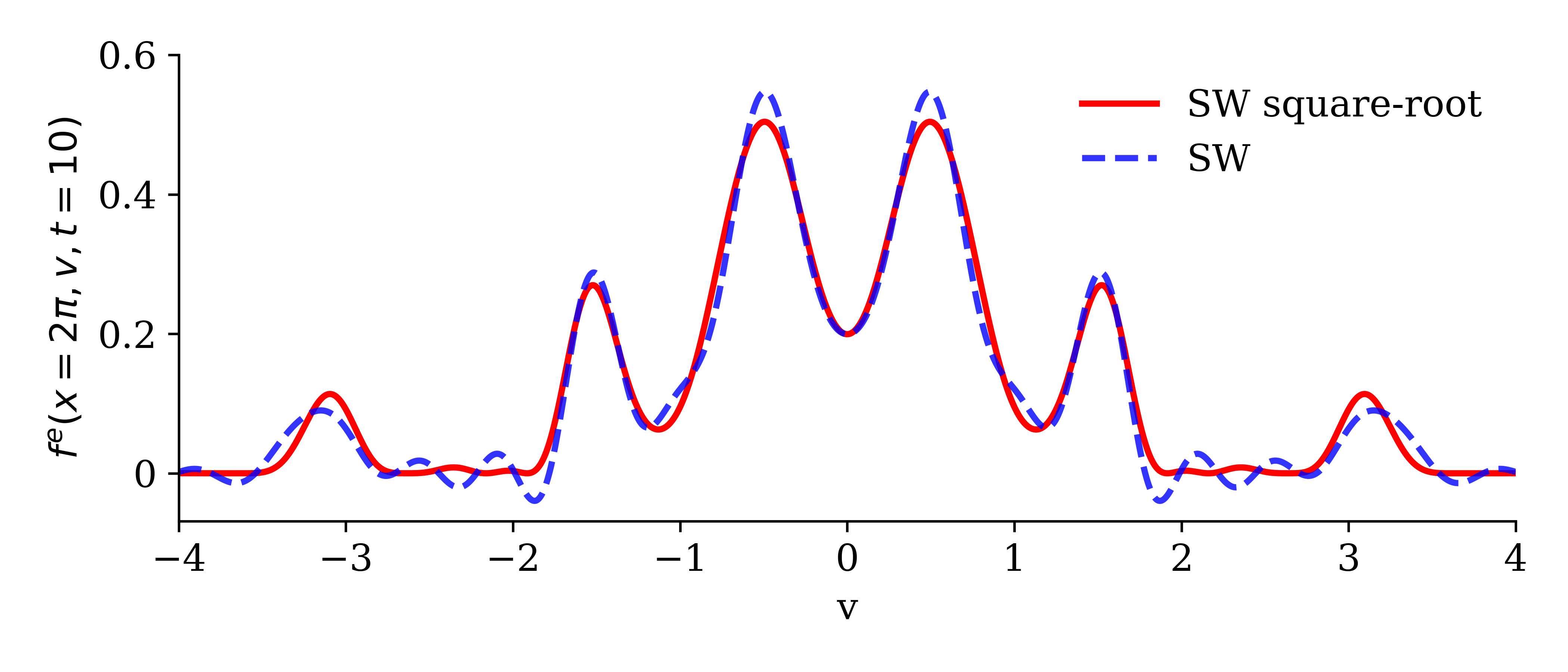}
    \caption{A cross-section of the nonlinear Landau damping electron particle distribution function $f^{e}(x=2\pi, v, t=10)$ via the two formulations: SW and SW square-root (with $N_{v}=101$). The SW formulation appears to have more filamentation near regions close to zero. }
    \label{fig:filamentation-nonlinear-landau}
\end{figure}

\begin{figure}
    \centering
     \begin{subfigure}[b]{0.49\textwidth}
         \centering
         \caption{SW formulation $N_{v}=101$ (odd)}
         \includegraphics[width=\textwidth]{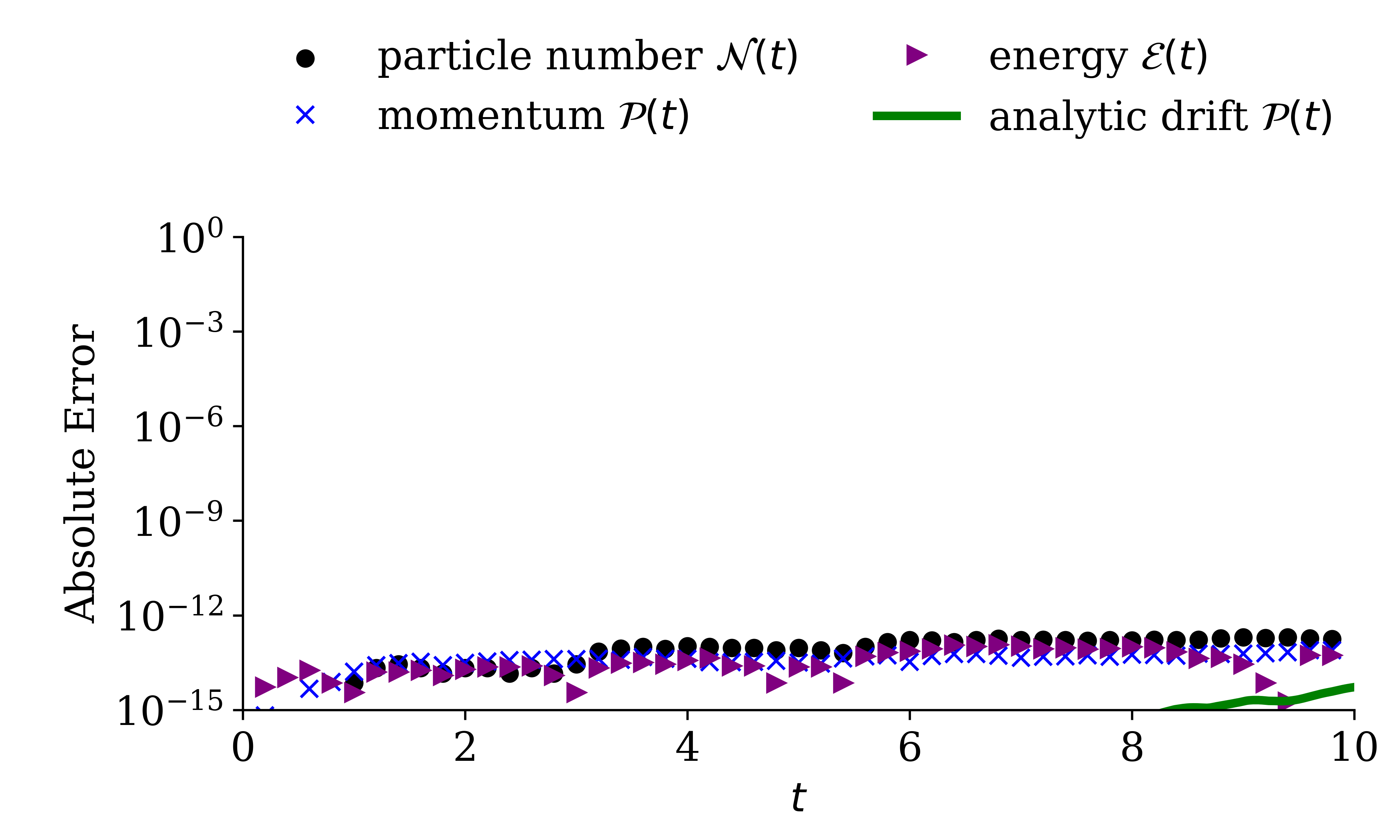}
         \label{fig:conservation-nonlinear-landau-sw-101}
     \end{subfigure}
    \begin{subfigure}[b]{0.49\textwidth}
         \centering
         \caption{SW formulation $N_{v}=100$ (even)}
         \includegraphics[width=\textwidth]{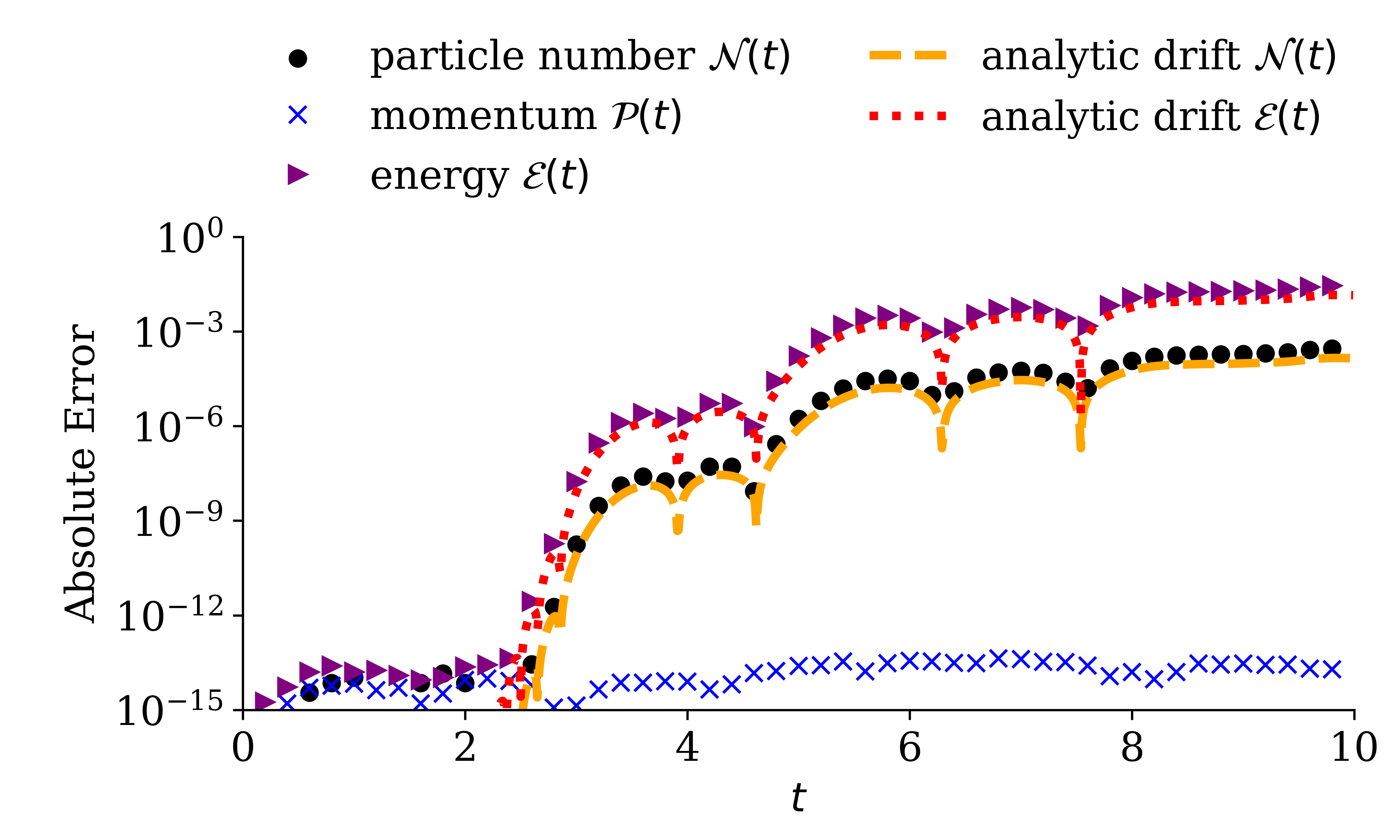}
         \label{fig:conservation-nonlinear-landau-sw-100}
     \end{subfigure}
    \begin{subfigure}[b]{0.49\textwidth}
         \centering
         \caption{SW square-root formulation $N_{v}=101$ (odd)}
         \includegraphics[width=\textwidth]{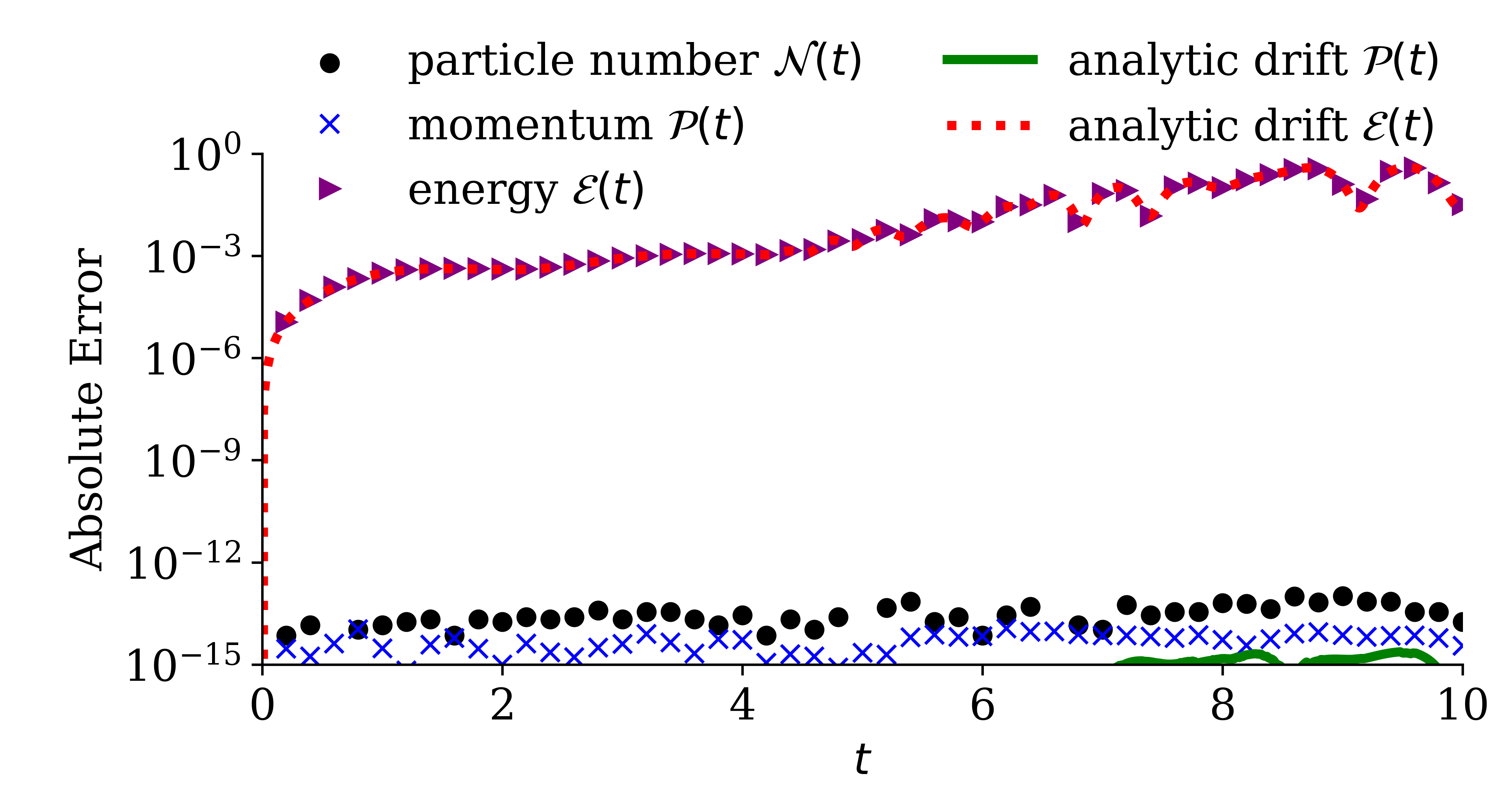}
         \label{fig:conservation-nonlinear-landau-sw-sqrt-101}
     \end{subfigure}
    \begin{subfigure}[b]{0.49\textwidth}
         \centering
         \caption{SW square-root formulation $N_v=100$ (even)}
         \includegraphics[width=\textwidth]{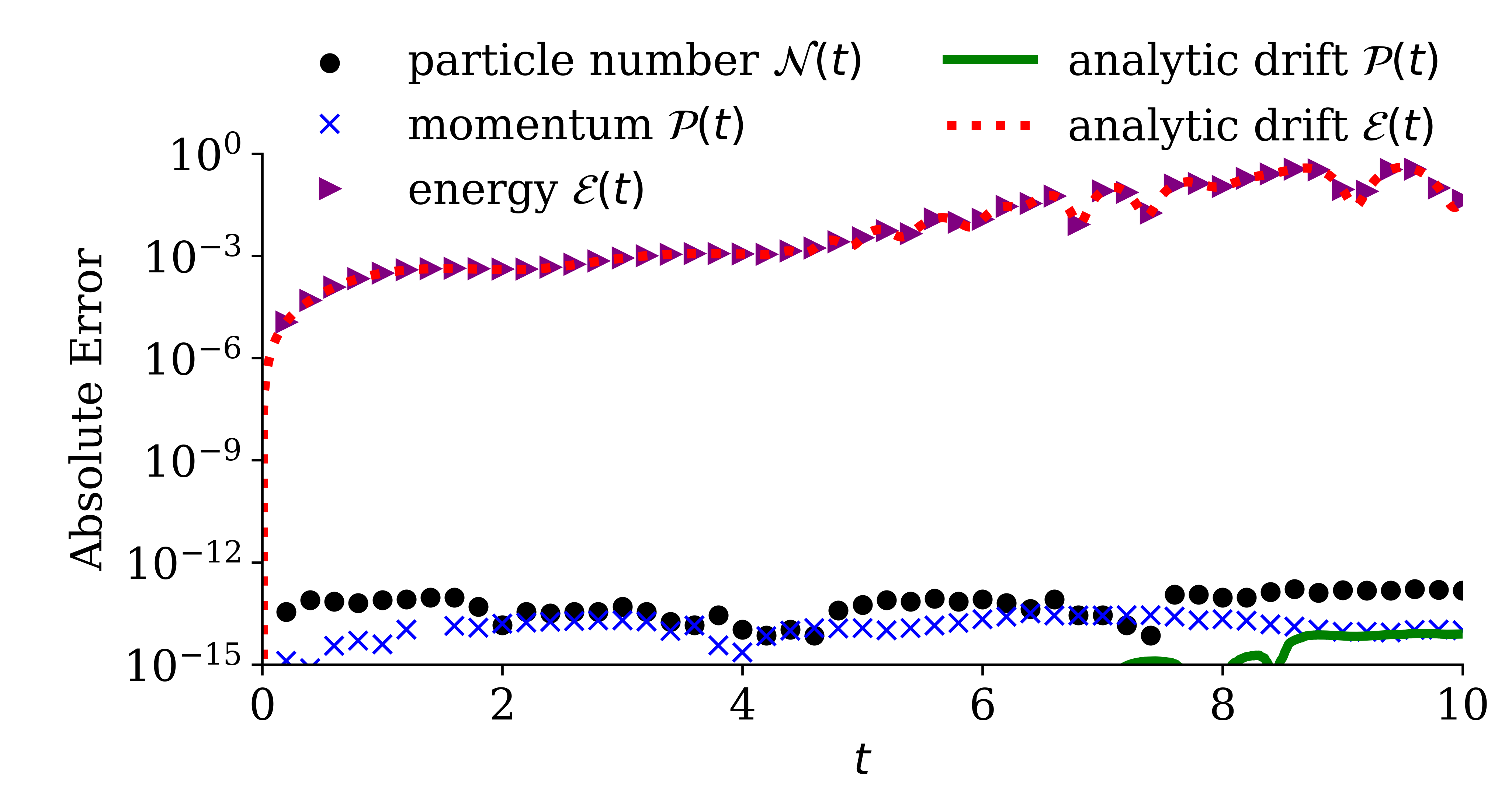}
         \label{fig:conservation-nonlinear-landau-sw-sqrt-100}
     \end{subfigure}
    \caption{Same as Figure~\ref{fig:conservation-linear-landau} for nonlinear Landau damping. The numerical drift rate for the SW square-root formulation is independent of $N_{v}$, whereas the SW formulation drift rate depends on whether $N_{v}$ is either odd or even.}
    \label{fig:conservation-nonlinear-landau}
\end{figure}

\subsection{Two-stream Instability}\label{sec:two-stream-numerical-results}
The two-stream instability is a classic velocity-space microinstability corresponding to two counter-streaming electron beams. Figure~\ref{fig:two-stream-electron-distribution-t-45} shows the electron distribution function in phase space at $t=45$ with $N_{v}=101$. The SW formulation results become negative in certain regions of phase space, which is relatively benign as such regions are far from the phase space vortex. Figure~\ref{fig:filamentation-two-stream} shows a cross-section of the electron distribution function for the two formulations, which indicates along with Figure~\ref{fig:two-stream-electron-distribution-t-45}, that the SW square-root formulation exhibits more filamentation and the negativity of the distribution function in SW is insignificant. In contrast to the nonlinear Landau damping test case, the SW square-root formulation results display a higher degree of filamentation in comparison to the SW formulation results. This suggests that the filamentation effects are problem-dependent. 
The numerical electric field growth rate for the SW and SW square-root formulations with $N_{v}=101$ is shown in Figure~\ref{fig:growth-rate-two-stream}. The numerical growth rate for both formulations agrees with the linear theory growth rate\footnote{The linear theory growth rate $\gamma=0.185$ for the two-stream instability test is obtained by solving numerically the dispersion equation for $k=1$ wavenumber, i.e. $9  + 4\zeta_{1}Z(\zeta_{1}) + 4\zeta_{2}Z(\zeta_{2}) = 0$, where $\zeta_{1} = 2(w_{r}-1 + \gamma i)$, $\zeta_{2} = 2(w_{r}+1 + \gamma i)$, and $Z(\zeta) = \pi^{-1/2} \int_{-\infty}^{\infty} (x-\zeta)^{-1}\exp(-x^2)\mathrm{d}x$, see~\cite[\S 3]{gary_1993_theory} for a detailed derivation.}.

Figure~\ref{fig:conservation-two-stream} compares the SW and SW square-root conservation laws for the two-stream instability. 
Since the velocity shifting parameter $u^{s} \neq 0$, the SW formulation only conserves particle number when $N_{v}$ is odd, see Figure~\ref{fig:conservation-two-stream-sw-101}, and does not conserve particle number, momentum, or energy when $N_{v}$ is even, see Figure~\ref{fig:conservation-two-stream-sw-101}. However, the momentum drift rate is very small in both cases. The SW square-root formulation conserves particle number regardless of whether $N_{v}$ is even or odd and momentum drift is negligible as shown in Figure~\ref{fig:conservation-two-stream-sw-sqrt-101} and Figure~\ref{fig:conservation-two-stream-sw-sqrt-100}. Here, the momentum drifts for both the SW and SW square-root formulations, see Eqn~\eqref{change-in-momentum-sw} and~\eqref{drift-in-momentum-sw-sqrt}, are negligible due to the phase space symmetry.
\begin{figure}
    \centering
     \begin{subfigure}[b]{0.5\textwidth}
         \centering
         \caption{SW formulation}
         \includegraphics[width=\textwidth]{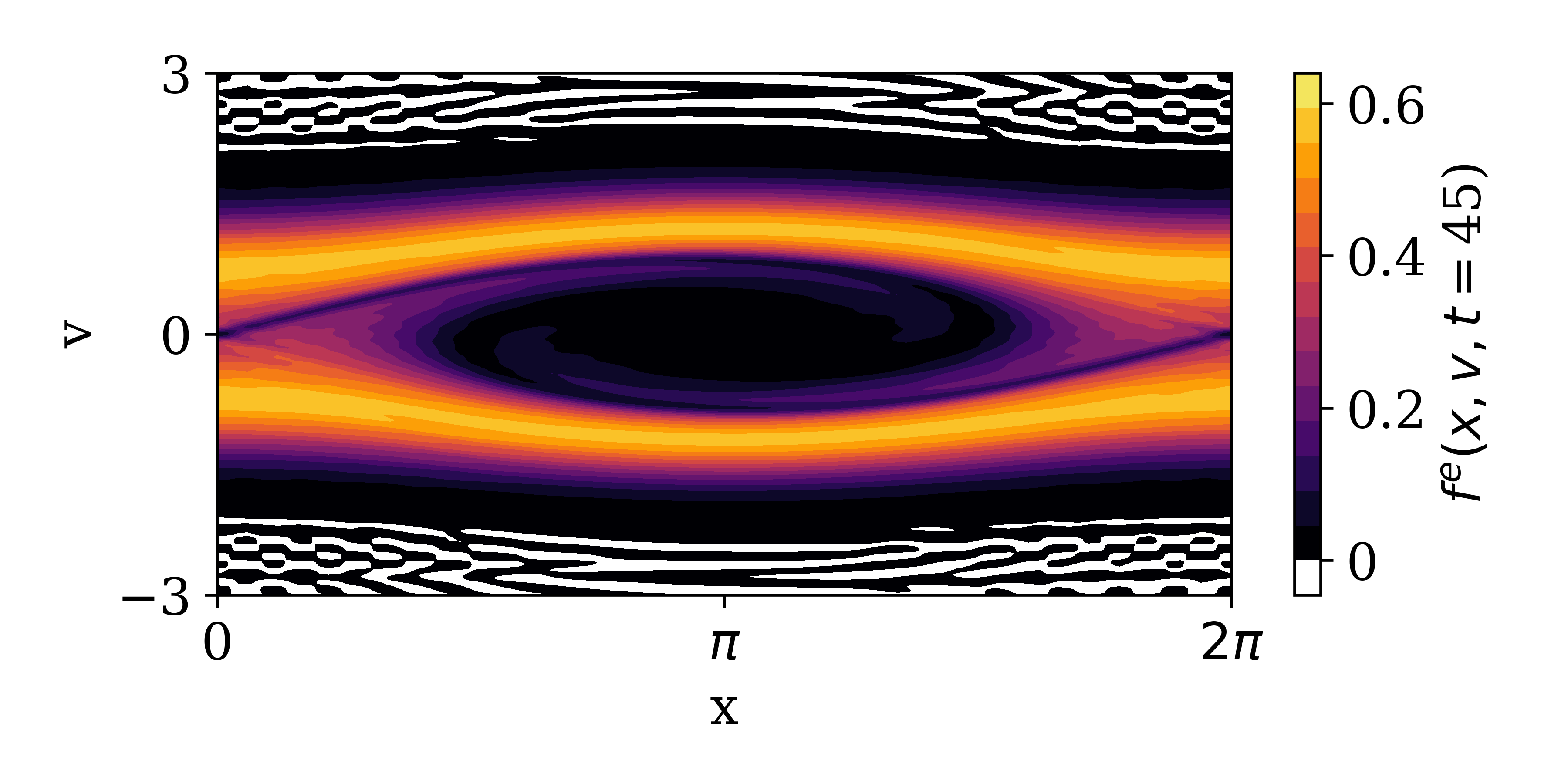}
         \label{fig:two-stream-sw-t-45}
     \end{subfigure}
     \hspace{-10pt}
    \begin{subfigure}[b]{0.5\textwidth}
         \centering
         \caption{SW square-root formulation}
         \includegraphics[width=\textwidth]{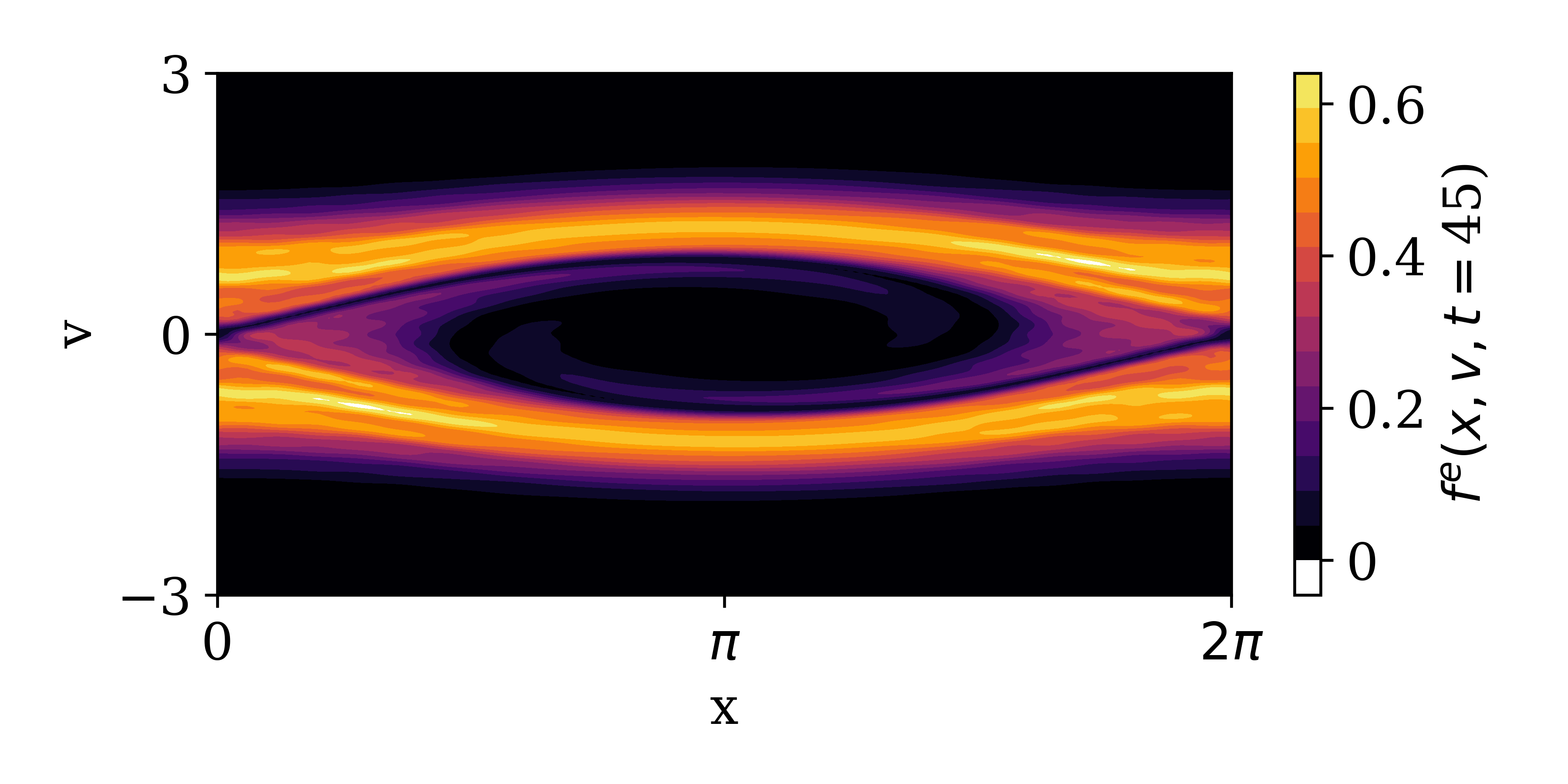}
         \label{fig:two-stream-sw-sqrt-t-45}
     \end{subfigure}
    \caption{The two-stream instability electron distribution function $f^{e}(x, v, t=45)$ with $N_{v}=101$ obtained by the (a)~SW and the (b)~SW square-root formulations. The SW formulation becomes negative in certain regions in phase space, yet such regions are far from the region of interest, i.e. the phase space vortex.}
    \label{fig:two-stream-electron-distribution-t-45}
\end{figure}

\begin{figure}
    \centering
        \begin{subfigure}[b]{0.5\textwidth}
         \centering
         \caption{Two-stream instability filamentation}
         \includegraphics[width=\textwidth]{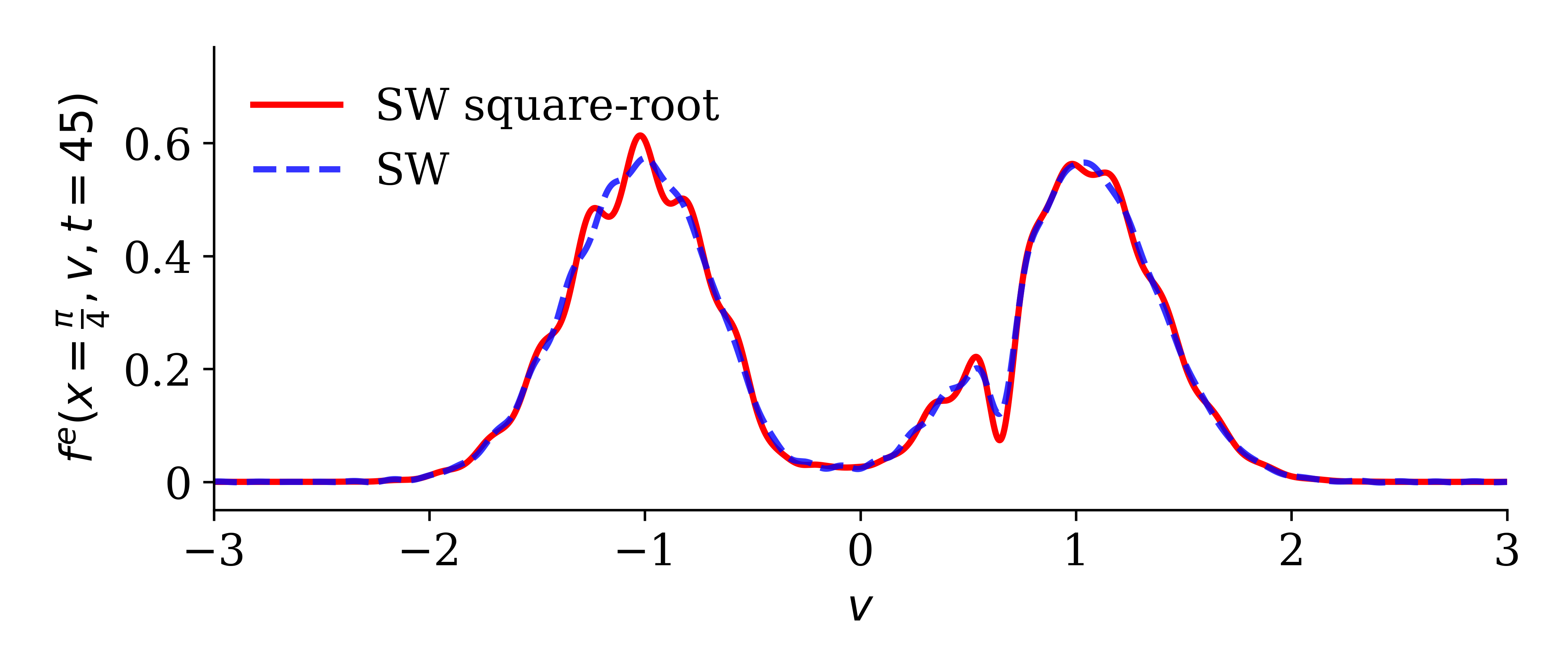}
         \label{fig:filamentation-two-stream}
     \end{subfigure}
    \hspace{-10pt}
    \begin{subfigure}[b]{0.5\textwidth}
         \centering
         \caption{Two-stream instability growth rate }
         \includegraphics[width=\textwidth]{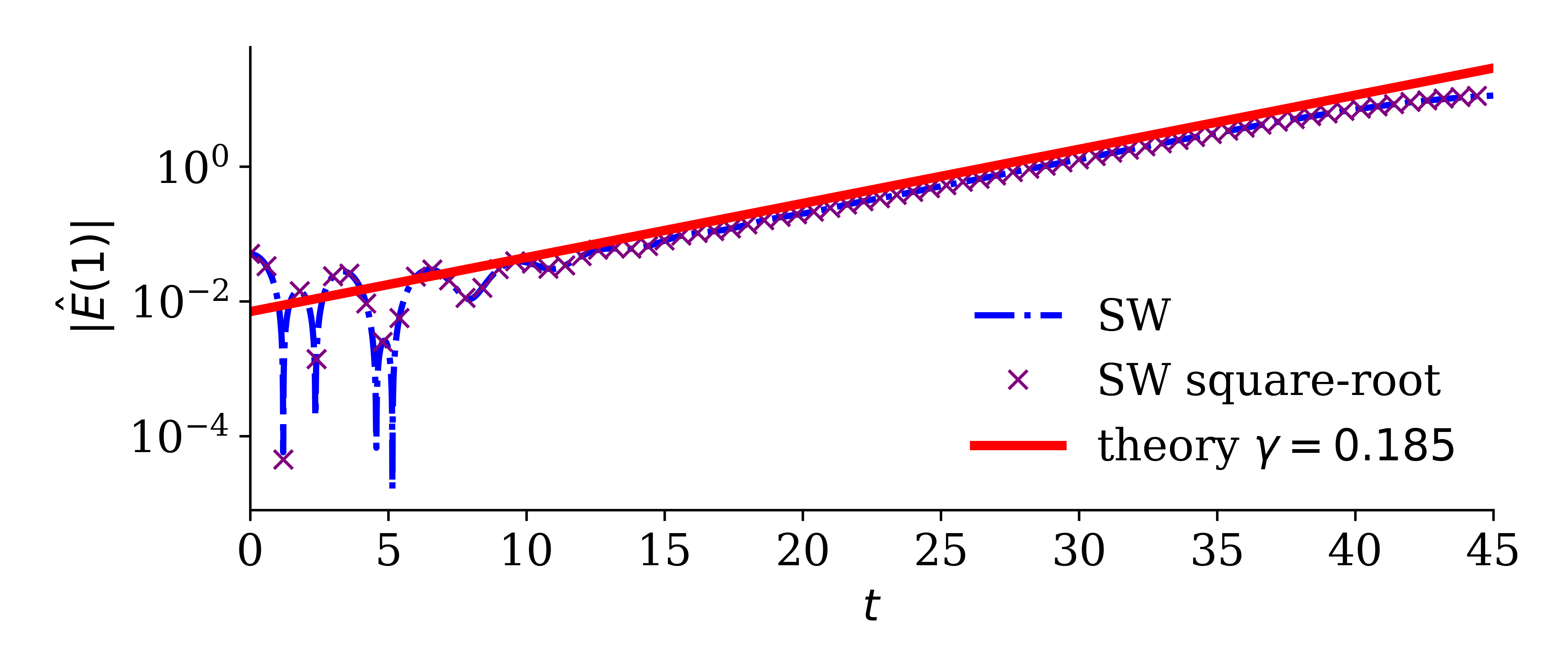}
         \label{fig:growth-rate-two-stream}
    \end{subfigure}
    \caption{Subfigure (a) shows a cross-section of the two-stream instability electron particle distribution function $f^{e}(x=\frac{\pi}{4}, v, t=45)$ via the two formulations: SW and SW square-root (with $N_{v}=101$). The SW square-root formulation appears to have more filamentation near the peak regions. Subfigure (b) presents the two-stream instability electric field amplitude as a function of time computed via the SW and SW square-root formulation with $N_{v}=101$. The electric field growth rate of the SW and SW square-root simulations agree with the theoretical growth rate~\cite[\S 3]{gary_1993_theory}.}
\end{figure}

\begin{figure}
    \centering
     \begin{subfigure}[b]{0.49\textwidth}
         \centering
         \caption{SW formulation $N_{v}=101$ (odd)}
         \includegraphics[width=\textwidth]{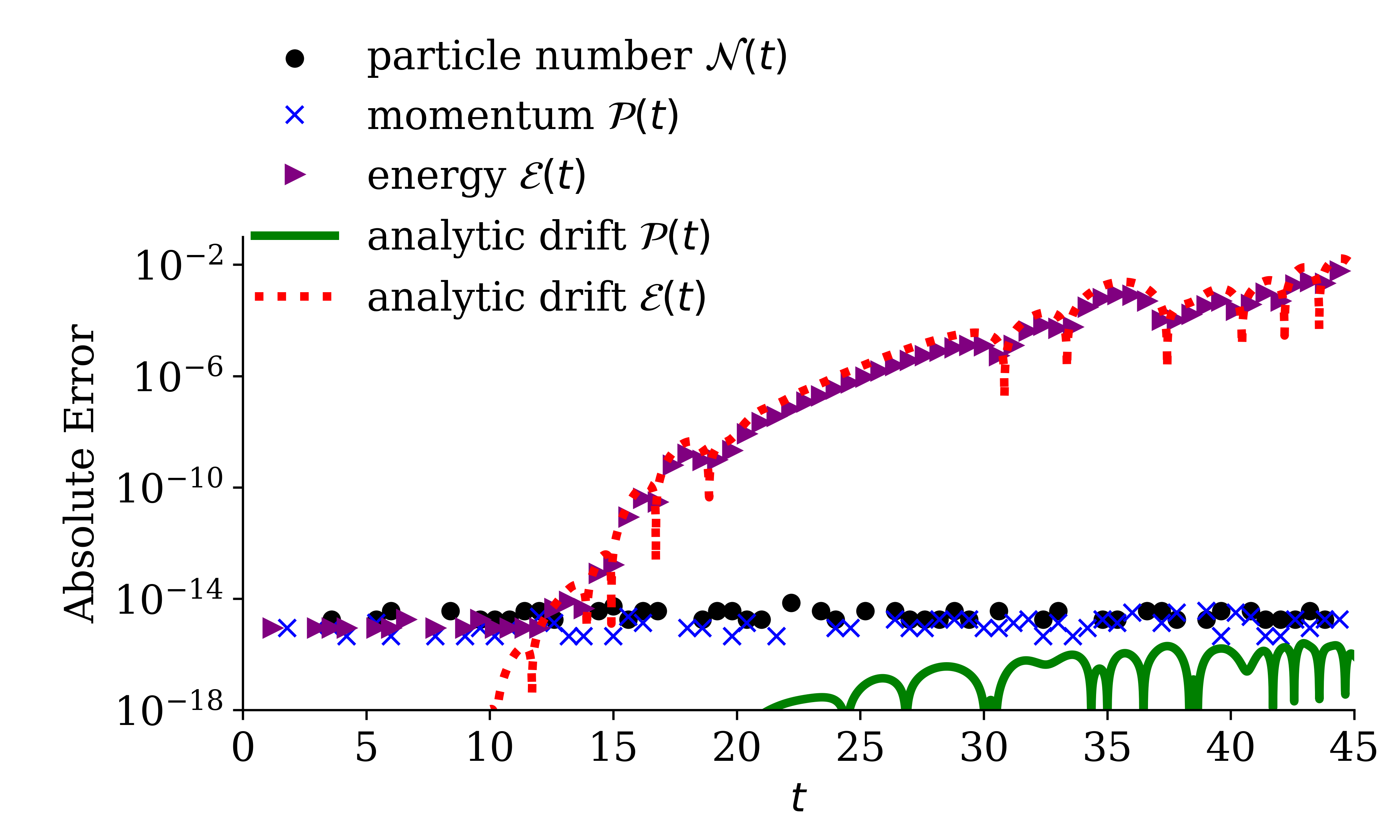}
         \label{fig:conservation-two-stream-sw-101}
     \end{subfigure}
    \begin{subfigure}[b]{0.49\textwidth}
         \centering
         \caption{SW formulation $N_{v}=100$ (even)}
         \includegraphics[width=\textwidth]{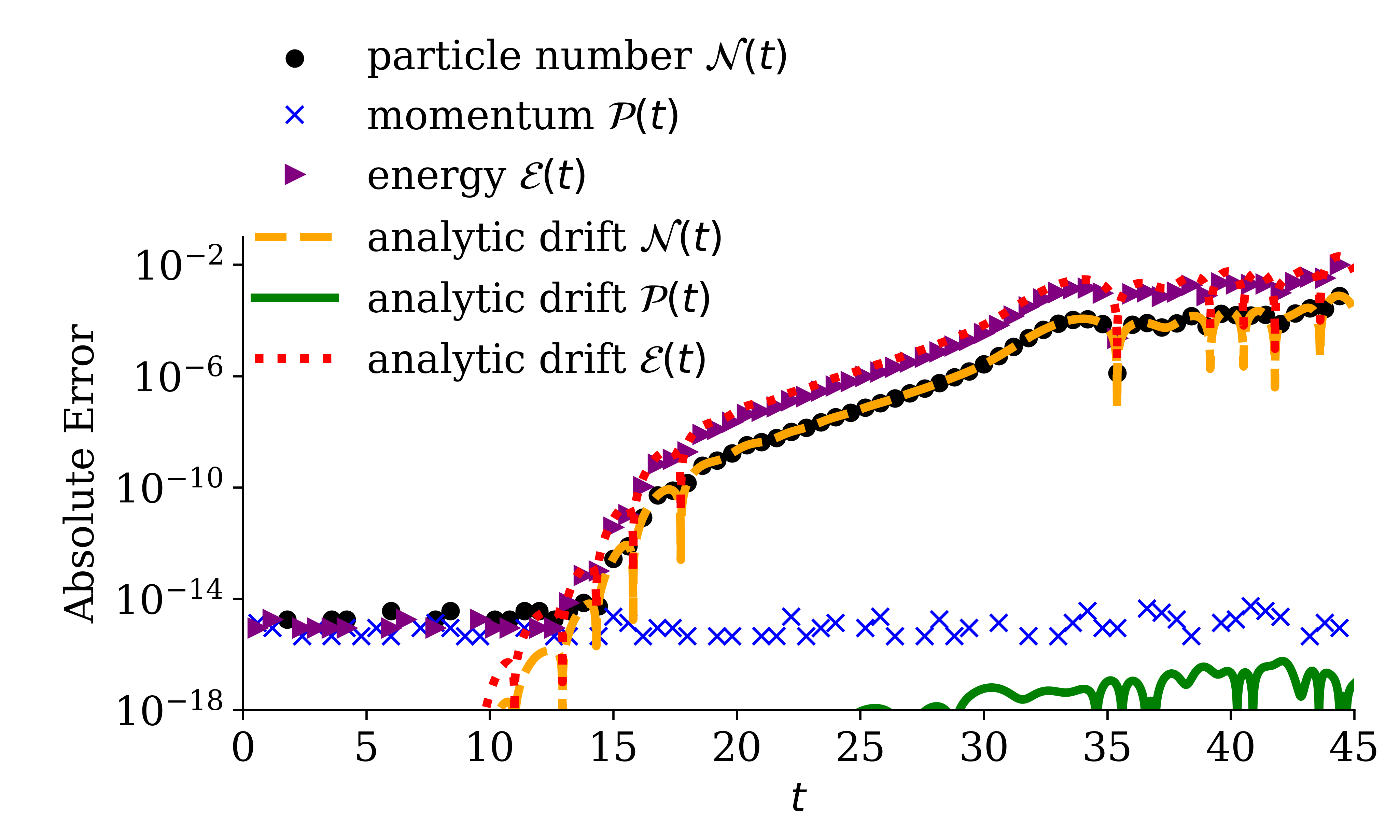}
         \label{fig:conservation-two-stream-sw-100}
     \end{subfigure}
    \begin{subfigure}[b]{0.49\textwidth}
         \centering
         \caption{SW square-root formulation $N_{v}=101$ (odd)}
         \includegraphics[width=\textwidth]{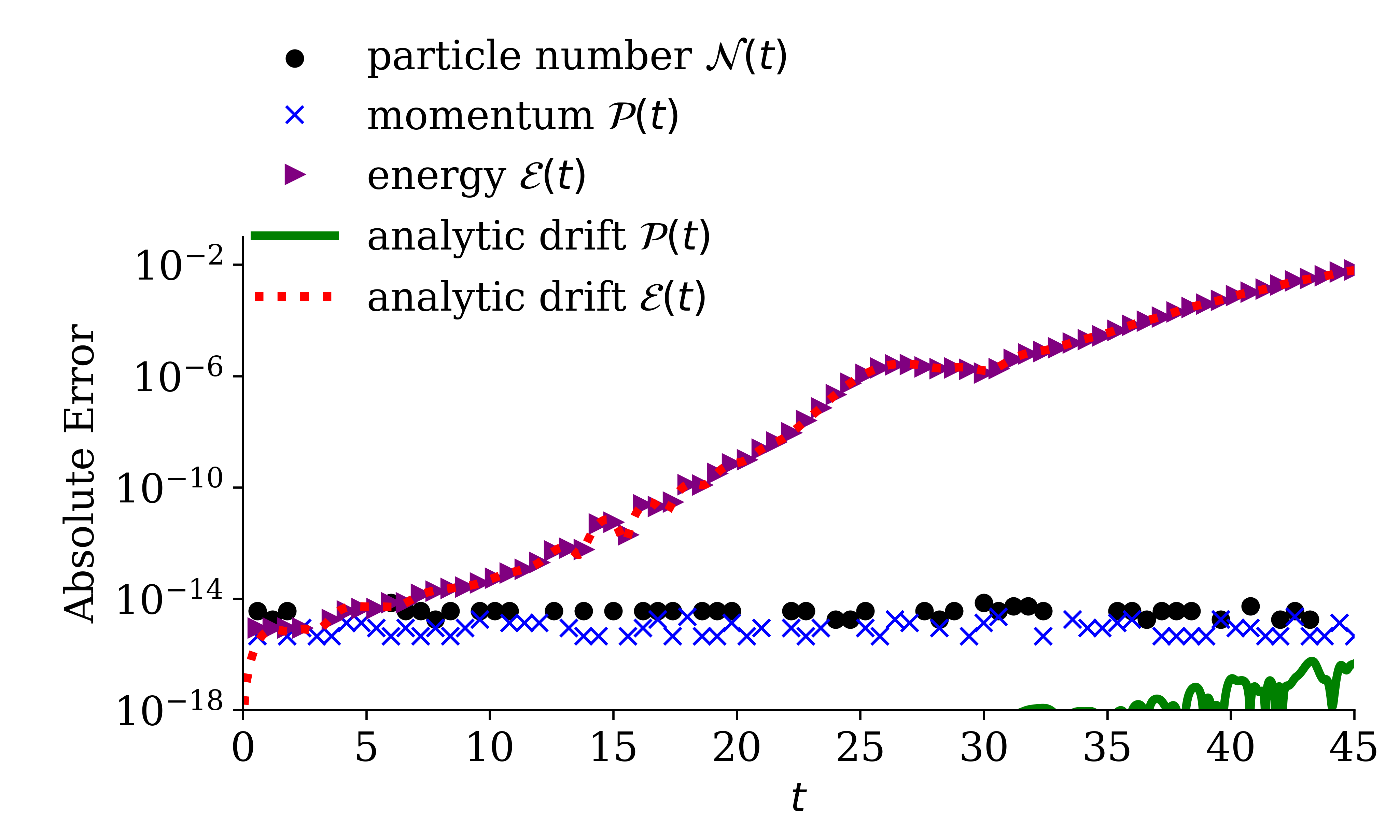}
         \label{fig:conservation-two-stream-sw-sqrt-101}
     \end{subfigure}
    \begin{subfigure}[b]{0.49\textwidth}
         \centering
         \caption{SW square-root formulation $N_v=100$ (even)}
         \includegraphics[width=\textwidth]{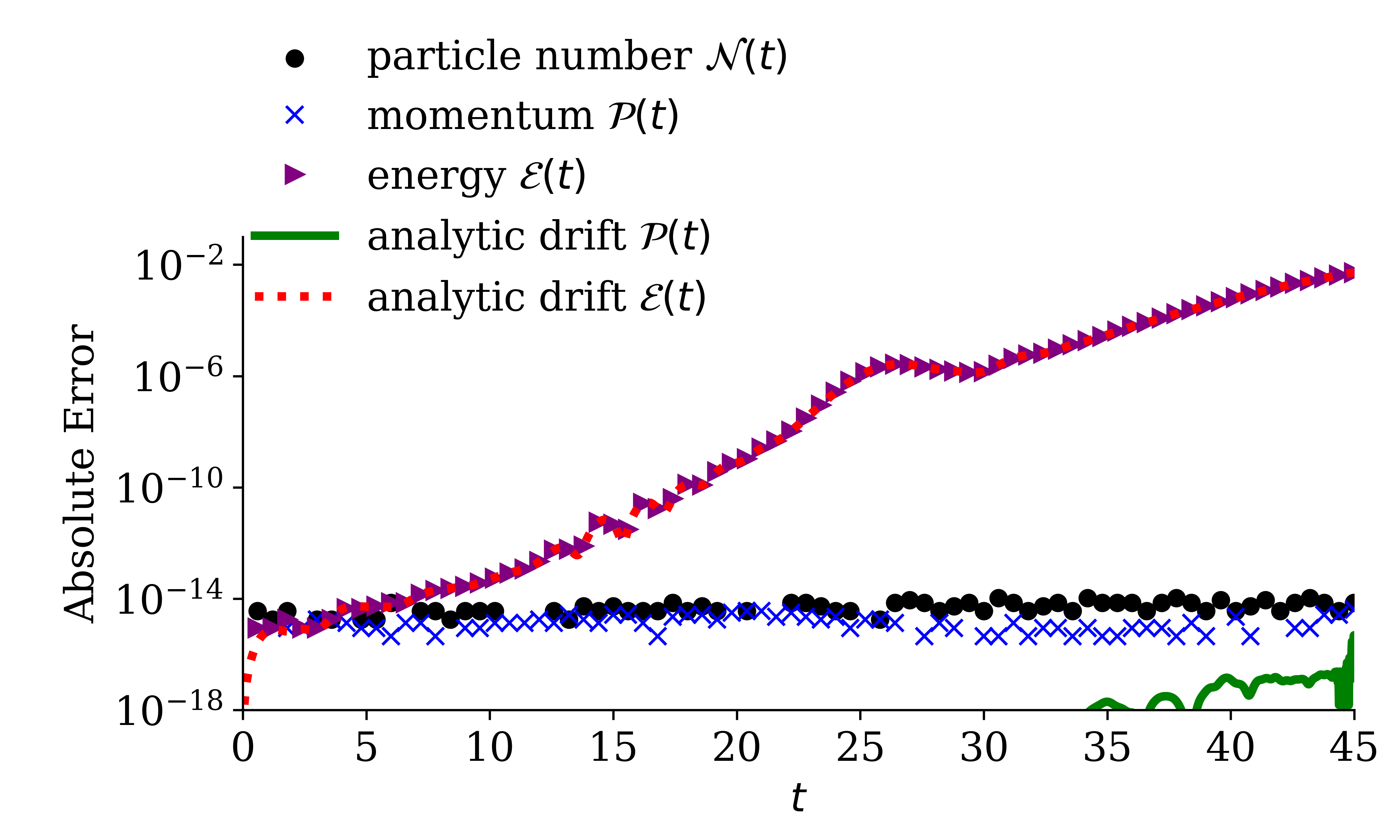}
         \label{fig:conservation-two-stream-sw-sqrt-100}
     \end{subfigure}
    \caption{Same as Figure~\ref{fig:conservation-linear-landau} for the two-stream instability. Due to the phase space symmetry in the two-stream instability, momentum is approximately conserved in both formulations. }
    \label{fig:conservation-two-stream}
\end{figure}

\subsection{Bump-on-tail Instability}\label{sec:bump-on-tail-numerical-results}
The bump-on-tail instability is an asymmetric velocity-space instability. Figure~\ref{fig:bump-on-tail-snapshots} shows the SW and SW square-root formulations distribution function evolution. The SW formulation results become negative at $t=10$ and at $t=20$ the distribution function becomes negative in the the majority of phase space along with the phase space vortex between the two beams. Figure~\ref{fig:filamentation-bump-on-tail} shows a cross-section of the electron distribution function at $t=20$, indicating that the SW and SW square-root formulations suffer from comparable filamentation effects, manifesting as artificial oscillations of the distribution function. Figure~\ref{fig:bump-on-tail-exchange-of-energy} shows the kinetic and potential energy exchange from the two formulations. Such macroscopic quantities are nearly identical for the two formulations. 

The conservation of particle number, momentum, and energy is shown in Figure~\ref{fig:conservation-bump-on-tail}. The numerical drift rates match the analytically derived drift rates for both formulations. Similar to the two-stream instability, the velocity shifting parameter $u_{e}\neq0$, thus only particle number is conserved in the SW formulation if the number of spectral terms is odd. The numerical drift rates for both formulations are comparable in magnitude. The momentum drift rate for the SW formulation (with $N_{v}=100$) and SW square-root formulations is larger than the two-stream instability drift rate due to the asymmetry in the velocity space of the bump-on-tail test case. 

\begin{figure}
    \centering
     \begin{subfigure}[b]{0.49\textwidth}
         \centering
         \caption{SW formulation}
         \includegraphics[width=\textwidth]{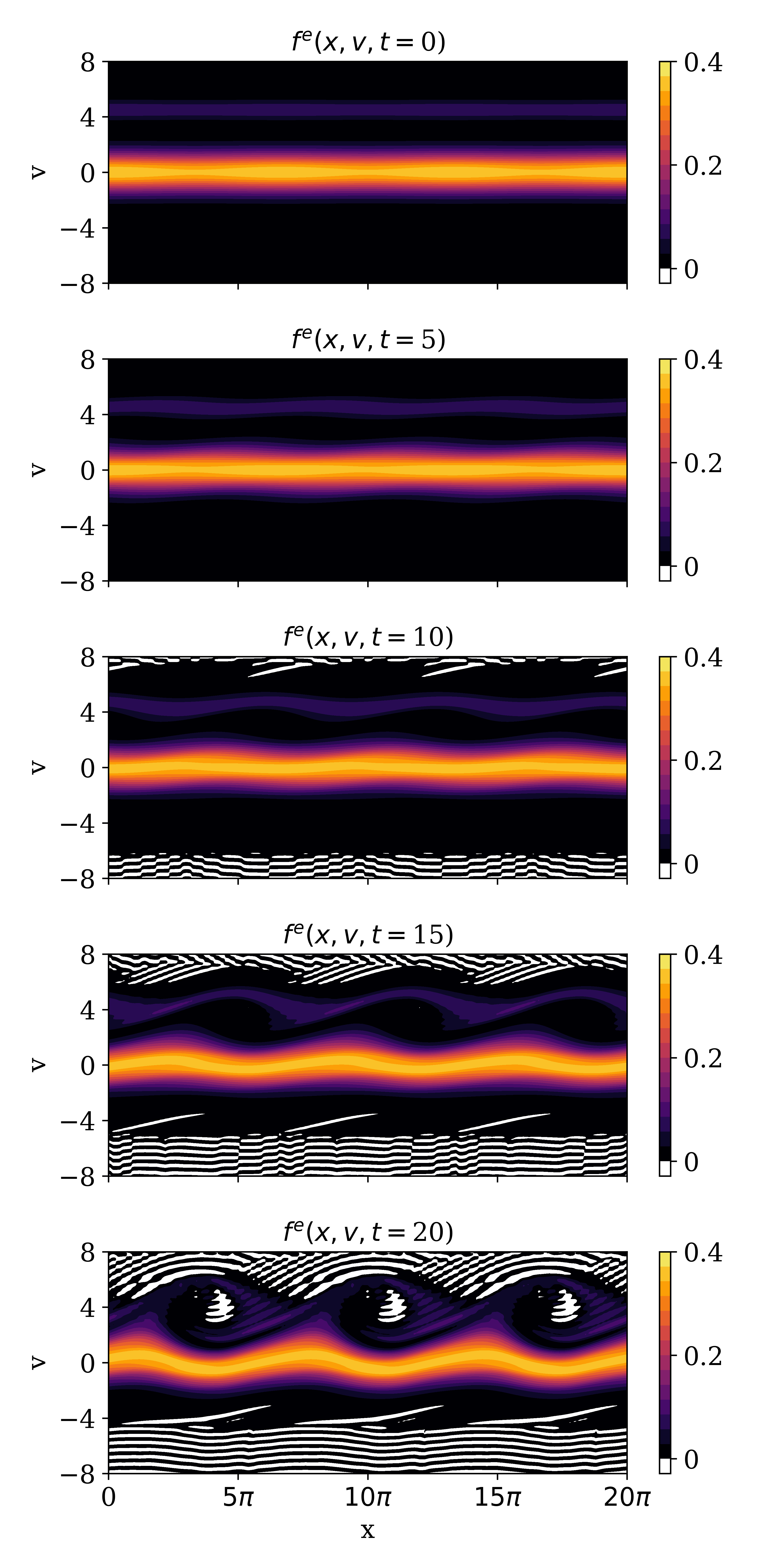}
         \label{fig:bump-on-tail-evolution-sw}
     \end{subfigure}
    \begin{subfigure}[b]{0.49\textwidth}
         \centering
         \caption{SW square-root formulation}
         \includegraphics[width=\textwidth]{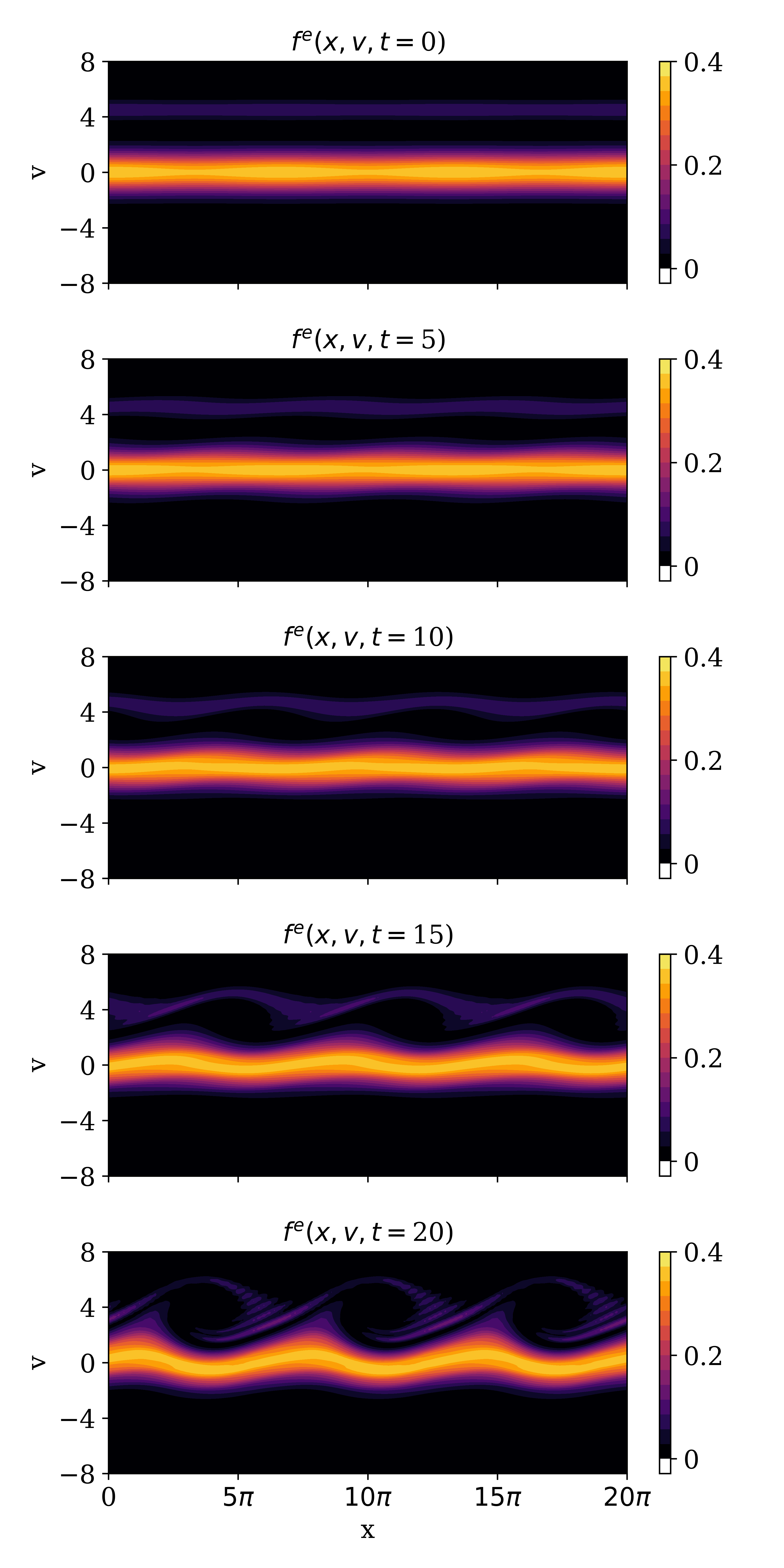}
         \label{fig:bump-on-tail-evolution-sw-sqrt}
     \end{subfigure}
    \caption{The evolution of the electron distribution function $f^{e}(x, v, t)$ for the bump-on-tail instability with $N_{v}=100$. The distribution function becomes negative in the (a)~SW formulation in the majority of phase space whereas the (b)~SW square-root formulation is positivity preserving by construction. }
    \label{fig:bump-on-tail-snapshots}
\end{figure}

\begin{figure}
    \centering
         \begin{subfigure}[b]{0.49\textwidth}
         \centering
         \caption{Bump-on-tail instability filamentation}
         \includegraphics[width=\textwidth]{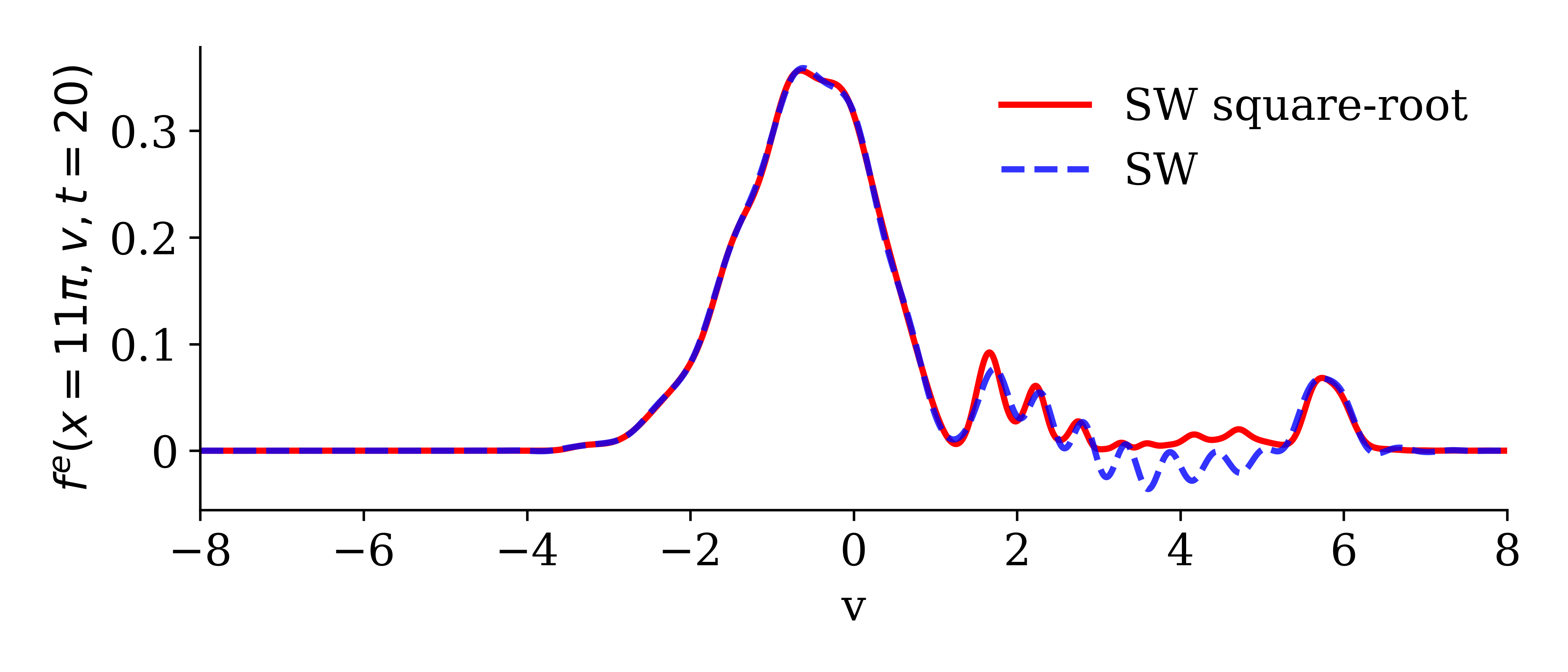}
         \label{fig:filamentation-bump-on-tail}
     \end{subfigure}
    \begin{subfigure}[b]{0.49\textwidth}
         \centering
         \caption{Bump-on-tail instability energy exchange}
         \includegraphics[width=\textwidth]{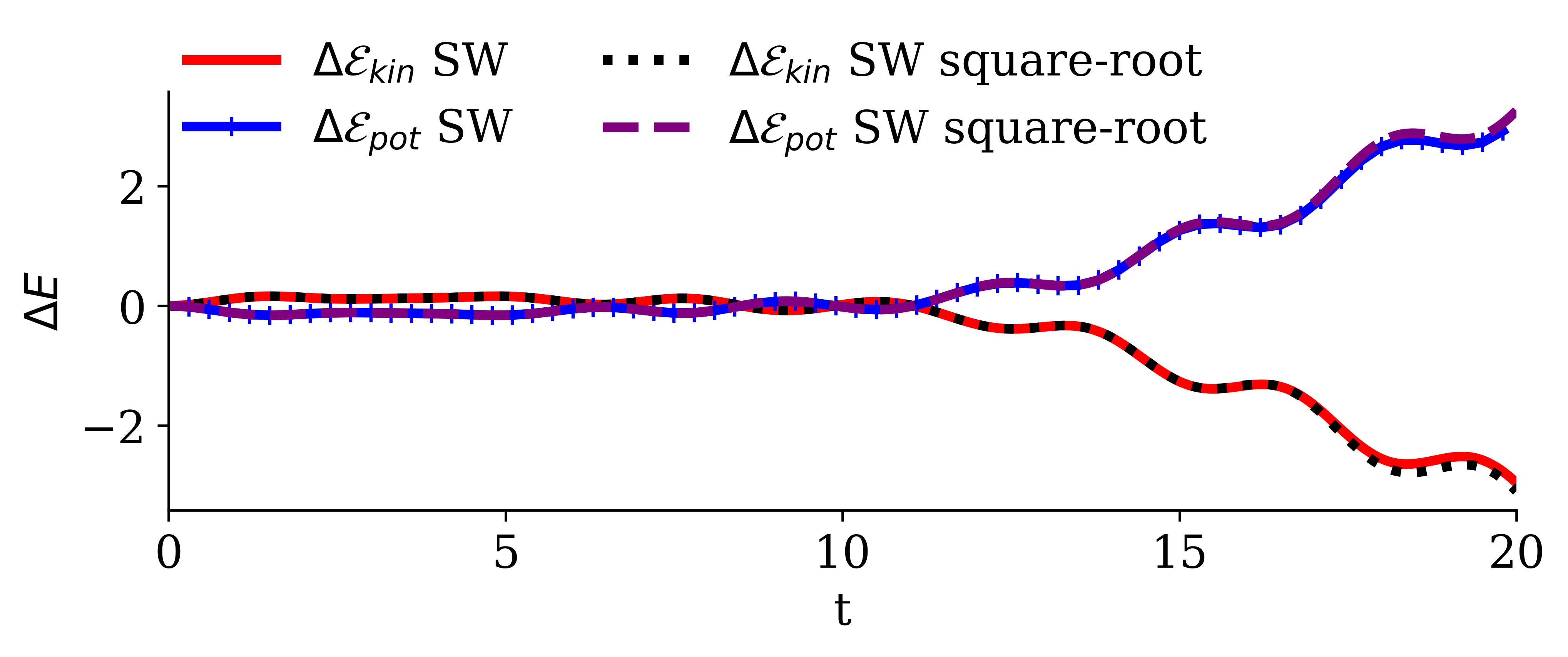}
         \label{fig:bump-on-tail-exchange-of-energy}
     \end{subfigure}
    \caption{Subfigure (a) shows a cross-section of the bump-on-tail electron particle distribution function $f^{e}(x=11\pi, v, t=20)$ via the two formulations: SW and SW square-root (with $N_{v}=100$). The filamentation is relatively comparable in the SW and SW square-root formulations. Subfigure (b) shows the evolution of the transferred kinetic and potential energy for the bump-on-tail instability with $N_{v}=100$. The SW and SW square-root macroscopic quantities--kinetic and potential energy-- are nearly identical. }
\end{figure}

\begin{figure}
    \centering
     \begin{subfigure}[b]{0.49\textwidth}
         \centering
         \caption{SW formulation $N_{v}=101$ (odd)}
         \includegraphics[width=\textwidth]{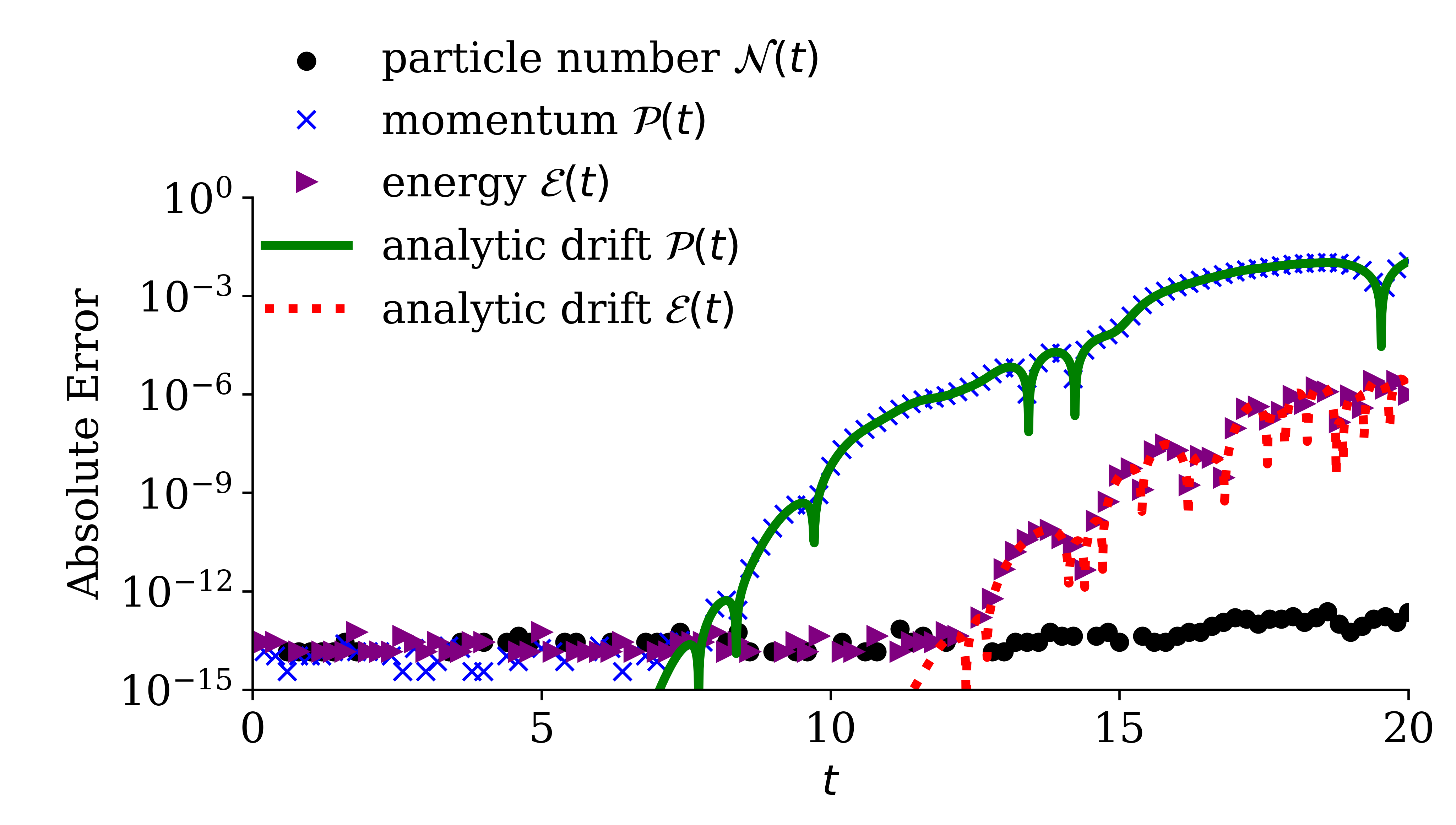}
         \label{fig:conservation-bump=on-tail-sw-101}
     \end{subfigure}
    \begin{subfigure}[b]{0.49\textwidth}
         \centering
         \caption{SW formulation $N_{v}=100$ (even)}
         \includegraphics[width=\textwidth]{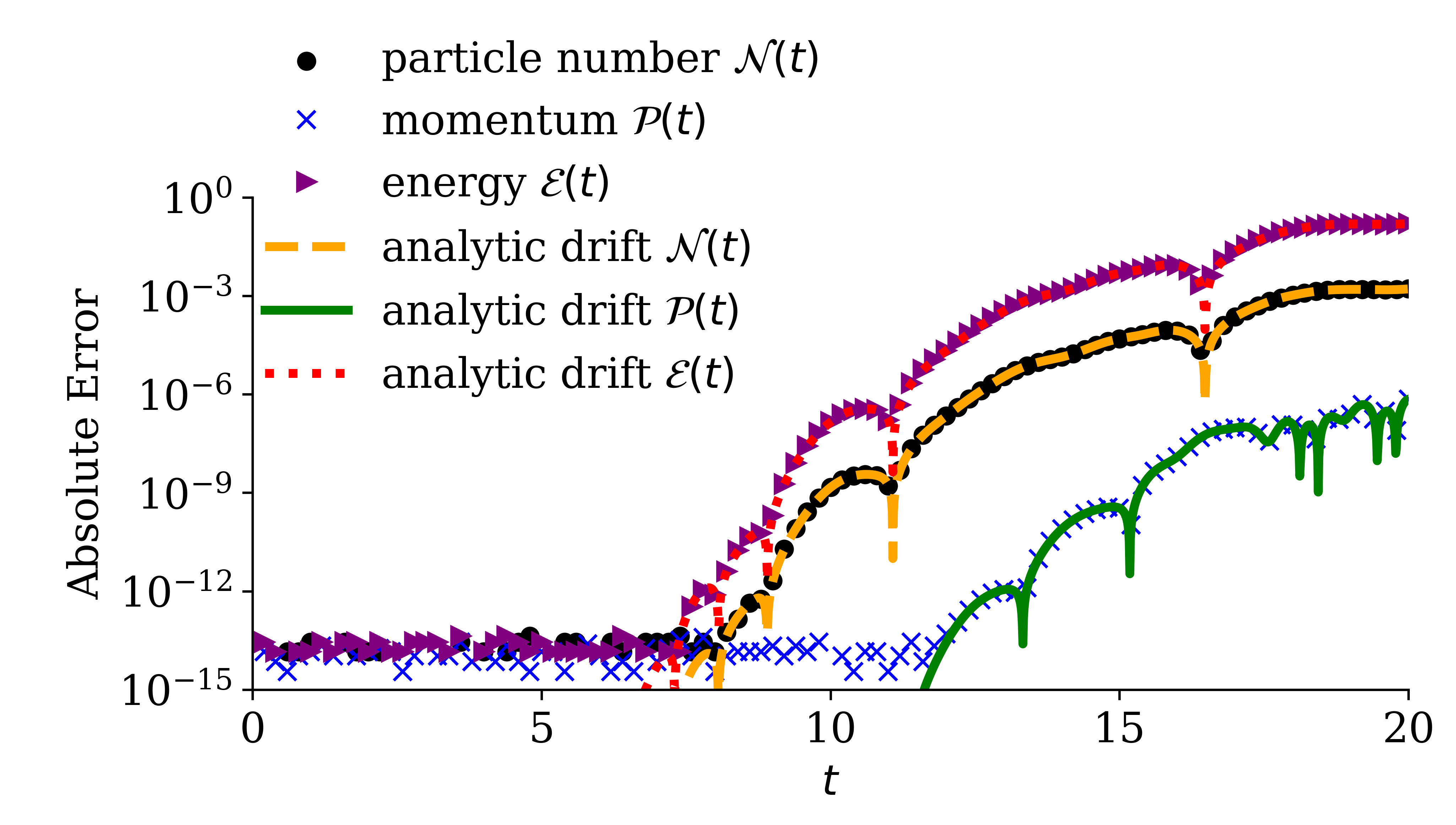}
         \label{fig:conservation-bump-on-tail-sw-100}
     \end{subfigure}
    \begin{subfigure}[b]{0.49\textwidth}
         \centering
         \caption{SW square-root formulation $N_{v}=101$ (odd)}
         \includegraphics[width=\textwidth]{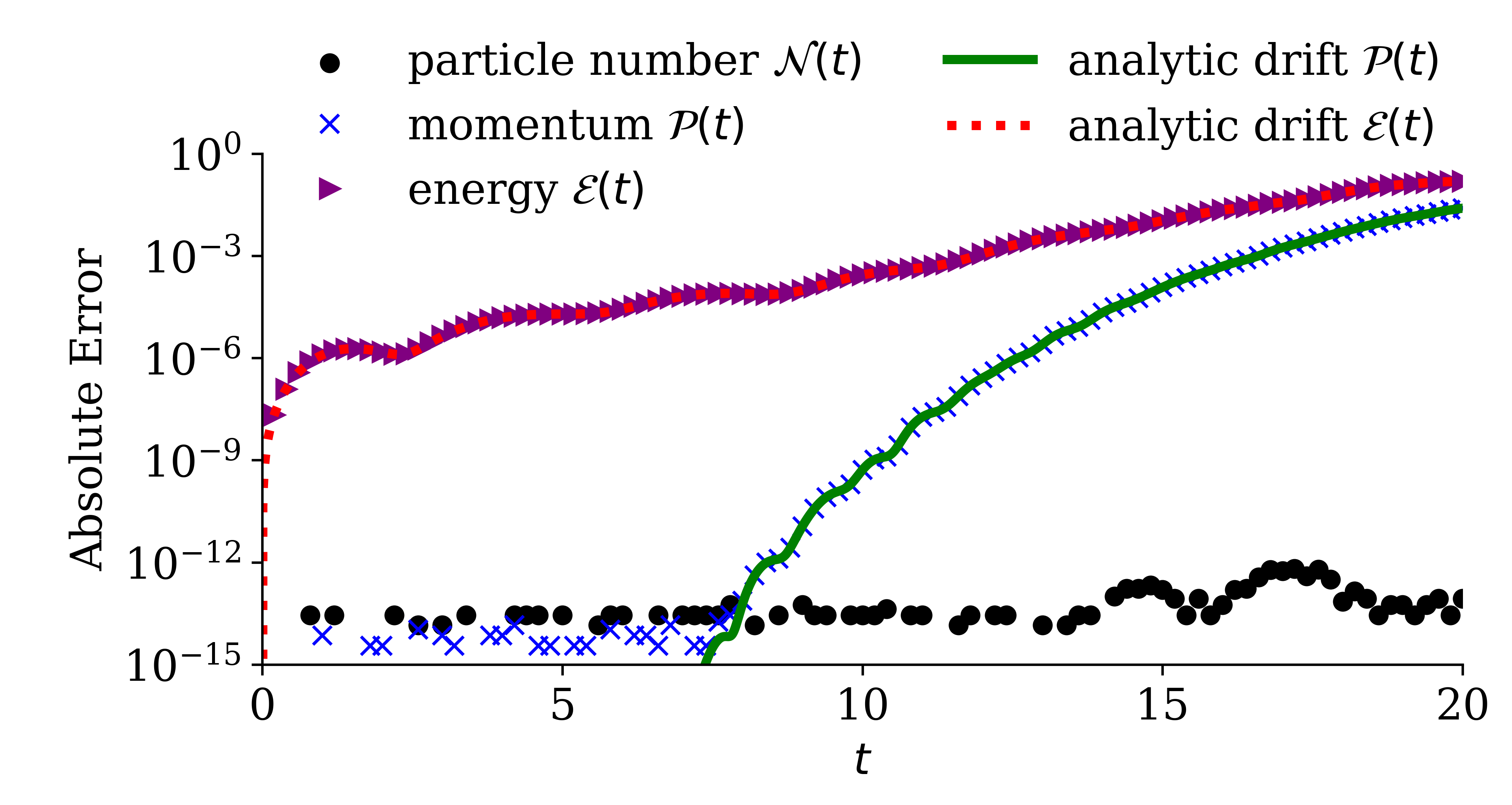}
         \label{fig:conservation-bump-on-tail-sw-sqrt-101}
     \end{subfigure}
    \begin{subfigure}[b]{0.49\textwidth}
         \centering
         \caption{SW square-root formulation $N_v=100$ (even)}
         \includegraphics[width=\textwidth]{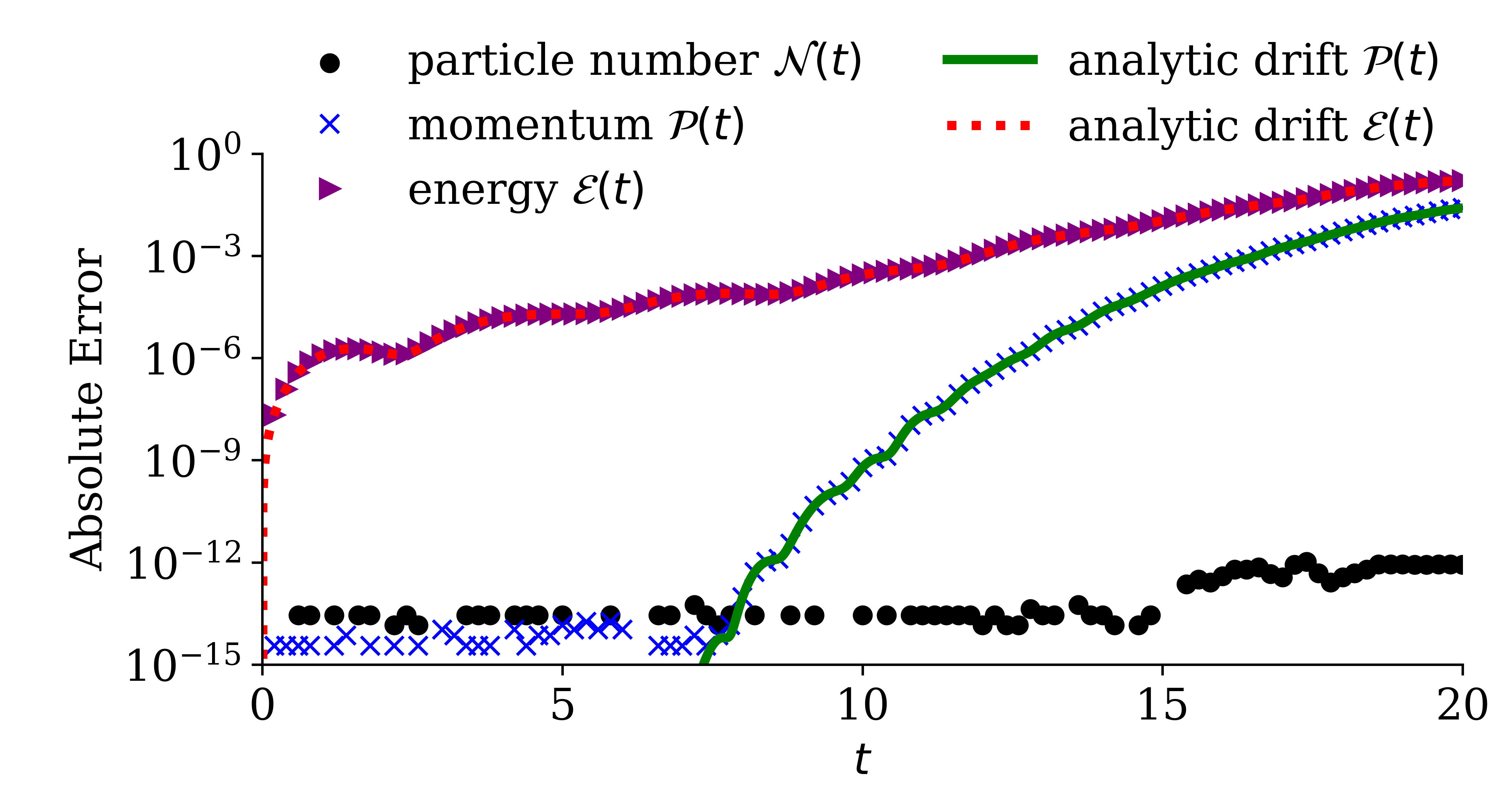}
     \end{subfigure}
    \caption{Same as Figure~\ref{fig:conservation-linear-landau} for the bump-on-tail instability. Due to the asymmetry in the velocity space in the bump-on-tail instability, the momentum drift in both formulations is more pronounced than in the previous symmetric test cases. }
    \label{fig:conservation-bump-on-tail}
\end{figure}

\subsection{Ion Acoustic Wave} \label{sec:ion-acoustic-numerical-results}
The ion acoustic wave is a multi-scale physics benchmark problem driven by electron thermal pressure and ion inertia. The ion-acoustic wave dispersion relation resembles an ordinary sound wave\footnote{
When the ion acoustic wave has a wavelength much longer than the Debye length $k \lambda_{D} \ll 1$ its dispersion relation is given by $\omega^2 = c_{s}^2 k^2$, where $c_{s} \coloneqq \sqrt{k_{B}(T^{e} + 3T^{i})/m^i}$ is the ion acoustic wave speed~\cite[\S 2]{gary_1993_theory}.}. We perturb the initial ion and electron distributions, such that
\begin{equation}
    f^{s}(x, v, t=0) = \frac{1}{\sqrt{2\pi} \alpha^{s}} \left(1 + \epsilon \cos\left(\frac{2\pi}{L} x\right) \right)\exp\left(-\frac{v^2}{2(\alpha^s)^2}\right),
\end{equation}
where $s=\{`e$', `$i$'$\}$, $L=10, t_{f}=600$, and $\epsilon=0.01$. We set the ratio between electron and ion temperature to $T^e/T^i = 10$ and mass to $m^e/m^i = 1/1836$, resulting in $\alpha^e = 1$ and $\alpha^i = 1/135$. Figure~\ref{fig:density-ion-acoustic-mag} shows the magnitude of the ion density first harmonic of the SW and SW square-root formulations with $N_{v}=N_{x}=51$ and $\Delta t =0.05, 1$. Figures~\ref{fig:electron-density-ion-acoustic} and~\ref{fig:ion-density-ion-acoustic} show the electron and ion density evolution for the SW square-root formulation with $\Delta t = 1$. The results with different $\Delta t$ are almost identical, showing that the implicit solver can step over the fast electron frequency in the system. Moreover, the numerical period of the ion density is approximately~389 (normalized to the electron plasma frequency) for both formulations, which is close to the analytic linear estimate of~375. As we increase the time step $\Delta t$, the results become less accurate, but as expected, the method remains stable. Overall, the ion acoustic wave example highlights one of the advantages (aside from conservation properties) of using an implicit temporal integrator which allows one to step over temporal scales. 

\begin{figure}
\centering
\begin{subfigure}[b]{0.49\textwidth}
\centering
\caption{Ion density amplitude}
\label{fig:density-ion-acoustic-mag}
\includegraphics[width=\textwidth]{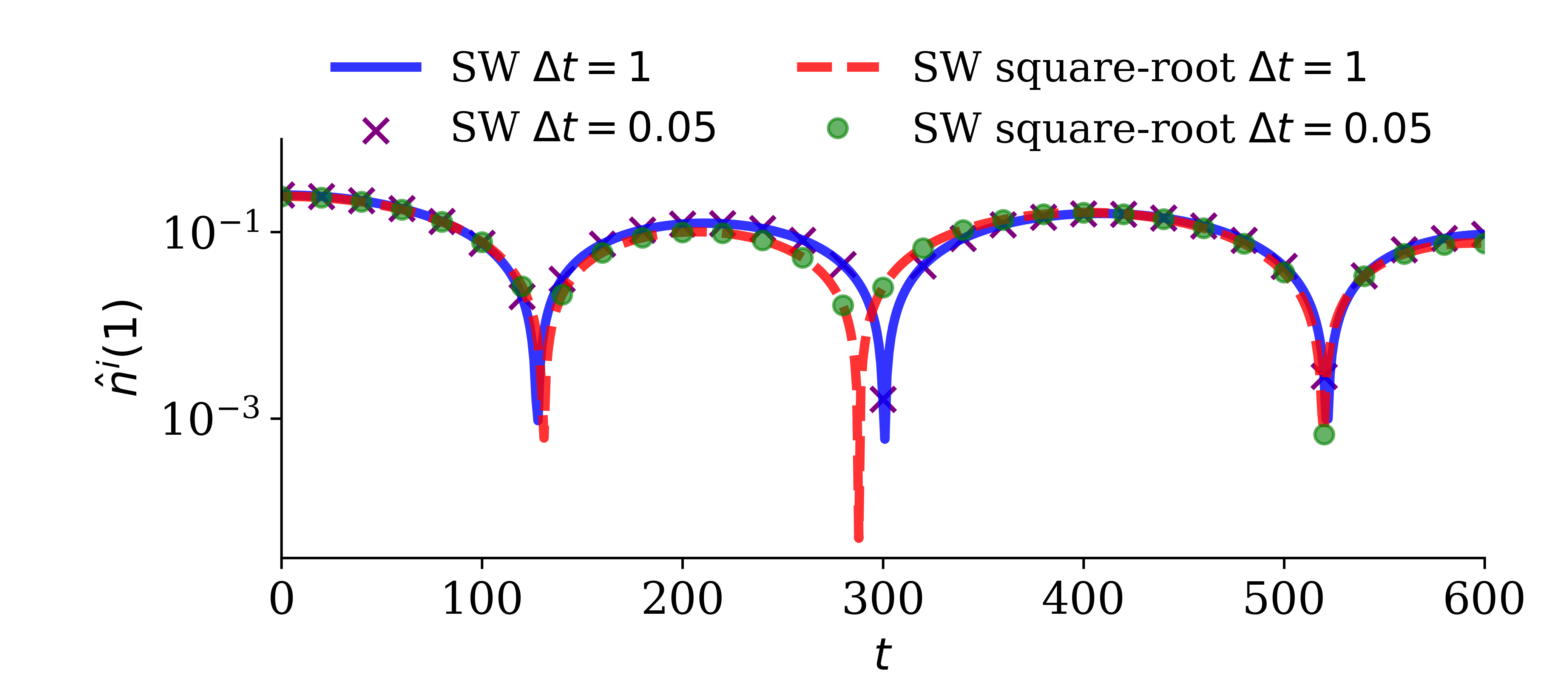}
\end{subfigure}
\hspace{-12pt}
\begin{subfigure}[b]{0.26\textwidth}
\centering
\caption{Electron density $n^{e}(x, t)$}
\label{fig:electron-density-ion-acoustic}
\includegraphics[width=\textwidth]{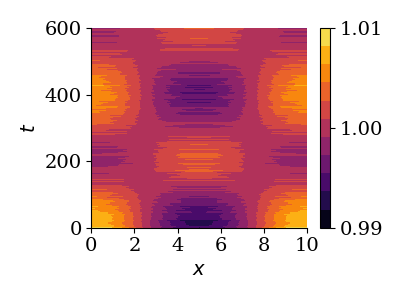}
\end{subfigure}
\hspace{-12pt}
\begin{subfigure}[b]{0.26\textwidth}
\centering
\caption{Ion density $n^{i}(x, t)$}
\label{fig:ion-density-ion-acoustic}
\includegraphics[width=\textwidth]{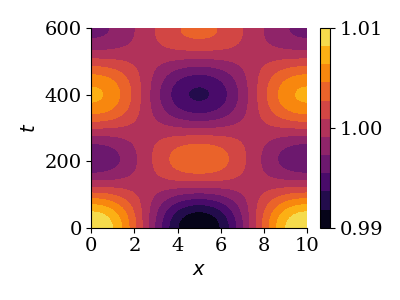}
\end{subfigure}
\caption{Subfigure (a) shows the magnitude of the ion density first harmonic of the SW and SW square-root formulations with $\Delta t = 0.05, 1$ and $N_{v}=N_{x} = 51$. Subfigure (b/c) shows the electron and ion density evolution for the SW square-root formulation with $\Delta t = 1$. The results show that the numerical scheme can successfully step over fast electron frequency in the system. }
\end{figure}

\section{Conclusions} \label{sec:conclusions}
We proposed the anti-symmetric discretization of the one-dimensional Vlasov-Poisson equations, which is based on the symmetrically-weighted Hermite spectral expansion in velocity, centered finite differencing in space, and an implicit Runge-Kutta temporal integrator. We showed that the method preserves the anti-symmetric structure of the advection operator (or equivalently the canonical Poisson bracket) of the Vlasov equation, which results in a stable method. The anti-symmetric discretization is applied to two formulations: SW and SW square-root formulation. The former discretizes the standard Vlasov-Poisson equations and the latter discretizes a continuous square-root variable transformation of the governing equations. The SW square-root formulation is positivity preserving by construction. We derived the conservation properties of the two formulations. The SW formulation conserves particle number and energy if the number of velocity spectral terms is odd and momentum if the number of spectral terms is even. Moreover, the conservation of momentum and energy only holds in the SW expansion if the velocity shifting parameter is set to zero. Alternatively, the SW square-root formulation conserves particle number independently of the number of spectral terms and velocity parameterization. The conservation drift rates for both formulations (besides the total energy drift in the SW square-root formulation) only depend on the last coefficient in the Hermite spectral expansion, which highlights that increasing the number of velocity spectral terms improves conservation properties and thus the validity of long-time simulations. Overall, this study introduces the construction of structure-preserving and stable Vlasov-Poisson solvers. 

A future goal of ours is to study techniques to mitigate filamentation. Typical filamentation strategies include (1)~artificial collisions, (2)~filtering, and (3)~flux-limiter techniques. However, applying such standard strategies to the SW and SW square-root formulations will break conservation laws and the method's anti-symmetric structure since its invariants depend on all the coefficients in the expansion. We plan to study such existing strategies and their influence on the SW and SW square-root formulations in the future. Another avenue of future work involves the adaptivity in time and space of the Hermite parameters $u^{s}$ and $\alpha^{s}$ in Eq.~\eqref{hermite-basis-function-definition}, which has been derived for the AW formulation with time-dependent parameters in~\cite{pagliantini_2023_adaptive}. We plan to investigate such adaptive algorithms in the SW and SW square-root framework and their impact on spectral convergence, stability, conservation, and anti-symmetry. 

\newpage
\appendix
\section{Anti-symmetric Representation via a Fourier Spectral Expansion in Space} \label{sec:Appendix-A}
We derive the Hermite-Fourier spectral discretization of the SW and SW square-root formulations and show that a Fourier discretization in space preserves the anti-symmetric structure of the advection operator in the Vlasov equation~\eqref{advection-operator-continuum}. We begin by introducing the Fourier basis functions, such that
\begin{equation}\label{fourier-basis-function-definition}
    \eta_{k}(x) \coloneqq \exp \left(\frac{2\pi i k x}{\ell}\right), \qquad \forall k \in \mathbb{N}_{0}.
\end{equation}
Since the spatial domain $\Omega_{x} \coloneqq [0, \ell]$ is bounded, the Fourier basis functions~$\{\eta_{k}(x)\}_{k \in \mathbb{N}_{0}}$ form an orthogonal basis with the following orthogonality relation 
\begin{equation}\label{orthogonal-property-fourier}
    \frac{1}{\ell} \int_{0}^{\ell} \eta_{k}(x) \eta_{m}(x) \mathrm{d} x= \delta_{k+m, 0}, \qquad \forall k, m \in \mathbb{Z}.
\end{equation}
We approximate the Hermite expansion coefficient $C_{n}^{s} \in \mathcal{W}$ in Eqns.~\eqref{expansion-seperation-of-variables} and~\eqref{spectral-approximation-sqrt} and the electric field $E \in \mathcal{W}$ in Eqns.~\eqref{vlasov-continuum}-\eqref{poisson-continuum} via a spatial Fourier expansion, i.e. 
\begin{equation*}
    C_{n}^{s}(x, t) \approx C_{n}^{s, N_{x}}(x, t) = \sum_{k=-N_{x}}^{N_{x}} C_{n, k}^{s}(t) \eta_{k}(x) \qquad \text{and} \qquad  E(x, t) \approx E^{N_{x}}(x, t) = \sum_{k=-N_{x}}^{N_{x}} E_{k}(t) \eta_{k}(x),
\end{equation*}
where the Fourier basis is defined in Eq.~\eqref{fourier-basis-function-definition}. The product $E^{N_{x}}(x, t)C_{n}^{s, N_{x}}(x, t)$ can be described as a Fourier convolution, such that
\begin{equation*}
    E^{N_{x}}(x, t) C_{n}^{s, N_{x}}(x, t) = \sum_{k=-2N_{x}}^{2N_{x}} \eta_{k}(x) [\mathbf{E}(t) \ast \mathbf{C}_{n}^{s}(t)][k],
\end{equation*}
where 
\begin{align*}
    \mathbf{E}(t) &= [E_{-N_{x}}(t), E_{-N_{x}+1}(t), \ldots,  E_{0}(t), \ldots, E_{N_{x}-1}(t), E_{N_{x}}(t)]^{\top} \in \mathbb{C}^{2N_{x}+1},\\
    \mathbf{C}_{n}(t) &= [C_{n, -N_{x}}^{s}(t), C_{n, -N_{x}+1}^{s}(t), \ldots,  C_{n, 0}^{s}(t), \ldots, C_{n, N_{x}-1}^{s}(t), C^{s}_{n, N_{x}}(t)]^{\top} \in \mathbb{C}^{2N_{x}+1},
\end{align*}
and $\ast$ denotes the following convolution 
\begin{equation}\label{convolution-definition}
    [\mathbf{E}(t) \ast \mathbf{C}_{n}^{s}(t)][k] = \sum_{j=-N_{x}}^{N_{x}} E_{k-j}(t) C_{n, j}^{s}(t)
\end{equation}
Multiplying the Vlasov equation~\eqref{f-vlasov-pde} by $\eta_{-m}(x)$, integrating over the spatial coordinate $x$, and using the orthogonality relation of the Fourier basis functions in Eq.~\eqref{orthogonal-property-fourier}, we get 
\begin{align}\label{f-vlasov-ode-fourier}
    \frac{\mathrm{d}}{\mathrm{d} t} \mathbf{C}_{n}^{s}(t) &=- \mathbf{D} \left(\alpha^{s} \sqrt{\frac{n}{{2}}} \mathbf{C}_{n-1}^{s}(t) + \alpha^{s} \sqrt{\frac{n+1}{2}} \mathbf{C}_{n+1}^{s}(t) + u^{s}  \mathbf{C}_{n}^{s}(t)\right)\\
    &+\frac{q^{s}}{m^{s}\alpha^{s}} \mathbf{E}_{\mathrm{conv}}(t)\left(\sqrt{\frac{n}{2}} \mathbf{C}_{n-1}^{s}(t) - \sqrt{\frac{n+1}{2}} \mathbf{C}_{n+1}^{s}(t)\right), \nonumber
\end{align}
where 
\begin{equation}\label{definiton-D-derivative}
    \mathbf{D} = \frac{2\pi i}{\ell} \text{diag}([-N_{x}, -N_{x}+1, \ldots, 0 , \ldots,  N_{x}-1, N_{x}]) \in \mathbb{C}^{(2N_{x}+1) \times (2N_{x}+1)}
\end{equation}
and the convolution operation defined in Eq.~\eqref{convolution-definition} can be constructed as a matrix multiplication, where the electric field is converted into a Toeplitz matrix
\begin{equation}\label{definition-E-convolve}
    \mathbf{E}_{\mathrm{conv}}(t) = \left[\begin{array}{cccccc}
        E_{0}(t)  & E_{-1}(t) & \ldots & E_{-N_{x}}(t)&  0 & \ldots\\
        E_{1}(t) & E_{0}(t) & E_{-1}(t) & \ldots & E_{-N_{x}}(t) &  0  \\
        & \ddots & \ddots & \ddots & \ddots &\\
        E_{N_{x}}(t) & E_{N_{x}-1}(t) & \ldots & E_{0}(t) & \ldots & E_{-N_{x}}(t)  \\
        & \ddots & \ddots & \ddots & \ddots &\\
        0 & E_{N_{x}}(t)  & E_{N_{x}-1}(t) & \ldots & E_{0}(t)&  E_{-1}(t) \\
        \ldots & 0 & E_{N_{x}}(t) & E_{N_{x}-1}(t) & \ldots &  E_{0}(t) \\
    \end{array}\right] \in \mathbb{C}^{(2N_{x}+1) \times (2N_{x}+1)}.
\end{equation}
Since the electric field $E \in \mathcal{W}$ is a real-valued function, then  $(\mathbb{C}^{2N_{x}+1},\mathbf{E}_{\mathrm{conv}}(t))= (\mathbb{C}^{2N_{x}+1}, \mathbf{E}_{\mathrm{conv}}(t)^{*})$, defined in Eq.~\eqref{definition-E-convolve}, is a complex symmetric matrix. Similarly, the derivative diagonal matrix $(\mathbb{C}^{2N_{x}+1}, \mathbf{D}) = (\mathbb{C}^{2N_{x}+1},-\mathbf{D}^{*})$, defined in Eq.~\eqref{definiton-D-derivative}, is a complex anti-symmetric matrix (also known as skew-Hermitian). 

The Fourier spectral discretization of the SW Poisson equation~\eqref{f-poisson-pde} is given by 
\begin{equation*}
     E_{k}(t) = \frac{\ell}{2\pi i k} \sum_{s} q^{s} \alpha^{s} \sum_{n=0}^{N_{v}-1} \mathcal{I}_{n} C_{n, k}^{s}(t),
\end{equation*}
and the SW square-root Poisson equation~\eqref{sqrt-f-poisson} reads as 
\begin{equation*}
    E_{k}(t) = \frac{\ell}{2\pi i k} \sum_{s} q^{s} \alpha^{s} \sum_{n=0}^{N_{v}-1} [\mathbf{C}_{n}^{s}(t) \ast \mathbf{C}_{n}^{s}(t)][k],
\end{equation*}
with $E_{0}(t) = 0$ from the uniqueness condition, i.e. $\int_{0}^{\ell} E(x, t) \mathrm{d} x= 0$. 
\begin{theorem}
    A Fourier spatial spectral expansion in Eq.~\eqref{vlasov-vector-form-pde} preserves the anti-symmetric structure of the advection operator~$(\mathcal{W}^{N_{v}}, A^{s}_{v}(x, t))$ in Eq.~\eqref{vlasov-vector-form-pde}, and thus preserves the anti-symmetric structure of the advection operator~$(\mathcal{V}, A^{s}(x, v, t))$ in Eq.~\eqref{vlasov-continuum-anti-symmetric-form}.
\end{theorem}
\begin{proof}
The semi-discrete Vlasov equation~\eqref{f-vlasov-ode-fourier} can be written in vector form as 
\begin{equation*}
    \frac{\mathrm{d}}{\mathrm{d}t} \mathbf{\Psi}^{s}(t) = \mathbf{A}^{s}(t) \mathbf{\Psi}^{s}(t)
\end{equation*}
where 
\begin{equation*}
    \mathbf{\Psi}^{s}(t) = \begin{bmatrix}
        \mathbf{C}^{s}_{0}(t)\\
        \mathbf{C}^{s}_{1}(t)\\
        \vdots \\
        \mathbf{C}^{s}_{N_{v}-1}(t)
    \end{bmatrix} \in  \mathbb{C}^{N_{v}(2N_{x}+1)}
\end{equation*}
and
\begin{equation*}
    \mathbf{A}^{s}(t) = \begin{bmatrix}
    -u^{s} \mathbf{D}&  -\sqrt{\frac{1}{2}}\mathbf{G}_{+}^{s}(t) & 0 &  \ldots & 0 \\
    \sqrt{\frac{1}{2}} \mathbf{G}_{-}^{s}(t) & -u^{s} \mathbf{D} & -\sqrt{\frac{2}{2}} \mathbf{G}_{+}^{s}(t) &  \ldots & 0\\
    & &  & & \\
    & \ddots & \ddots &  \ddots&   \\
    & &  & & \\
    0 & 0 &  \sqrt{\frac{N_v-2}{2}} \mathbf{G}_{-}^{s}(t)& -u^{s} \mathbf{D} & -\sqrt{\frac{N_v-1}{2}} \mathbf{G}_{+}^{s}(t)   \\
    0 & 0 & 0 & \sqrt{\frac{N_v-1}{2}} \mathbf{G}_{-}^{s}(t)& -u^{s} \mathbf{D}
 \end{bmatrix} \in \mathbb{C}^{N_{v} (2N_{x}+1)\times N_{v}(2N_{x}+1)},
\end{equation*}
such that
\begin{equation*}
    \mathbf{G}_{+}^{s}(t) = \frac{q^{s}}{m^s \alpha^s} \mathbf{E}_{\mathrm{conv}}(t) + \alpha^{s} \mathbf{D} \qquad \mathrm{and} \qquad  \mathbf{G}_{-}^{s}(t) = \frac{q^{s}}{m^s \alpha^s} \mathbf{E}_{\mathrm{conv}}(t) - \alpha^{s}\mathbf{D}.
\end{equation*}
Since $(\mathbb{C}^{2N_{x}+1}, \mathbf{D}) = (\mathbb{C}^{2N_{x}+1}, -\mathbf{D}^{*})$ and $(\mathbb{C}^{2N_{x}+1}, \mathbf{E}_{\mathrm{conv}}(t)) = (\mathbb{C}^{2N_{x}+1}, \mathbf{E}_{\mathrm{conv}}(t)^{*})$, the discretized advection operator is anti-symmetric where $(\mathbb{C}^{N_{v}(2N_{x}+1)}, \mathbf{A}^{s}(t)) = (\mathbb{C}^{N_{v}(2N_{x}+1)}, -\mathbf{A}^{s}(t)^{*})$. 
\end{proof}

\section{Relating the SW and SW Square-Root Expansion Coefficients} \label{sec:Appendix-B}
We derive an analytic relation between the Hermite coefficients of the SW formulation~$C_{\mathrm{sw}, n}^{s}(x, t)$ and the SW square-root formulation~$C_{\mathrm{swsr}, n}^{s}(x, t)$, whereby
\begin{alignat*}{3}
    f^{s, N_{v}}(x, v, t) &=\sum_{n=0}^{N_{v}-1} C_{\mathrm{sw}, n}^{s}(x, t) \psi_{n}(\xi^{s}) \qquad &&\text{SW expansion, see Eq.}~\eqref{expansion-seperation-of-variables},\\
    \sqrt{f^{s, N_{v}}(x, v, t)} &= \sum_{n=0}^{N_{v}-1} C_{\mathrm{swsr}, n}^{s}(x, t) \psi_{n}(\xi^{s}) \qquad &&\text{SW square-root expansion, see Eq.}~\eqref{spectral-approximation-sqrt},
\end{alignat*}
where $\psi_{n}(\xi^{s})$, defined in Eq.~\eqref{hermite-basis-function-definition}, is the SW Hermite basis function of degree $n \in \mathbb{N}_{0}$. Thus, 
\begin{equation*}
    \sum_{n=0}^{N_{v}-1} C_{\mathrm{sw}, n}^{s}(x, t) \psi_{n}(\xi^{s}) = \left(\sum_{n=0}^{N_{v}-1} C_{\mathrm{swsr}, n}^{s}(x, t) \psi_{n}(\xi^{s}) \right)^2.
\end{equation*}
Multiplying the equation above by $\psi_{m}(\xi^{s})$ with $m \in \mathbb{N}_{0}$ and integrating with respect to the scaled velocity coordinate $\xi^{s}$, results in 
\begin{align}
    \sum_{n=0}^{N_{v}-1} C_{\mathrm{sw}, n}^{s}(x, t)&\underbrace{\int_{\mathbb{R}}\psi_{n}(\xi^{s}) \psi_{m}(\xi^{s}) \mathrm{d} \xi^{s}}_{=\delta_{n, m} \text{ from Eq.}~\eqref{orthogonal-property}} = \int_{\mathbb{R}} \left(\sum_{n=0}^{N_{v}-1} C_{\mathrm{swsr}, n}^{s}(x, t) \psi_{n}(\xi^{s}) \right)^2 \psi_{m}(\xi^s) \mathrm{d} \xi^{s}, \nonumber\\
    C_{\mathrm{sw}, m}^{s}(x, t) &= \sum_{n=0}^{N_{v}-1} \sum_{k=0}^{N_{v}-1} C_{\mathrm{swsr}, n}^{s}(x, t) C_{\mathrm{swsr}, k}^{s}(x, t) \underbrace{\int_{\mathbb{R}}  \psi_{n}(\xi^{s}) \psi_{k}(\xi^{s}) \psi_{m}(\xi^s) \mathrm{d} \xi^{s}}_{\coloneqq \mathcal{M}_{n, k, m}}. \label{triple-product-integral}
\end{align}
We employ the following three identities of the ``\textit{physicist}'' Hermite polynomials $\mathcal{H}_{n}(\xi^{s})$, defined in Eq.~\eqref{hermite-polynomials}, from~\cite[\S V]{magnus_1966_formulas} to derive a closed-form algebraic expression of the integral $\mathcal{M}_{n, k, m} \in \mathbb{R}$ from Eq.~\eqref{triple-product-integral}:
\begin{alignat}{3}
    &\mathrm{[linearization]} &&\mathcal{H}_{m}(\xi^{s}) \mathcal{H}_{n}(\xi^{s}) =\sum_{r=0}^{\mathrm{min}(m, n)} \mathcal{R}_{m, n, r} \mathcal{H}_{m+n-2r}(\xi^{s}), \label{linearization-identity}\\
    &\mathrm{[scaling]} \hspace{3.5 cm} &&\mathcal{H}_{n}(\lambda \xi^{s}) = \sum_{p=0}^{\lfloor \frac{n}{2} \rfloor} \mathcal{P}_{n, p}(\lambda) \mathcal{H}_{n-2p}(\xi^{s}), \qquad \text{with} \qquad \lambda \in \mathbb{R} \backslash \{0, 1\},\label{scaling-identity}\\
    &\mathrm{[orthogonality]} \qquad &&\int_{\mathbb{R}} \mathcal{H}_{n}(\xi^{s}) \mathcal{H}_{m}(\xi^{s}) \exp(-(\xi^{s})^2)\mathrm{d} \xi^{s} = \frac{\delta_{n, m}}{\gamma^{2}_{n}}, \label{orthogonality-identity-hermite-poly}
\end{alignat}
where 
\begin{equation*}
    \mathcal{R}_{m, n, r} \coloneqq 2^{r} r! \binom{m}{r} \binom{n}{r},\qquad \mathcal{P}_{n, p}(\lambda)\coloneqq \lambda^{n-2p}(\lambda^{2} -1)^{p} \binom{n}{2p} \frac{(2p)!}{p!},\qquad \text{and} \qquad \gamma_{n} \coloneqq (\sqrt{\pi} 2^{n} n!)^{-\frac{1}{2}}.
\end{equation*}
By the definition of the SW Hermite basis functions $\psi_{n}(\xi^{s})$ in Eq.~\eqref{hermite-basis-function-definition} and the three identities above~\eqref{linearization-identity}--\eqref{orthogonality-identity-hermite-poly}, the integral $\mathcal{M}_{n, k, m} \in \mathbb{R}$ from Eq.~\eqref{triple-product-integral} simplifies to 
\begin{alignat*}{3}
    \mathcal{M}_{n, k, m} &= \int_{\mathbb{R}}  \psi_{n}(\xi^{s}) \psi_{k}(\xi^{s}) \psi_{m}(\xi^s) \mathrm{d} \xi^{s} = \gamma_{n} \gamma_{k} \gamma_{m} \int_{\mathbb{R}} \underbrace{\mathcal{H}_{n}(\xi^{s}) \mathcal{H}_{k}(\xi^{s})}_{\text{linearization}} \mathcal{H}_{m}(\xi^{s}) \exp\left(-\frac{3 (\xi^{s})^2}{2}\right) \mathrm{d} \xi^{s} \\
    &= \gamma_{n} \gamma_{k} \gamma_{m} \sum_{r=0}^{\mathrm{min}(k, n)} \mathcal{R}_{k, n, r} \int_{\mathbb{R}} \mathcal{H}_{k+n - 2r}(\xi^{s})\mathcal{H}_{m}(\xi^{s}) \exp\left(-\frac{3 (\xi^{s})^2}{2}\right) \mathrm{d} \xi^{s} \Hquad &&\text{from }~\eqref{linearization-identity}\\
    &= \sqrt{\frac{2}{3}} \gamma_{n} \gamma_{k} \gamma_{m} \sum_{r=0}^{\mathrm{min}(k, n)} \mathcal{R}_{k, n, r} \int_{\mathbb{R}} \underbrace{\mathcal{H}_{k+n - 2r}\left(\sqrt{\frac{2}{3}} y\right)}_{\text{scaling}}\underbrace{\mathcal{H}_{m}\left(\sqrt{\frac{2}{3}} y\right)}_{\text{scaling}}\exp\left(-y^2\right) \mathrm{d} y \Hquad &&\text{s.t. }y\coloneqq\sqrt{\frac{3}{2}}\xi^{s}\\
    &=\sqrt{\frac{2}{3}} \gamma_{n} \gamma_{k} \gamma_{m} \sum_{r=0}^{\mathrm{min}(k, n)} \mathcal{R}_{k, n, r}\sum_{q=0}^{\lfloor \frac{k+n - 2r}{2}\rfloor} \mathcal{P}_{k+n - 2r, q}\left(\sqrt{\frac{2}{3}}\right) \sum_{l=0}^{\lfloor \frac{m}{2}\rfloor} \mathcal{P}_{m, l}\left(\sqrt{\frac{2}{3}}\right) \Hquad &&\text{from }~\eqref{scaling-identity}\\
    &\qquad\qquad\cdot\underbrace{\int_{\mathbb{R}} \mathcal{H}_{k+n - 2r - 2q}\left(y\right) \mathcal{H}_{m - 2l}\left(y\right) \exp\left(-y^2\right) \mathrm{d} y}_{\text{orthogonality}}\\
    &= \sqrt{\frac{2}{3}} \gamma_{n} \gamma_{k} \gamma_{m} \sum_{r=0}^{\mathrm{min}(k, n)} \mathcal{R}_{k, n, r}\sum_{q=0}^{\lfloor \frac{k+n - 2r}{2}\rfloor} \mathcal{P}_{k+n - 2r, q}\left(\sqrt{\frac{2}{3}}\right) \sum_{l=0}^{\lfloor \frac{m}{2}\rfloor} \mathcal{P}_{m, l}\left(\sqrt{\frac{2}{3}}\right) \frac{\delta_{k+n-2r-2q, m-2l}}{\gamma^{2}_{m-2l}} \Hquad &&\text{from }~\eqref{orthogonality-identity-hermite-poly}.
\end{alignat*}
Therefore, 
\begin{align*}
    C_{\mathrm{sw}, m}^{s}(x, t) &= \sum_{n=0}^{N_{v}-1} \sum_{k=0}^{N_{v}-1} C_{\mathrm{swsr}, n}^{s}(x, t) C_{\mathrm{swsr}, k}^{s}(x, t) \sqrt{\frac{2}{3}} \gamma_{n} \gamma_{k} \gamma_{m} \sum_{r=0}^{\mathrm{min}(k, n)} \mathcal{R}_{k, n, r}\sum_{q=0}^{\lfloor \frac{k+n - 2r}{2}\rfloor} \mathcal{P}_{k+n - 2r, q}\left(\sqrt{\frac{2}{3}}\right) \\
    &\qquad\qquad\qquad\cdot\sum_{l=0}^{\lfloor \frac{m}{2}\rfloor} \mathcal{P}_{m, l}\left(\sqrt{\frac{2}{3}}\right) \frac{\delta_{k+n-2r-2q, m-2l}}{\gamma^{2}_{m-2l}}.
\end{align*}

\section*{Code Availability}
The public repository~\url{https://github.com/opaliss/Antisymmetric-Vlasov-Poisson} contains a collection of Jupyter notebooks in Python~3.9 with the code used in this study.

\section*{Acknowledgement} \label{sec:acknowledgement}
O.K. and G.L.D. are grateful for the useful and informative discussions with Cecilia Pagliantini, Gianmarco Manzini, and Daniel Livescu.
O.I. was partially supported by the Los Alamos National Laboratory Vela Fellowship. O.I. and B.K. were partially supported by the National Science Foundation under Award 2028125 for ``SWQU: Composable Next Generation Software Framework for Space Weather Data Assimilation and Uncertainty Quantification''. 
O.K. and G.L.D. were supported by the Laboratory Directed Research and Development Program of Los Alamos National Laboratory under project number 20220104DR. Los Alamos National Laboratory is operated by Triad National Security, LLC, for the National Nuclear Security Administration of the U.S. Department of Energy (Contract No. 89233218CNA000001).
F.D.H. was supported by the U.S. Department of Energy, Office of Science, Office of Fusion Energy Sciences, Theory Program, under Award No. DE-FG02-95ER54309. 
\bibliographystyle{abbrv}
\bibliography{references}

\end{document}